\documentclass[11pt]{amsart}
\usepackage{ytableau}
\input xy
\xyoption{all}

\theoremstyle{plain}
\newtheorem{thm}{Theorem}[section]
\newtheorem{theorem}[thm]{Theorem}

\newtheorem{lemma}[thm]{Lemma}
\newtheorem{corollary}[thm]{Corollary}
\newtheorem{proposition}[thm]{Proposition}
\theoremstyle{definition}
\newtheorem{remark}[thm]{Remark}

\newtheorem{definition}[thm]{Definition}

\newtheorem{example}[thm]{Example}

\newtheorem{setup}[thm]{Set-up}
\numberwithin{equation}{section}

\newcommand{\w}{\widetilde}

\newcommand{\p}{\partial}

\newcommand{\sB}{{\mathcal B}}
\newcommand{\sC}{{\mathcal C}}
\newcommand{\sD}{{\mathcal D}}

\newcommand{\sH}{{\mathcal H}}

\newcommand{\sK}{{\mathcal K}}
\newcommand{\sL}{{\mathcal L}}

\newcommand{\sO}{{\mathcal O}}
\newcommand{\sP}{{\mathcal P}}

\newcommand{\sS}{{\mathcal S}}

\newcommand{\sU}{{\mathcal U}}
\newcommand{\sV}{{\mathcal V}}
\newcommand{\sW}{{\mathcal W}}

\newcommand{\C}{{\mathbb C}}

\newcommand{\F}{{\mathbb F}}

\newcommand{\I}{{\mathbb I}}

\newcommand{\N}{{\mathbb N}}
\newcommand{\BP}{{\mathbb P}}

\newcommand{\R}{{\mathbb R}}

\newcommand{\Z}{{\mathbb Z}}

\newcommand{\gr}{{\mathbf g}{\mathbf r}}
\newcommand{\bL}{{\rm Lifts}}

\newcommand{\End}{{\rm End}}

\newcommand{\fg}{{\mathfrak g}}

\newcommand{\fe}{{\mathfrak e}}
\newcommand{\fl}{{\mathfrak l}}

\newcommand{\fgl}{{\mathfrak g}{\mathfrak l}}
\newcommand{\ff}{{\mathfrak f}}

\newcommand{\fh}{{\mathfrak h}}

\newcommand{\fv}{{\mathfrak v}}
\newcommand{\fw}{{\mathfrak w}}

\newcommand\Aut{\rm Aut}

\def\Hom{\mathop{\rm Hom}\nolimits}

\def\min{\mathop{\rm min}\nolimits}

\title[Convergence of formal equivalence]{Generalized Tanaka prolongation and convergence of formal equivalence between embeddings}

\author{Jaehyun Hong and Jun-Muk Hwang}

\thanks{This work was supported by the Institute for Basic Science (IBS-R032-D1).}

\begin{document}
\begin{abstract}
The works of Commichau--Grauert and Hirschowitz  showed  that  a formal equivalence between embeddings of a compact complex manifold is convergent, if the embeddings have sufficiently positive normal bundles in a suitable sense.  We show that the convergence still holds under the weaker assumption of semi-positive normal bundles  if some geometric conditions are satisfied. Our result can be applied to  many examples of general minimal rational curves, including  general lines on a smooth hypersurface of degree less than $n$ in the $(n+1)$-dimensional projective space.
As a key ingredient of our arguments,  we formulate and prove a generalized version of Tanaka's prolongation procedure for geometric structures subordinate to vector distributions, a result  of independent interest.  When applied to  the universal family of the deformations of the compact submanifolds satisfying our geometric conditions, the generalized Tanaka prolongation gives a natural absolute parallelism on a suitable fiber space. A formal equivalence of embeddings must preserve these absolute parallelisms,  which implies its convergence.
\end{abstract}

\maketitle


\tableofcontents

\section{Introduction}\label{s.introduction}
Our main concern is the following notion for a compact complex submanifold in a complex manifold.

\begin{definition}\label{d.FP} For a compact complex submanifold $A$ in a complex manifold $X$, denote by $(A/X)_{\infty}$ the formal neighborhood of $A$ in $X$.
We say that $A \subset X$ satisfies {\em the formal principle with convergence},  if
 any formal isomorphism $\phi: (A/X)_{\infty} \to (\widetilde{A}/\widetilde{X})_{\infty}$ to the formal neighborhood of a compact submanifold $\widetilde{A}$ in another complex manifold $\widetilde{X}$ is convergent, namely,
 there exists a biholomorphic map $\Phi: U \to \w{U}$ between Euclidean neighborhoods $U \subset X$ of $A$ and $\w{U} \subset \w{X}$ of $\w{A}$
  such that $\phi= \Phi|_{(A/X)_{\infty}},$ giving the following commutative diagram.
  $$ \begin{array}{ccccccc}   X & \supset & U & \stackrel{\Phi}{\longrightarrow} & \widetilde{U} & \subset  & X \\ & & \cup & & \cup & & \\ & & (A/X)_{\infty} & \stackrel{\phi}{\longrightarrow} & (\widetilde{A}/\widetilde{X})_{\infty}. & & \end{array}$$

  A weaker notion is the formal principle (see Section VII.4 of \cite{GPR}): a  submanifold  $A \subset X$
 satisfies {\em the formal principle} if the existence of a formal isomorphism $\phi: (A/X)_{\infty} \to (\widetilde{A}/\widetilde{X})_{\infty}$  implies the existence of  a biholomorphic map
$\Phi: U \to \w{U}$ between some Euclidean neighborhoods,   {\em not necessarily satisfying } $\phi= \Phi|_{(A/X)_{\infty}}.$
\end{definition}

The formal principle with convergence is a much stronger property than the formal principle (in the survey \cite{Ks86}, the formal principle with convergence was called  the formal principle in the strong sense). For example, when $A$ is just a point in a complex manifold $X$, it satisfies the formal principle, but does not satisfy the formal principle with convergence (also see Example \ref{ex.product} below).
Whereas it is difficult to find  an example of $A \subset X$ that  {\em does not satisfy} the formal principle (the first such example was discovered by Arnold in \cite{Ar}), it is not easy to give an example that {\em does satisfy} the formal principle with convergence. The only known examples of the latter  until now have been those covered by the works of Commichau-Grauert \cite{CG} and Hirschowitz \cite{Hir}. These examples have sufficiently positive normal bundles in a suitable sense: Commichau-Grauert's result requires, among others, the ampleness of the normal bundle $N_{A/X}$, while Hirschowitz's result assumes that the submanifold $A \subset X$ has sufficiently many deformations that have nonempty intersection with $A$ (see Theorem 1.11 in Chapter VIII of \cite{GPR} for a summary of Hirschowitz's result).

In \cite{Hw19}, the second-named author studied a more general situation than Hirschowitz's  and obtained the following result, Theorem 1.5 in \cite{Hw19}.

 \begin{theorem}\label{t.Hw}
 Let $X$ be a complex manifold and let $\sK \subset {\rm Douady}(X)$ be a subset  of the Douady space of $X$ with the associated universal family morphisms
 $$ \sK \stackrel{\rho}{\longleftarrow} \sU \stackrel{\mu}{\longrightarrow} X$$ such that
 \begin{itemize}
\item[(i)] $\sK$ is a connected open subset in the smooth loci of ${\rm Douady}(X)$; \item[(ii)] $\rho$ is a smooth proper morphism with connected fibers;
 \item[(iii)] $\mu$ is submersive at every point of $\sU$; and
 \item[(iv)] for the submanifold  $A \subset X$  corresponding to any point in $\sK$,
 the normal bundle $N_{A/X}$ satisfies for any $ x \neq x' \in A$,
$$H^0(A, N_{A/X} \otimes {\bf m}_x) \neq H^0(A, N_{A/X} \otimes {\bf m}_{x'}) $$ as subspaces of $H^0(A, N_{A/X})$. \end{itemize} Then there exists a nowhere-dense subset $\sS \subset \sK$ such that the submanifold  $A \subset X$ corresponding to any  point of $\sK \setminus \sS$ satisfies the  formal principle. \end{theorem}

Recall here that   for each complex space $X$, we have its Douady space denoted by ${\rm Douady}(X)$, a complex space parametrizing compact complex subspaces of $X$ (the complex-analytic version of the Hilbert scheme  in algebraic geometry). There is the associated
 universal family morphisms $ {\rm Douady}(X) \leftarrow {\rm Univ}(X) \rightarrow X,$ the restriction of which to $\sK$ gives the morphisms $\rho$ and $\mu$ in Theorem \ref{t.Hw}.   We refer the reader to the introductory survey in Section VIII.1 of \cite{GPR} for more details on Douady spaces.

When compared with \cite{Hir}, the advantage of Theorem \ref{t.Hw} is that the normal bundle could be semi-positive, not necessarily positive. In fact,  the most interesting application of Theorem \ref{t.Hw}  is the following case.

\begin{definition}\label{d.unbendable}
A smooth rational curve $\BP^1 \cong A \subset X$ in a complex manifold is {\em unbendable} if its normal bundle is isomorphic to $\sO(1)^{\oplus r} \oplus \sO^{\oplus (\dim X - r-1)}$ for some nonnegative integer $r < \dim X$. \end{definition}

Deformations of any unbendable rational curve $A \subset X$ give rise to a family $\sK \leftarrow \sU \rightarrow X$ satisfying all the conditions of Theorem \ref{t.Hw}.
There are many interesting examples of unbendable rational curves arising from algebraic geometry. Especially, all general minimal rational curves are unbendable (see  \cite{Hw01} or \cite{HwM} for details and many explicit  examples).

The key question we want to study is: under what additional conditions the conclusion of Theorem \ref{t.Hw} can be strengthened to the formal principle with convergence? Of course, Hirschowitz's result provides such additional conditions, but they are too strong for the applications to unbendable rational curves. In fact, unbendable rational curves satisfy Hirschowitz's condition only when $r= \dim X -1$, namely, when the normal bundle is  ample. This is rather restrictive. For example, among  minimal rational curves, only lines in projective space can satisfy $r= \dim X-1.$ Our goal is to find weaker conditions satisfied by a large class of minimal rational curves.

To discuss this question, we work in the following general setting.

\begin{setup}\label{set.sK}
Let $X$ be a complex manifold and let $\sK \subset {\rm Douady}(X)$ be a subset  of the Douady space of $X$ with the associated universal family morphisms
 $$ \sK \stackrel{\rho}{\longleftarrow} \sU \stackrel{\mu}{\longrightarrow} X$$ such that
 \begin{itemize}
\item[(i)] $\sK$ is a connected open subset in the smooth loci of ${\rm Douady}(X)$; \item[(ii)] $\rho$ is a smooth proper morphism with connected fibers; and
 \item[(iii)] there is a nonempty Zariski-open subset $\sU' \subset \sU$ such that $\mu$ is submersive at every point of $\sU'$. \end{itemize}
\end{setup}

Under what additional conditions can we have the formal principle with convergence for a general member of $\sK$?  We have mentioned that when $A$  is  just a point  in any complex manifold $X$, it  cannot satisfy the formal principle with convergence. An  obvious generalization is the following.

\begin{example}\label{ex.product}
Let $A \subset Y$ be a compact complex manifold in a complex manifold $Y$ and let $X = Y \times \C$ be its product with the complex line. Then $A \subset X$ does not satisfy the formal principle with convergence. \end{example}

Thus we need a condition to avoid cases like Example \ref{ex.product}. We propose the following.

\begin{definition}\label{d.bracketgen}
In Set-up \ref{set.sK}, we say that $\sK$ is  a {\em  bracket-generating family  } if the following condition is satisfied at a general point $y \in \sU$.
\begin{itemize} \item Suppose there is a neighborhood  $ \rho(y) \in W \subset \sK$ and a neighborhood $\mu(y) \in O \subset X$ with a submersion $\zeta: O \to R$ to a complex manifold $R$ such that for  any $z \in W$, the intersection $\mu(\rho^{-1}(z)) \cap O$ is contained in a fiber of $\zeta$. Then $R$ must be a point. \end{itemize} \end{definition}

If $\sK$ is a bracket-generating family of submanifolds in $X$, then there cannot be a nontrivial (meromorphic) foliation on $X$ such that general members are contained in the leaves of the foliation. An equivalent condition is that the distribution spanned by the tangent spaces of members of $\sK$ in a suitable sense is a bracket-generating distribution in the sense of Definition \ref{d.distribution} (2) (see Definition \ref{d.sKdistribution} and Proposition \ref{p.sKdistribution}).
For example, deformations of $A \subset X$ in Example \ref{ex.product} cannot belong to a bracket-generating family. Thus it is reasonable to impose the bracket-generating condition on $\sK$ in the setting of Theorem \ref{t.Hw} to be able to have the formal principle with convergence. We also need the following notion.

\begin{definition}\label{d.tau}
In Set-up \ref{set.sK}, assume that both  $m := \dim \sU - \dim \sK$ and  $r := \dim \sU - \dim X$ are positive. For each $y \in \sU'$ where $\mu$ is submersive, set $x := \mu(y)$ (resp. $z := \rho(y)$) and  let $\mu^{\sharp} (y) \in {\rm Gr}(m ; T_x X)$ (resp.   $\rho^{\sharp} (y) \in {\rm Gr}(r; T_z \sK)$) be the point in the Grassmannian of $m$-dimensional (resp. $r$-dimensional)  subspace in $T_x X$ (resp. $T_z \sK$) corresponding to $${\rm d} \mu ({\rm Ker}({\rm d}_y \rho)) \ \subset \ T_x X \  \mbox{ (resp. } {\rm d} \rho ({\rm Ker}({\rm d}_y \mu)) \ \subset \ T_z \sK \mbox{).}$$ This defines a holomorphic map $\mu^{\sharp}: \sU' \to {\rm Gr}(m; TX)$ (resp. $\rho^{\sharp}: \sU' \to {\rm Gr}(r; T\sK)$) to the Grassmannian bundle:
$$ \begin{array}{ccccc}
{\rm Gr}(r ; T\sK) & \stackrel{\rho^{\sharp}}{\longleftarrow} & \sU' & \stackrel{\mu^{\sharp}}{\longrightarrow} & {\rm Gr}(m; TX) \\
\downarrow & & \cap & & \downarrow \\ \sK & \stackrel{ \rho}{\longleftarrow} & \sU & \stackrel{ \mu }{\longrightarrow} & X. \end{array} $$
\end{definition}

 We mention (see Proposition \ref{p.rhosharp}) that the condition (iv) of Theorem \ref{t.Hw} implies that $\rho^{\sharp}$ is generically immersive, namely, immersive at a general point of $\sU'$. If $\sK$ is a family of unbendable rational curves with $r >0$ (equivalently, the normal bundle of the rational curves are nontrivial), then both $\mu^{\sharp}$ and $\rho^{\sharp}$ are generically immersive (see Lemma \ref{l.unbendable}).
Our main result is the following.

\begin{theorem}\label{t.FPC}
In Set-up \ref{set.sK}, assume that $\sK$ is a bracket-generating family and
 both $\rho^{\sharp}$ and $\mu^{\sharp}$ are generically immersive.
    Then there is a Zariski-open subset $\sK_{\flat} \subset \sK$ such that the submanifold $A \subset X$ corresponding to any point of $\sK_{\flat}$ satisfies the formal principle with convergence. \end{theorem}

There are lots of examples of $\sK$ satisfying the conditions in Theorem \ref{t.FPC}.
We list some of them in Section \ref{s.examples}.
    Here, let us  mention just the following case (see Example \ref{ex.hypersurface}).

\begin{corollary}\label{c.hypersurface}
A general line on  a smooth hypersurface of degree less than $n$ in $\BP^{n+1}, n \geq 4,$  satisfies the formal principle with convergence, while  a general line on a smooth hypersurface of degree $n$  does not satisfy the formal principle with convergence.
\end{corollary}

    Our strategy to prove  Theorem \ref{t.FPC} is to  deduce it from  the following result.

    \begin{theorem}\label{t.AP}
    In Set-up \ref{set.sK}, assume that $\sK$ is bracket-generating and both $\rho^{\sharp}$ and $\mu^{\sharp}$ are generically immersive. Then there exists a nonempty Zariski-open subset $\sU_{\flat} \subset \sU$ such that each $x \in \sU_{\flat}$ admits a neighborhood $O_x \subset \sU_{\flat}$ canonically equipped with
\begin{itemize} \item a  complex manifold $\overline{O}_x$ with a submersion $\overline{O}_x \to O_x$; and \item an absolute parallelism (namely, a frame for the tangent bundle) on $\overline{O}_x$. \end{itemize}
\end{theorem}

 Here, the absolute parallelism is canonical up to a universal choice of some data on the moduli space of Lie algebras of certain type (see Definition \ref{d.choice} for a precise formulation). Also the neighborhood $O_x$ could be enlarged to an \'etale neighborhood (namely, a finite unramified covering of a Zariski-open subset) in $\sU$ (see Remark \ref{r.nB}).

     Why does Theorem \ref{t.AP} imply Theorem \ref{t.FPC}? We can explain it heuristically as follows. Let $\varphi: (o/\C^n)_{\infty} \to (\widetilde{o}/\widetilde{\C}^n)_{\infty}$ be a formal isomorphism  between the formal neighborhoods of points on $n$-dimensional complex manifolds. It is obvious that if $\varphi$ sends a holomorphic coordinate system in a neighborhood of $o \in \C^n$ to a holomorphic coordinate system  in a neighborhood of $\widetilde{o} \in \widetilde{\C}^n,$ then $\varphi$ converges.   One can generalize  this to a less trivial fact (see Theorem \ref{t.KN}) that if $\varphi$ sends a holomorphic absolute parallelism in a neighborhood of $o \in \C^n$ to a holomorphic absolute parallelism in a neighborhood of $\widetilde{o} \in \widetilde{\C}^n$, then $\varphi$ converges. The canonical nature of the absolute parallelism in Theorem \ref{t.AP} guarantees that a formal isomorphism between formal neighborhoods of submanifolds in the setting of  Theorem \ref{t.FPC} can be lifted to a formal isomorphism preserving  canonical absolute parallelisms  from Theorem \ref{t.AP}, hence must converge.

\medskip
It seems to us that  Theorem \ref{t.AP} has significance beyond its application to  Theorem \ref{t.FPC}. Among the two conditions in Theorem \ref{t.AP}, the bracket-generating condition is certainly necessary, because it is necessary for Theorem \ref{t.FPC} as discussed above. Is  the  assumption  that the maps $\rho^{\sharp}$ and $\mu^{\sharp}$ are generically immersive   necessary in  Theorem \ref{t.AP}?   We leave the question for future investigation.

\medskip
     The proof of Theorem \ref{t.AP} employs  Noboru Tanaka's prolongation procedure developed in \cite{Ta70}. This procedure constructs a canonical absolute parallelism on a natural fiber space over a manifold equipped with a  differential system satisfying certain algebraic properties.
    Tanaka's original work
    \cite{Ta70}, as well as its simplified version in \cite{AD}, requires that the differential system has constant symbols, a condition too strong for our purpose. For this reason, we need a  generalization of  Tanaka's prolongation  to systems with non-constant symbols. A large part of this paper is devoted to prove this generalization,  Theorem \ref{t.Tanaka1}.  It seems to have been  a folklore in the subject that it is possible to generalize Tanaka's prolongation to systems with nonconstant symbols.  But a rigorous presentation  of such a generalization has never been worked out explicitly to our knowledge.    Part of our goal is to carry out this generalized Tanaka prolongation in full details.

   One difficulty in proving such a general version of Tanaka's prolongation lies in the enormous complexity of the inductive argument. As a matter of fact,  Tanaka's original  proof in \cite{Ta70} for systems with constant symbols are already extremely intricate.  Fortunately for us, Alexeevsky and David  in \cite{AD} have simplified Tanaka's proof considerably and it turns out that their proof is readily adaptable for systems with nonconstant symbols. One technical issue in this adaptation arises from the arguments in \cite{AD} involving  principal connections, which  do not make sense when symbols are not constant. To bypass this problem, we introduce   the notion of $\fl$-exponential actions (see Subsection \ref{ss.Bgroup}). It  enables us to use 1-parameter groups to produce local connections (in Subsection \ref{ss.connection}), which play the role of principal connections in our setting.

    Since it is expected that this general version of Tanaka prolongation could be useful  in many problems in differential geometry beside the proof of Theorem \ref{t.AP},   we  present it in Sections \ref{s.prelim} -- \ref{s.ProofInduction} in such a way that it can be read independently from the rest of the paper.
    Although our line of proofs essentially follows that of   \cite{AD}, we believe that our arguments are simpler and more streamlined   at a number of steps, even for systems with constant symbols.
    In particular, our proof can be read without prior knowledge of \cite{Ta70} or \cite{AD}. However, if the reader is not familiar with Tanaka theory, we recommend reading the introduction parts of \cite{Ta70} and \cite{AD}, and especially Zelenko's excellent overview of Tanaka prolongation in \cite{Ze}.

To apply the generalized  Tanaka prolongation   to the situation in Theorem \ref{t.AP}, we need  the theory of pseudo-product structures, another subject initiated by Noboru Tanaka in \cite{Ta85}. The generalized Tanaka prolongation for pseudo-product structures requires certain algebraic conditions. The main complex-geometric component of our work is verifying these algebraic conditions from the geometric conditions in Theorem \ref{t.AP}.

It is worthwhile to compare our proof of Theorem \ref{t.FPC}  with the proof  of Theorem \ref{t.Hw} in \cite{Hw19}.  Whereas the proof of Theorem \ref{t.FPC} uses (generalized) Tanaka prolongation,  the proof of Theorem \ref{t.Hw}  uses Cartan's prolongation procedure systemized by Morimoto \cite{Mo83}. Cartan prolongation yields an involutive system, a solution of which  gives the biholomorphic equivalence $\Phi$ in Definition \ref{d.FP} needed to prove the formal principle. But such a solution $\Phi$ is usually not unique and can not  prove the formal principle with convergence.  Tanaka prolongation, which requires stronger conditions than Cartan prolongation,   is a more refined version of Cartan's procedure, yielding a stronger output of a natural absolute parallelism.  We may summarize the methodological difference as follows.

    $$ \begin{array}{ccccc} \begin{array}[t]{c} \mbox{Cartan} \\ \mbox{prolongation} \end{array}  & \Rightarrow & \begin{array}[t]{c}    \mbox{involutive} \\ \mbox{system} \end{array} & \Rightarrow &
    \begin{array}[t]{c} \mbox{Formal} \\ \mbox{Principle} \end{array} \\
    & & & & \\
   \begin{array}[b]{c}    \mbox{Tanaka} \\ \mbox{prolongatoin} \end{array}  & \Rightarrow  & \begin{array}[b]{c}     \mbox{absolute} \\ \mbox{parallelism} \end{array} &  \Rightarrow &
    \begin{array}[b]{c} \mbox{Formal Principle} \\ \mbox{with Convergence}  \end{array} \end{array} $$

    \medskip
   The paper is organized as follows.  Sections \ref{s.prelim} -- \ref{s.ProofInduction} are devoted to the   generalization of Tanaka's prolongation and can be read independently from the rest of the paper. Section \ref{s.prelim} collects some preliminaries for the generalized Tanaka prolongation. It includes a review of  some basic results from \cite{AD} in Subsections \ref{ss.graded} -- \ref{ss.filtmfd} and the introduction of  new tools to handle systems with nonconstant symbols in Subsections \ref{ss.Bgroup} and \ref{ss.connection}. In Section \ref{s.Tanaka}, we formulate  the generalized Tanaka prolongation, starting with an explanation of the basic set-up   and giving  the full statement of the prolongation theorem, Theorem \ref{t.Tanaka1}, and some immediate corollaries.  The next two sections are devoted to the proof of Theorem \ref{t.Tanaka1}: Section \ref{s.n=1} is the proof of the first step of the prolongation  and Section \ref{s.ProofInduction} is the proof of the inductive argument.     In Section \ref{s.PP}, we introduce the concept of a pseudo-product structure and study its  Levi-nondegeneracy condition. In Section \ref{s.PPsK}, we  show that  a pseudo-product  structure with suitable conditions  gives rise to a geometric structure in the generalized Tanaka prolongation and apply this  to  the natural pseudo-product structure on the universal family in Set-up \ref{set.sK} to prove Theorem \ref{t.AP}. In Section \ref{s.converge}, we explain the idea of  deducing Theorem \ref{t.FPC} from Theorem \ref{t.AP}, from a more general perspective. Section \ref{s.examples} lists several examples of the families $\sK$ satisfying the conditions of Theorem \ref{t.FPC}.

Throughout the paper, we work in the holomorphic setting and all geometric objects are complex-analytic. But it is clear that the arguments in Sections \ref{s.prelim} -- \ref{s.ProofInduction}  work verbatim in  $C^{\infty}$ or real-analytic settings.

\section{Generalized Tanaka prolongation: Preliminaries}\label{s.prelim}

All vector spaces are over complex numbers and finite-dimensional unless stated otherwise.

\subsection{Lie groups associated with a graded vector space}\label{ss.graded}
  \begin{definition}\label{d.graded}
Fix a graded vector space
$$\fv = \fv^{-k} \oplus \fv^{-k +1} \oplus \cdots \oplus \fv^{\ell -1} \oplus \fv^{\ell},$$ with integers $k \geq 1$ and $\ell \geq -1$. We use the convention $$\fv^{-k-j}=0 \mbox{ and  }
\fv^{\ell + j} = 0   \mbox{ for any } j >0.$$ \begin{itemize}
\item[(i)] For any $n \geq -k$, define the truncated graded vector space
$$\fv^{<n+1} =  \fv^{-k} \oplus \fv^{-k +1} \oplus \cdots \oplus \fv^{n-1} \oplus \fv^{n} $$ and write $ \fv^- := \fv^{<0}.$
\item[(ii)]
When $\fw = \oplus_{i \in \Z} \fw^i$ is another graded vector space, we define for each integer $j,$
$$\Hom^j(\fv, \fw):=\{h \in \Hom(\fv, \fw) \mid h(\fv^i) \subset \fw^{i+j} \mbox{ for all } i   \}
$$ such that $\Hom(\fv, \fw) = \oplus_{j \in \Z} \Hom^j (\fv, \fw)$ is a graded vector space.
\item[(iii)]
For any integer $m \geq 1$,  define
 \begin{eqnarray*}
 \fgl^m(\fv) &: = & \{ h \in \fgl(\fv) \mid h(\fv^i) \subset \fv^{i+m} \mbox{ for all } i\}, \\
 \fgl_m(\fv) &:=& \oplus_{i \geq m} \ \fgl^i (\fv), \\
 {\rm GL}_m(\fv) & :=& \{ {\rm Id}_{\fv} + A \in {\rm GL}(\fv) \mid A \in \fgl_m(\fv)\}. \end{eqnarray*} \end{itemize} \end{definition}

 We skip the proof of the following two elementary lemmata.

 \begin{lemma}\label{l.inclusion}
For a graded vector space $\fv$, we have natural inclusions $$\fgl(\fv^{<n}) \subset \fgl(\fv^{<n+1}) \mbox{ and } {\rm GL}(\fv^{<n}) \subset {\rm GL}(\fv^{<n+1})$$ given by
 \begin{eqnarray*} \fgl(\fv^{<n}) &=& \{ h \in \fgl(\fv^{<n+1}) \mid h (\fv^{<n}) \subset \fv^{<n},  h(\fv^n) =0 \} \\ {\rm GL}(\fv^{<n}) & = & \{ h \in {\rm GL}(\fv^{<n+1}) \mid h(\fv^{<n}) \subset \fv^{<n},  h|_{\fv^n} ={\rm Id}_{\fv^n}\}. \end{eqnarray*} \end{lemma}

 \begin{lemma}\label{l.fglj}
 Let us use the notation in Definition \ref{d.graded} (iii).
\begin{itemize} \item[(i)] For each $m \geq 1$, the subspace $\fgl_m(\fv) \subset \fgl(\fv)$ is a Lie subalgebra and ${\rm GL}_m(\fv) \subset {\rm GL}(\fv)$ is the corresponding Lie subgroup.   \item[(ii)] For each $m \geq 2$, the subgroup ${\rm GL}_m(\fv)$ is a normal subgroup of ${\rm GL}_1(\fv)$ and any element $A$ of the  quotient group ${\rm GL}_1(\fv)/{\rm GL}_m(\fv)$ can be written as \begin{equation}\label{e.decompo} A = A^0 + A^1 + \cdots + A^{m-1} \end{equation}  with $A^j \in \fgl^j(\fv)$ and $A^0 = {\rm Id}_{\fv}$. \end{itemize}
\end{lemma}

 \begin{definition}\label{d.barG}
Let  $\fv = \oplus_{i=-k}^{\ell} \fv^i$ be a graded vector space. \begin{itemize} \item[(i)] For $n \geq 0$, define a Lie subalgebra
$\fh(\fv^{<n+1}) \subset \fgl_1(\fv^{<n+1})$ by
$$\fh(\fv^{<n+1}) :=  \fgl_{n+1}(\fv^{<n}) + \Hom(\fv^{<n}, \fv^n) = \fgl_{n+1}(\fv^{<n+1}) + \Hom(\oplus_{j=0}^{n-1} \fv^j, \fv^n).$$
Here, an element  $A \in \Hom(\fv^{<n}, \fv^n)$ (resp. $ A\in \Hom(\oplus_{j=0}^{n-1} \fv^j, \fv^n)$) is regarded as an element of $ \fgl(\fv^{<n+1})$ by setting $A (\fv^n) =0$ (resp. $A(\fv^-) = A(\fv^n) =0$).
\item[(ii)]
Define
$ H(\fv^{<n+1}):= {\rm Id}_{\fv^{<n+1}} + \fh(\fv^{<n+1}) \subset {\rm GL}(\fv^{<n+1}),$  a closed subgroup in ${\rm GL}(\fv^{<n+1})$ with  Lie algebra $\fh(\fv^{<n+1}).$ \end{itemize} Note that
$\fh(\fv^{<1}) = \fgl_1(\fv^{<1}) \mbox{ and } H(\fv^{<1}) = {\rm GL}_1(\fv^{<1})  .$ \end{definition}

The next lemma is immediate.

\begin{lemma}\label{l.Hn+1}
For each $n\geq 0$, the subgroup $${\rm GL}_{n+1}(\fv^{<n}) = {\rm Id}_{\fv^{<n+1}} + \fgl_{n+1}(\fv^{<n}) \subset H (\fv^{<n+1})$$ is a normal subgroup of $H(\fv^{<n+1})$ and the quotient group is isomorphic to the vector group  $\Hom(\fv^{<n}, \fv^n)$. \end{lemma}

\begin{lemma}\label{l.barG} Fix $n \geq 0$.
For $A \in \fh(\fv^{<n+1}) \subset \fgl_1(\fv^{<n+1})$, let $A^m \in \fgl^m(\fv^{< n+1})$ be its component of degree $m$. Then
\begin{itemize} \item[(i)] $A^m \in \Hom(\fv^{n-m}, \fv^n)$ for $1 \leq m \leq n$; and
\item[(ii)] if $({\rm Id}_{\fv^{<n+1}} + A)^{-1} = {\rm Id}_{\fv^{<n+1}} + \check{A}$  in $H(\fv^{<n+1})$ for some $\check{A} \in \fh(\fv^{<n+1})$, then $\check{A}^m = -A^m$ for $1 \leq m \leq n+1.$
    \end{itemize} \end{lemma}

    \begin{proof}
    (i) is immediate. From $$({\rm Id}_{\fv^{<n+1}} + \sum_{j \geq 1} \check{A}^j) \circ ({\rm Id}_{\fv^{<n+1}} + \sum_{m \geq 1} A^m) = {\rm Id}_{\fv^{<n+1}},$$ we have $$\sum_{j \geq 1} \check{A}^j + \sum_{m \geq 1} A^m + \sum_{j, m \geq 1} \check{A}^j \circ A^m = 0.$$
    But $\check{A}^j \circ A^m = 0$ if $2 \leq j + m \leq n+1$ by (i). This implies (ii). \end{proof}

\subsection{Prolongation of a fundamental graded Lie algebra}

\begin{definition}\label{d.prolongLie}
A graded Lie algebra $\fg^{-} \oplus \fg^0 = \oplus_{i=-k}^0 \fg^i$  is a {\em fundamental graded Lie algebra} if $\fg^-$ is generated by $\fg^{-1}$ and the adjoint representation $\fg^0 \to \End(\fg^{-})$ is injective. For a fundamental graded Lie algebra $\fg^{-} \oplus \fg^0 = \oplus_{i=-k}^0 \fg^i$,
define a (possibly infinite-dimensional) graded Lie algebra $\fg= \oplus_{i=-k}^{\infty} \fg^i$  inductively by setting for each $i\geq 1$, \begin{equation}\label{e.prolong} \fg^i := \{ A \in \fgl^i(\fg^{<i}) \mid A([u,v]) = [A(u), v] + [u, A(v)] \mbox{ for } u, v \in \fg^-\}.\end{equation} This graded Lie algebra $\fg=  \oplus_{j=-k}^{\infty} \fg^j$ is called the {\em universal prolongation} of $\fg^{-} \oplus \fg^0$.  If $\fg^i =0$ for $i > \ell$ for some $\ell \geq -1$, namely, if the universal prolongation is finite dimensional,   we write $\fg = \oplus_{i=-k}^{\ell} \fg^i$. \end{definition}

\begin{remark}\label{r.fundamental}
Our notion of a   fundamental graded Lie algebra is slightly different from the standard one (e.g., \cite{AD}, \cite{Ta70} and \cite{Yat}), which refers to only the part  $\fg^-.$
\end{remark}

\begin{definition}\label{d.partial}
Let  $\fg = \oplus_{i=-k}^{\ell} \fg^i$ be  the universal prolongation of a fundamental graded Lie algebra $\fg^- \oplus \fg^{0}.$
Write $${\rm Tor}^1(\fg) := \Hom^1(\fg^- \wedge \fg^-, \fg) = \oplus_{i, j <0} \Hom (\fg^i \wedge \fg^j, \fg^{i+j +1})$$ and
$${\rm Tor}^{n+1}(\fg):=\Hom^{n+1}(\fg^{-1}\wedge \fg^{-}, \fg^{<n+1}) \oplus \Hom(\oplus_{i=0}^{n-1}(\fg^{-1} \wedge \fg^i), \fg^{n-1})$$
 for $n \geq 1$.
\begin{itemize}
\item[(i)] For $A \in \fh(\fg^{<1}) = \fgl_1(\fg^{<1})$, let $A^1$ be its component of degree 1. Define   $\p_{\fg}^1: \fh(\fg^{<1})   \to {\rm Tor}^1(\fg) $ by setting for  $A \in \fgl_1(\fg^{<1})$ and  $u, v \in \fg^-$, $$\p_{\fg}^1 A(u,v) := A^1([u,v]) - [A^1(u), v] -[u, A^1(v)].$$
\item[(ii)]
For $A   \in \fh(\fg^{<n+1})=  \mathfrak{gl}_{n+1}(\fg^{<n+1}) + \Hom(\oplus_{i=0}^{n-1}\fg^i, \fg^n),$
let $A^m$ be its component of degree $m\geq 1$.
For $n \geq 1$, define $\p_{\fg}^{n+1}: \fh(\fg^{<n+1}) \rightarrow {\rm Tor}^{n+1}(\fg) $
by setting
$$ \partial_{\fg}^{n+1} A(u,v) :=
 A^{n+1}([u,v]) - [A^{n+1}(u), v] -[u, A^{n+1}(v)]$$ if $ u\wedge v \in \fg^{-1} \wedge \fg^- $ and $$  \partial_{\fg}^{n+1} A(u,v) :=
   -[u, A^{n-i}(v)] $$ if $ u\wedge v \in \fg^{-1} \wedge \fg^i $ with $0 \leq i \leq n-1.$
\end{itemize}
\end{definition}

\begin{lemma}\label{l.partial} In Definition \ref{d.partial}, the kernel of $\p_{\fg}^{n+1}$ is $\fg^{n+1} + \fgl_{n+2}(\fg^{<n+1})$ for any $n \geq 0$.
\end{lemma}

\begin{proof} We prove the case $n \geq 1$, skipping the easy case of $n=0$.
Note that we have a direct sum decomposition of the  operator $$\partial_{\fg}^{n+1} = \partial_- + \partial_0 + \cdots + \partial_{n-1}$$ with $$\partial_-= \p_{\fg}^{n+1}|_{\fgl_{n+1}(\fg^{<n+1})}: \fgl_{n+1}(\fg^{<n+1}) \to \Hom^{n+1}(\fg^{-1} \wedge \fg^-, \fg^{<n+1})$$ and $$\partial_i = \p_{\fg}^{n+1}|_{\Hom(\fg^i, \fg^n)}: \Hom(\fg^i, \fg^n) \to \Hom(\fg^{-1} \wedge \fg^i, \fg^{n-1})$$  for $ 0 \leq i \leq n-1.$
We claim that \begin{itemize} \item[(1)] ${\rm Ker}(\partial_i) =0$ for any $0 \leq i \leq n-1$; and \item[(2)] ${\rm Ker} (\partial_-)  = \fg^{n+1} + \mathfrak{gl}_{n+2}(\fg^{<n+1}).$
\end{itemize} By the above direct sum decomposition, the lemma follows from the claim.

 To check (1), let $A^{n-i}$ be an element of $\Hom(\fg^i, \fg^n)$ such that $\partial_i A^{n-i}=0$. Then $[u, A^{n-i}(v)]=0$ for any $u \in \fg^{-1}$ and $v \in \fg^i$. In other words, the element $A^{n-i}(v)$  of $\fg^n \subset \Hom(\fg^-, \fg^{<n})$ satisfies $A^{n-i}(v)|_{\fg^{-1}}=0$. This implies $A^{n-i}(v)|_{\fg^-}=0$ because $\fg^-$ is generated by $\fg^{-1}$ and $$A^{n-i}(v)[x,y] = [A^{n-i}(v)(x), y] + [x, A^{n-i}(v)y]$$
for all $x,y \in \fg^-$ from (\ref{e.prolong}).
 Thus $A^{n-i}=0$.

For (2), in the natural decomposition $\fgl_{n+1}(\fg^{<n+1}) = \fgl_{n+2}(\fg^{<n+1}) \oplus \fgl^{n+1}(\fgl^{<n+1})$, it is obvious that $\fgl_{n+2}(\fg^{<n+1}) \subset {\rm Ker} (\partial_-)$. Thus it suffices to check
$$ {\rm Ker}(\p_-) \cap \fgl^{n+1}(\fg^{<n+1}) = \fg^{n+1}.$$
 Let $A $ be an element of $\fgl^{n+1}(\fg^{<n+1})$. Then $\partial_- A =0$ if and only if
\begin{equation} \label{e.partial}
A^{n+1}([u,v]) - [A^{n+1}(u), v] -[u, A^{n+1}(v)]=0
\end{equation}
 for any $u \in \fg^{-1}$ and $ v \in \fg^-$. Thus (\ref{e.prolong}) implies  $\fg^{n+1}  \subset {\rm Ker}(\partial_-)$. On the other hand, if (\ref{e.partial}) holds,
then by Jacobi identity,
$$A^{n+1}([[u_1,u_2],v]) - [A^{n+1}([u_1,u_2]), v] -[[u_1, u_2], A^{n+1}(v)]=0$$   for any  $ u_1,u_2 \in \fg^{-1}$ and $ v \in \fg^-$.
Since $\fg^-$ is generated by $\fg^{-1}$, the equality (\ref{e.partial}) holds for any $u,v \in \fg^-$, proving  $A \in \frak g^{n+1}$ by (\ref{e.prolong}).
\end{proof}

\subsection{Filtration on a vector space}
\begin{definition}\label{d.filtration}
   A {\em   filtration} $V_{\bullet}$ on a vector space $V$ is a descending sequence of subspaces $$V= V_{-k} \supset V_{-k+1} \supset V_{-k+2}  \supset \cdots \supset V_{\ell -1} \supset V_{\ell} $$ with integers $k \geq 1$ and $\ell \geq -1$. We call $\ell$ the {\em height} of the filtration.  We use the convention $V_{-k-j} := V_{-k}$ and $V_{\ell + j} := 0$ for any positive integer $j$. Fix a positive integer $m$. \begin{itemize}
   \item[(i)] Define the associated graded vector space
   $$\gr_{(m)}(V_{\bullet}) := \oplus_{i=-k}^{\ell} \gr_{(m)}^i(V_{\bullet}) \mbox{ with }
   \gr_{(m)}^i(V_{\bullet}) := V_i / V_{i+m}.$$ We usually write $\gr(V_{\bullet})$ for $\gr_{(1)}(V_{\bullet}).$
   \item[(ii)] For each $-k \leq i \leq \ell$, the  quotient homomorphism ${\rm pr}^i_{(m)}: V_i \to \gr_{(m)}^i(V_{\bullet}) $   is  defined by $${\rm pr}^i_{(m)}(v) = v \mod V_{i+m}$$ for  $v \in V_i.$
        \item[(iii)] For each $-k \leq i \leq \ell$, define the linear homomorphism  $ {\rm pr}_{(m \to 1)}^i: \gr_{(m)}^i(V_{\bullet}) \to \gr_{(1)}^i(V_{\bullet}) $
   by $$ {\rm pr}_{(m\to 1)}^i(v \mod V_{i+m}) = v \mod V_{i+1} $$ for  $v \in V_i$ and
         the homomorphism of graded vector spaces ${\rm pr}_{(m \to 1)}: \gr_{(m)}(V_{\bullet}) \to \gr_{(1)}(V_{\bullet})$  by $${\rm pr}_{(m \to 1)} = \oplus_{i= -k}^{\ell} {\rm pr}^i_{(m \to 1)}.$$
        \end{itemize}
 Fix a graded vector space $\fv = \oplus_{i=-k}^{\ell} \fv^i$ with $\dim \fv^i = \dim V_i/V_{i+1}$ for all $-k \leq i \leq \ell$
         and a graded vector space isomorphism $$I= \oplus_{i=-k}^{\ell}  I^i \ : \ \fv \to \gr_{(1)}(V_{\bullet})$$ where $I^i : \fv^i \to  V_{i}/V_{i+1}$ is an isomorphism of vector spaces for each $i$.    \begin{itemize}
         \item[(iv)]  A linear isomorphism $\widetilde{I} \in {\rm Isom} (\fv,  V)$ is a {\em lift} of $I$ if, writing $\widetilde{I}^i:= \widetilde{I}|_{\fv^i}$, we have   $${\rm Im}(\widetilde{I}^i) =\widetilde{I}(\fv^i) \subset V_i \mbox{ and } {\rm pr}^i_{(1)} \circ \widetilde{I}^i = I^i $$ for each $-k \leq i \leq \ell.$ In other words, the following diagram is commutative.
             $$ \begin{array}{ccc} & & V_i \\ & \widetilde{I}^i \nearrow & \downarrow {\rm pr}^i_{(1)} \\ \fv^i & \stackrel{I^i}{\longrightarrow} & \gr^i_{(1)}(V_{\bullet}) \end{array} $$
              Equip $\fv$ with the filtration
         $$\fv_{\bullet} = (\fv_j := \oplus_{i= j}^{\ell} \fv^i, \ -k \leq j \leq \ell).$$
         Then $\widetilde{I}$ is a filtration-preserving isomorphism between $\fv_{\bullet}$ and $V_{\bullet}$. Denote by $\bL(I) \subset {\rm Isom}(\fv, V)$ the set of all lifts of $I$
    \item[(v)] A graded vector space homomorphism $J: \fv \to \gr_{(m)}(V_{\bullet})$ is an $m$-{\em lift} of $I$ if $${\rm pr}^i_{(m\to 1)} \circ J|_{\fv^i} =I|_{\fv^i}$$ for all $-k \leq i \leq \ell.$ In other words, the following diagram is commutative.
        $$\begin{array}{ccc} & & \gr_{(m)}(V_{\bullet}) \\ & J \nearrow & \downarrow {\rm pr}_{(m \to 1)} \\ \fv & \stackrel{I}{\longrightarrow} & \gr_{(1)}(V_{\bullet}).\end{array} $$ Denote by $\bL_{(m)}(I) \subset \Hom(\fv, \gr_{(m)}(V_{\bullet}))$ the set of all $m$-lifts of $I$.
        \item[(vi)]
            Define  ${\rm pr}^I_{(m)}:  \bL(I) \rightarrow \bL_{(m)}(I)$   by $${\rm pr}^I_{(m)}(\widetilde{I})|_{\fv^i} := {\rm pr}^i_{(m)} \circ \widetilde{I} |_{\fv^i} $$
for any $\widetilde{I} \in \bL(I)$ and  $-k \leq i \leq \ell$. 
        \end{itemize} \end{definition}

    The following is from Theorem 5 and Theorem 6 of \cite{AD}. The proof is elementary.

    \begin{lemma}\label{l.lift}
    In the setting of Definition \ref{d.filtration}, the following holds.
        \begin{itemize}
\item[(i)]    The group ${\rm GL}_1(\fv)$ acts simply transitively on $\bL(I)$ on the right: an element  $A \in {\rm GL}_1(\fv)$ sends $\widetilde{I} \in \bL(I)$ to $  \widetilde{I} \circ A$.
    \item[(ii)] For each $m \geq 1$, the
    homomorphism ${\rm pr}^I_{(m)}: \bL(I) \to \bL_{(m)}(I)$ is the
    quotient of $\bL(I)$ by the normal subgroup ${\rm GL}_m(\fv) \subset {\rm GL}_1(\fv).$
     \item[(iii)]
       For $J \in \bL_{(m)}(I),$ $v^i \in \fv^i, -k \leq i \leq \ell$,  and $A^j \in \fgl^j(\fv), 0 \leq  j \leq m-1,$ the element $J(A^j(v^i)) \in V_{i+j}/V_{i+j+m}$ determines a unique element in $V_i/V_{i+m}$, which we denote by  $$\overline{J(A^j(v^i))} \in V_i/V_{i+m} = \gr_{(m)}^i (V_{\bullet}).$$
     For $A = A^0 + A^1 + \cdots+ A^{m-1} \in {\rm GL}_1(\fv)/{\rm GL}_m(\fv)$  from (\ref{e.decompo}) and $v^i \in \fv^i$, define
        $$  J \cdot  A(v^i) :=  \overline{J(A^0(v^i))} + \overline{J(A^1(v^i))} + \cdots + \overline{J(A^{m-1}(v^i))} \ \in \ \gr^i_{(m)}(V_{\bullet}).$$ This defines an element $J \cdot A \in \bL_{(m)}(I)$.  Via the association $J \mapsto J \cdot A$,  the group  ${\rm GL}_1(\fv)/{\rm GL}_m(\fv)$ acts simply transitively on $\bL_{(m)}(I)$ on the right.
            \end{itemize} \end{lemma}

\subsection{Principal bundles and frame bundles}\label{ss.principal}

\begin{definition}\label{d.horizontal}
Let $\psi: Q \to M$ be a submersion between two complex manifolds. \begin{itemize}
\item[(i)] For a point $y \in Q$, a subspace $\sH_y \subset T_y Q$ is {\em $\psi$-horizontal} if $T_y Q = \sH_y \oplus {\rm Ker}({\rm d}_y \psi)$.
       \item[(ii)] A vector subbundle $\sH \subset TQ$ is a {\em $\psi$-connection} if the fiber $\sH_y \subset T_y Q$ is $\psi$-horizontal for each $y \in Q.$
        \end{itemize}
        \end{definition}

Recall the following standard terminology.

\begin{definition}\label{d.principal}
Let $G$ be a complex Lie group and $M$ be a complex manifold. \begin{itemize} \item[(i)] A complex manifold $P$  equipped with a right $G$-action and  a submersion $\psi: P \to M$
 is a {\em principal bundle with the structure group} $G$, if $G$ acts simply transitively on each fiber of $\psi$.
\item[(ii)] In (i), let $R_g: P \to P$ be the action of $g \in G$.  A vector subbundle $\sH \subset T P$ satisfying $TP = \sH \oplus {\rm Ker}({\rm d} \psi)$  and ${\rm d} R_g (\sH) = \sH$ for all $g \in G$ is called a {\em principal connection} on $P$. A principal connection is a $\psi$-connection in the sense of Definition \ref{d.horizontal}: for each $y \in P$, the fiber $\sH_y \subset T_y P$ is $\psi$-horizontal.    \end{itemize} \end{definition}

\begin{definition}\label{d.frame}
Let $M$ be a complex manifold. Fix a vector space $W$ with $\dim W = \dim M$.
\begin{itemize}\item[(i)] For each $x \in M$, define $$\F_x M := {\rm Isom}(W, T_x M)$$ the set of linear isomorphisms from $W$ to $T_x M$. The union $\F M = \cup_{x \in M} \F_x M$ with the natural projection $\pi^{\F M}: \F M \to M$ is the {\em frame bundle} of $M$. It is a principal bundle on $M$ with the structure group ${\rm GL}(W). $
A holomorphic section of $\pi^{\F M}$ is called an {\em absolute parallelism} on $M$.
\item[(ii)] For any $x \in M$ and $y \in \F_x M$, let $\theta_y: T_{y} (\F M) \to W$ be the composition $$T_{y}(\F M) \stackrel{ {\rm d} \pi^{\F M}}{\longrightarrow} T_x M \stackrel{ y^{-1}}{\longrightarrow} W.$$   The {\em soldering form} $\theta$ is the $W$-valued 1-form on $\F M$ whose value at $y \in \F M$ is $\theta_{y}$.
\item[(iii)] For a point $y \in \F M$, a $\pi^{\F M}$-horizontal subspace $\sH_y \subset T_y \F M$ and a vector $w \in W$, denote by $w^{\sH_y} \in \sH_y$ the unique element of $\sH_y$ satisfying $\theta(w^{\sH_y}) = w.$
 \item[(iv)]    The {\em torsion} of a $\pi^{\F M}$-horizontal subspace $\sH_y \subset T_{y} (\F M)$ is the homomorphism $\tau^{\sH_y} \in \Hom( \wedge^2 W, W)$ which sends $u \wedge v \in \wedge^2 W$ to  $$\tau^{\sH_y}(u, v) := {\rm d} \theta (u^{\sH_y}, v^{\sH_y}) \ \in W.$$
\item[(v)] If $\sH \subset T(\F M)$ is a principal connection, then each fiber $\sH_y \subset T_y \F M$ is a $\pi^{\F M}$-horizonal subspace. The $\Hom(\wedge^2 W, W)$-valued function $\tau^{\sH}$ on $\F M$ given by $y \mapsto  \tau^{\sH_y}$ is called the {\em torsion of the principal connection} $\sH$.   \end{itemize} \end{definition}

\subsection{Filtration on a complex manifold}\label{ss.filtmfd}
\begin{definition}\label{d.filtered}
Let $n \geq -1$ be an integer.  Let $M$ be a complex manifold.  \begin{itemize} \item[(i)]  A
{\em filtration of  height $n$} on $M$ is a descending sequence $\sD_{\bullet}$ of vector subbundles
$$TM = \sD_{-k} \supset \sD_{-k+1} \supset \cdots \supset \sD_{n},$$
which gives a filtration on the tangent space  $T_x M$ for each $x \in M$. It is convenient to define $\sD_{n+j} =0$ for $j \geq 1$.  \item[(ii)]
A filtration on $M$ as in (i)  is a {\em Tanaka filtration}, if the Lie brackets of  local sections satisfy $$[\sD_{i}, \sD_{j} ] \subset \sD_{i+j} \mbox{ for  } i, j \leq 0.$$
 \item[(iii)] For each $x \in M$, the Lie brackets of local vector fields in (ii) equip the graded vector space $${\rm symb}_{x}(\sD_{\bullet}) := \oplus_{i=-k}^{-1} (\sD_{i}/\sD_{i+1})_x$$ with a structure of a nilpotent graded Lie algebra. This nilpotent graded Lie algebra is called the {\em symbol algebra} of the Tanaka filtration at $x$.
\end{itemize}
\end{definition}

\begin{definition}\label{d.filtrationlift}
Let $\sD_{\bullet}$ be a filtration of height $n$ on a complex manifold $M$.
Let $f: Q \to M$ be a submersion from a complex manifold $Q$. Define a filtration $\sD^f_{\bullet}$ on $Q$ of height $n+1$ by
$$\sD^f_i = ({\rm d} f)^{-1} \sD_i  \mbox{ for } -k \leq  i \leq n \mbox{ and }\sD^f_{n+1} = {\rm Ker}({\rm d} f).$$ The filtration $\sD^f_{\bullet}$ is called the $f$-{\em lift} of the filtration $\sD_{\bullet}$.
\end{definition}

We skip the proof of the following elementary lemma.

\begin{lemma}\label{l.vertical}
In Definition \ref{d.filtrationlift}, the following holds.  \begin{itemize}
\item[(i)]
If $R: Q \to Q$ is a biholomorphic automorphism of $Q$ satisfying $R(f^{-1}(x)) = f^{-1}(x)$ for any $x \in M$, then ${\rm d} R: TQ \to TQ$ sends $\sD^f_i$ to itself for any $-k \leq i \leq n+1$.
\item[(ii)] If $\sD_{\bullet}$ is a Tanaka filtration, then so is $\sD^f_{\bullet}$. In this case, the differential ${\rm d}_y f: T_y Q \to T_x M$ for any $x \in M$ and $y \in f^{-1}(x)$  induces a graded Lie algebra isomorphism between the symbol  algebras   ${\rm symb}_y(\sD^f_{\bullet})$ and ${\rm symb}_x(\sD_{\bullet})$. \end{itemize}  \end{lemma}

\begin{definition}\label{d.Lifts}
Let $\sD_{\bullet}$ be a filtration of height $n \geq -1$ on a complex manifold $M$.
Fix  a graded vector space $\fv = \oplus_{i= -k}^{\ell} \fv^i$ with $\ell \geq n$ satisfying $\dim \fv^i = {\rm rank} (\sD_i/\sD_{i+1})$ for each $-k \leq i \leq n$, implying $\dim M = \dim (\fv^{<n+1}).$
\begin{itemize} \item[(i)]
Let $\gr (\sD_{\bullet})$ be the vector bundle on $M$  whose fiber at $x \in M$ is the graded vector space $$ \gr(\sD_{\bullet})_x := (\sD_{-k}/\sD_{-k+1})_x \oplus \cdots \oplus (\sD_{n -1}/\sD_{n})_x \oplus (\sD_{n})_x.$$ Note that $\gr(\sD_{\bullet})_x$ is $\gr(V_{\bullet}) = \gr_{(1)}(V_{\bullet})$ for the filtration $V_{\bullet}$ on $V= T_x M$ given by $$T_x M = (\sD_{-k})_x \supset (\sD_{-k+1})_x \supset \cdots \supset (\sD_n)_x.$$
\item[(ii)] Let $\fv^{<n+1} \times M$ be the trivial vector bundle. A vector bundle isomorphism $$ I: \fv^{<n+1} \times M \ \longrightarrow \ \gr(\sD_{\bullet})$$ is a {\em graded parallelism} of $(M, \sD_{\bullet})$ if for each $x \in M$, the vector space isomorphism $I_x: \fv^{<n+1} \to \gr(\sD_{\bullet})_x$ is an isomorphism of graded vector spaces.
    \item[(iii)] Assume that we have a graded parallelism $I$ of $(M, \sD_{\bullet})$. For each $x \in M$, we have the filtration $V_{\bullet}$ on $V= T_x M$ in (i) and the associated subsets $\bL(I_x) \subset {\rm Isom}(\fv^{<n+1}, V)$ and $\bL_{(m)}(I_x)
        \subset \Hom(\fv^{<n+1}, \gr_{(m)}(V_{\bullet}))$ from Definition \ref{d.filtration} (iv) and (v). We call $$ \bL(I):= \cup_{x \in M} \bL(I_x) $$   the {\em bundle of  lifts} of $I$ and denote by $\lambda^I: \bL(I) \to M$ the natural projection.   By Lemma \ref{l.lift} (i), this is a principal bundle on $M$ with the structure group ${\rm GL}_1(\fv^{<n+1})$. By setting $W = \fv^{<n+1}$ in Definition \ref{d.frame}, we can regard the principal bundle $\bL(I)$ as a principal subbundle of the frame bundle:
    $$  \begin{array}{ccc} \bL(I) & \subset & \F M \\
     \lambda^I \downarrow & & \downarrow \pi^{\F M} \\ M & = & M\end{array} $$ with the  inclusion of the structure groups ${\rm GL}_1(\fv^{<n+1}) \subset {\rm GL}(W)$.
   \item[(iv)]
    Define the {\em bundle of $m$-lifts  } of $I$,
     $$\bL_{(m)}(I) := \cup_{x \in M} \bL_{(m)}(I_x)$$  for each $m \geq 1$ with the natural projection $\lambda^I_{(m)}: \bL_{(m)}(I) \to M$.  From Definition \ref{d.filtration} (vi), we have a natural bundle map ${\rm pr}^I_{(m)}: \bL(I) \to \bL_{(m)}(I):$
     $$\begin{array}{ccc} \bL(I) & \stackrel{{\rm pr}^I_{(m)}}{\longrightarrow} & \bL_{(m)}(I) \\
     \lambda^I \downarrow & & \downarrow \lambda^I_{(m)} \\ M & = & M. \end{array} $$
   By Lemma \ref{l.lift} (ii), the map  ${\rm pr}^I_{(m)}$ is the quotient of $\bL(I)$ by the normal subgroup ${\rm GL}_m(\fv^{<n+1}) \subset {\rm GL}_1(\fv^{<n+1})$.  \end{itemize}
\end{definition}

\subsection{$B$-groups and $\fl$-exponential actions}\label{ss.Bgroup}

\begin{definition}\label{d.Bgroup}
Let $B$ be a fixed complex manifold. \begin{itemize} \item[(i)] A vector bundle $\sV$ on $B$ with a vector bundle homomorphism $\Lambda: \wedge^2 \sV \to \sV$ is a {\em $B$-Lie algebra} if for each $b \in B$, the homomorphism $\Lambda_b: \wedge^2 \sV_b \to \sV_b$ of vector spaces determines a Lie algebra structure on the fiber $\sV_b$.  \item[(ii)]
A complex manifold $\mathsf{G}$ equipped with a submersion $\gamma: \mathsf{G} \to B$ and a distinguished section $\epsilon: B \to \mathsf{G}$  is a {\em $B$-group} if for each $b \in B$, the fiber $\mathsf{G}(b) :=\gamma^{-1}(b)$ has the structure of a complex Lie group with the neutral element $\epsilon (b)$ such that the group operation depends holomorphically on $B$.
 A $B$-group $\gamma': \mathsf{G}' \to B$ which is a fiber subbundle of  $\gamma$ is  a  {\em $B$-subgroup} of $\mathsf{G}$ if $\mathsf{G}'(b)= (\gamma')^{-1}(b)$ is a subgroup of $\mathsf{G}(b)= \gamma^{-1}(b)$ for each $b \in B$.
\item[(iii)] For a $B$-group $\gamma: \mathsf{G} \to B$,  the relative tangent bundle  ${\rm Ker}({\rm d} \gamma)|_{\epsilon (B)}$  is a $B$-Lie algebra, which we denote by ${\rm Lie}\gamma: {\rm Lie}\mathsf{G} \to B.$ Its fiber $({\rm Lie}\gamma)^{-1}(b)$ can be identified with the Lie algebra ${\rm Lie} \mathsf{G}(b)$ of the Lie group $\mathsf{G}(b)$. The exponential maps along fibers determine a holomorphic map $$\exp^{\mathsf G}: {\rm Lie}\mathsf{G} \to  \mathsf{G}, \ \mbox{ satisfying } {\rm Lie} \gamma = \gamma \circ \exp^{\mathsf G}.$$
\item[(iv)]
  Let $\gamma: \mathsf{G} \to B$ be a $B$-group and let $\psi: Q \to B$ be a submersion from a complex manifold $Q$. A {\em  $B$-group action of $\mathsf{G}$ on $Q$ with respect to $\psi$} is a holomorphic map from the fiber product $Q \times_B \mathsf{G}$ to $ Q$ such that for each $b \in B$, its restriction $\psi^{-1}(b) \times \gamma^{-1}(b) \to \psi^{-1}(b)$ on the fibers over $b$ is a right action of the group $\mathsf{G}(b) = \gamma^{-1}(b)$ on the complex manifold $\psi^{-1}(b)$. \end{itemize} \end{definition}

    Note that a  $B$-group action in Definition \ref{d.Bgroup} (iv) is not a genuine group action on the manifold $Q$. We have only the action of each fiber of $\gamma$ on  each fiber of $\psi: Q\to B$. For this reason, it is convenient to introduce the following notion.

\begin{definition}\label{d.exponential}
Let $\fl$ be a fixed vector space and let $Q$ be a complex manifold. A family $\{R_{\exp(l)}: Q \to Q\mid l \in \fl\}$ of biholomorphic automorphisms of $Q$ is called an {\em $\fl$-exponential action} on $Q$
if \begin{itemize}
\item[(i)] for each $z \in Q$, the map $\fl \to Q$ defined by  $l \mapsto R_{\exp(l)}(z)$ is holomorphic; and
    \item[(ii)] $R_{\exp((t+t') l)} = R_{\exp(tl)} \circ R_{\exp(t' l)}$ for any  $t, t' \in \C$ and $l \in \fl$, namely,  the family $\{ R_{\exp(tl)} \mid t \in \C\}$ is a one-parameter group of automorphisms of $Q$. \end{itemize} The holomorphic vector field  on $Q$ generating the one-parameter group in (ii)  is  denoted by $l^Q$ and called the {\em fundamental vector field on $Q$ corresponding to} $l \in \fl.$ \end{definition}

\begin{example}\label{ex.vectorfield}
In Definition \ref{d.Bgroup} (iv), let $\fl \times B$ be the trivial vector bundle on $B$ given by a vector space $\fl$ and assume that we have an isomorphism
 $\zeta: \fl \times B \to {\rm Lie} \mathsf{G}$ of vector bundles on $B$.   Then each $l \in \fl$ determines a holomorphic section $\exp^{\mathsf G} (\zeta(l \times B))$ of $\gamma$, which we write simply as $\exp(l)$.  By the fiberwise group action, we have a family of biholomorphic automorphisms  $\{R_{\exp (l)}:  Q \to Q \mid  l\in \fl \},$ which defines an $\fl$-exponential action on $Q$. The fundamental vector field $l^Q$ on $Q$  corresponding to any $l \in \fl$ satisfies  ${\rm d} \psi (l^Q) =0.$   \end{example}

\begin{lemma} \label{l.fundamental} In Example \ref{ex.vectorfield}, pick a point $y \in Q$ and set $b=\psi(y)$.  Let $\zeta_b: \fl \to {\rm Lie} \mathsf{G}(b)$ be the restriction of the vector bundle isomorphism $\zeta$ to the fibers at $b$.   Then for any $l_1, l_2 \in \fl$,
$$[l^Q_1, l^Q_2]_y = [\zeta_b(l_1),  \zeta_b(l_2)]^Q_y,$$
where the bracket on the left hand side is the Lie bracket of vector fields on $Q$ and the bracket on the right hand side is the Lie bracket of the Lie algebra  ${\rm Lie} \mathsf{G}(b).$
\end{lemma}

\begin{proof}
The claim can be checked by restricting all the vector fields to the fiber $\psi^{-1}(b)$. So it follows from the standard result on right group actions on manifolds, for example, Proposition 4.1 in Chapter I of \cite{KN}.
\end{proof}

\begin{remark} If $B$ is one point, then a $B$-group (resp. $B$-Lie algebra) is a usual group (resp. Lie algebra). Then a  $B$-group action is simply a group action and an $\fl$-exponential action is just the group action composed with the exponential map from the Lie algebra. \end{remark}

\subsection{$\beta$-principal bundles and connections}\label{ss.connection}

    \begin{definition}\label{d.Bbundle}
    Fix a $B$-group $\gamma: \mathsf{G} \to B$. Let $M$ be a complex manifold equipped with a submersion $\beta: M \to B$. \begin{itemize}
\item[(i)]
A complex manifold $P$ with a submersion $\psi: P \to M$ is called a {\em $\beta$-principal bundle with the structure $B$-group} $\mathsf{G}$, if there is a  $B$-group action of $\mathsf{G}$ on $P$ with respect to the submersion $\beta \circ \psi: P \to B$, which gives the fiber $(\beta \circ \psi)^{-1} (b)$ for each $b \in B$  the structure of a principal bundle over the manifold $\beta^{-1}(b)$ with the structure group $\mathsf{G}(b)= \gamma^{-1}(b).$
\item[(ii)] In (i), if $\gamma': \mathsf{G}' \to B$ is a $B$-subgroup of $\mathsf{G},$
a $\beta$-principal bundle $\psi': P' \to M$ with the structure $B$-group $\mathsf{G}'$ is a {\em $\beta$-principal subbundle}
of $P$ if $P'$ is a fiber subbundle of $P$ and the  group operation of $(\gamma')^{-1}(b)$ on $(\beta \circ \psi')^{-1}(b)$ is compatible with the group operation of $\gamma^{-1}(b)$ on $(\beta \circ \psi)^{-1}(b)$ for each $b \in B$.
\end{itemize}
\end{definition}

A $\beta$-principal bundle does not have a natural group action on the right, unlike an ordinary  principal bundle. Consequently, the usual notion of a principal connection does not work. To remedy this, we introduce the following.

\begin{definition}\label{d.connection}
Let $\gamma: \mathsf{G} \to B,$  $\beta: M \to B$ and  $\psi: P \to M$ be as in Definition \ref{d.Bbundle}. Assume that we have \begin{itemize} \item[(1)]  a vector bundle isomorphism
 $\zeta: \fl \times B \cong {\rm Lie} \mathsf{G},$ which induces  an $\fl$-exponential action $R_{\exp(l)}: P \to P$  and the fundamental vector field $ l^P$ on $P$  for each $l \in \fl$ from Example \ref{ex.vectorfield};  \item[(2)]
 a holomorphic section $s: M \to P$ of $\psi$ with the image $\Sigma = s(M) \subset P$; and
 \item[(3)] a neighborhood $O \subset \fl$ of the zero $o \in \fl$ such that the exponential map $\exp^{\mathsf{G}}: {\rm Lie}\mathsf{G} \to \mathsf{G}$ sends $\zeta(O \times B) \subset {\rm Lie}\mathsf{G}$ biholomorphically to a neighborhood of $\epsilon (B) \subset \mathsf{G}$.
 \end{itemize}
 Let $U \subset P$ be the image of $\Sigma$ under the $B$-group action of the neighborhood $\exp^{\mathsf{G}}(\zeta(O \times B)) \subset \mathsf{G}$ in (3) such that $U$ is a neighborhood of $\Sigma$ in $P$.
Define a vector subbundle  $\sH \subset TP|_U$ by  $$\sH_{R_{\exp(l)}(s(x))} := {\rm d}_{s(x)} R_{\exp (l)} (T_{s(x)} \Sigma)$$ for each $x \in M$ and $l \in O$, such that $T_y P = \sH_{y} \oplus {\rm Ker}({\rm d}_y \psi)$  for any $y \in U \subset P$. The vector subbundle $\sH$ is called the {\em $\psi$-connection} on $U \subset P$ determined by $\zeta$ and $\Sigma$.
       \end{definition}

\begin{definition}\label{d.torsion}
In Definition \ref{d.connection}, fix a vector space $W$ with $\dim W = \dim M$ and define the frame bundle $\F M$ in terms of $W$. Assume that we have an inclusion
 $\mathsf{G} \subset {\rm GL}(W) \times B$ as $B$-groups  and  $\psi: P \to M$ is a $\beta$-principal subbundle of the frame bundle $\F M$ with the structure $B$-groups related by the given inclusion $\mathsf{G} \subset {\rm GL}(W) \times B$.   \begin{itemize} \item[(i)]
The {\em soldering form} $\theta$ on $P$ means the restriction of the soldering form of $\F M$ to $P$, a $W$-valued 1-form on $P$.
\item[(ii)] Assume that we have a trivialization $\zeta: \fl \times B \to  {\rm Lie} \mathsf{G}$ and a section $\Sigma \subset P$ of $\psi$ which determine the $\psi$-connection $\sH \subset TP|_U$ in a neighborhood $U$ of $\Sigma \subset P$, as in Definition \ref{d.connection}.
         For an element $w \in W$, the vector field $w^{\sH}$ whose value at $y \in U$ is
         $w^{\sH_y}$ from Definition \ref{d.frame} is called the {\em $\sH$-horizontal vector field} corresponding to $w \in W$.
             \end{itemize} \end{definition}

The following lemma is a  generalization of some standard results (namely, equations (11) - (16) in \cite{AD}) for principal subbundles of frame bundles to $\beta$-principal subbundles of frame bundles. Its proof is a modification of the corresponding arguments for principal subbundles. We give full proofs for the reader's convenience.

\begin{lemma}\label{l.Sternberg}
In the setting of Definition \ref{d.torsion}, the section $\exp(l)$ of $\gamma: \mathsf{G} \to B$ (resp. the section $\zeta(l)$ of ${\rm Lie}\mathsf{G} \to B$) determined by $l \in \fl$ in Example \ref{ex.vectorfield} can be regarded as a ${\rm GL}(W)$-valued  (resp. $\fgl(W)$-valued) function on $B$ via the inclusion $\mathsf{G} \subset {\rm GL}(W) \times B$.  Then the following holds.
\begin{itemize}
\item[(i)]  For any $l \in \fl$,
$$R^*_{\exp (l)} \theta = \exp(-l) \circ \theta \ \mbox{ and } \ l^P \, \rfloor {\rm d} \theta = - \zeta( l) \circ \theta. $$
\item[(ii)] Let $\sH_y, \sH'_y \subset T_y P$ be two $\psi$-horizontal subspaces at $y \in P \subset \F M$. For $u, v \in W$, we have unique $a, b \in \fl$ such that the fundamental vector fields $a^P, b^P$ on $P$ satisfy
    $$a^P_y = u^{\sH'_y} - u^{\sH_y} \mbox{ and } b^P_y = v^{\sH'_y} - v^{\sH_y}.$$ Then the torsions from Definition \ref{d.frame} (iv) satisfy $$\tau^{\sH'_y} (u, v) - \tau^{\sH_y}(u,v) = - \zeta(a) (v) + \zeta(b) (u).$$
\end{itemize}
Let $\sH$ be the $\psi$-connection on $U \subset P$ determined by $\Sigma \subset O \subset \fl$ and $\zeta$ in Definition \ref{d.connection}, and let $w^{\sH}$ for $w \in W$ be the $\sH$-horizontal vector field on $U$ from Definition \ref{d.torsion}.  \begin{itemize}
\item[(iii)] For  $l \in O$, $${\rm d} R_{\exp (l)} (w^{\sH})= (\exp(-l) (w))^{\sH} \mbox{ and } [l^P, w^{\sH}] = (\zeta(l) (w))^{\sH}$$ at points of $U$ where both sides make sense.
    \item[(iv)]  For any $y \in U $ and $u, v \in W$,
    $$\tau^{\sH_y}(u, v) = - \theta_y ([u^{\sH}, v^{\sH}]) = - y^{-1} \circ {\rm d}_y \psi ([u^{\sH}, v^{\sH}]). $$
    \item[(v)] For any $l \in O, y \in U$ and $u, v \in W$,
    $$\tau^{\sH_{R_{\exp(l)}(y)}}(u, v) = \exp(-l) \circ \tau^{\sH_y}( \exp(l) (u), \exp (l) (v)),$$ at points of $U$ where both sides make sense.
\end{itemize} \end{lemma}

\begin{proof}
For any $l \in \fl, y \in P$ and  a tangent vector $\vec{u} \in T_y P$, we claim \begin{equation}\label{e.trans} {\rm d}_{y}
\psi (\vec{u}) = ({\rm d}_{R_{\exp (l)(y)}} \psi  \circ {\rm d}_{y} R_{\exp (l)}) (\vec{u}). \end{equation}
In fact, for an arc $\{ y_t \in P | t \in \Delta \}$ with $y_0 = y$ and $\vec{u} = \frac{\rm d}{{\rm d} t}|_{t=0} y_t,$ write $x_t = \psi(y_t) \in M$ and denote by $l_t \in \fgl(W)$ the value of $\zeta(l)$ at $\beta(x_t)$.  Then $$({\rm d}_{R_{\exp (l)}(y)} \psi \circ {\rm d}_{y} R_{\exp (l)}) (\vec{u}) = \frac{\rm d}{{\rm d} t}|_{t=0} ( \psi \circ R_{\exp (l)} (y_t)) = \frac{\rm d}{{\rm d} t}|_{t=0} x_t = {\rm d}_y
\psi(\vec{u}),$$ which proves (\ref{e.trans}).

Let $\exp(l_b) \in {\rm GL}(W)$ be the value at $b = \beta(\psi(y))$ of the ${\rm GL}(W)$-valued function $\exp(l) $ on $B$.
By (\ref{e.trans}), $$\theta ( {\rm d}_{y} R_{\exp (l)} \vec{u}) = (y \circ \exp (l_b))^{-1} ({\rm d}_{y} \psi (\vec{u})) = (\exp (l_b))^{-1} \theta (\vec{u}),$$ which proves the first equation in (i). Taking derivative of
$R_{\exp (t l)}^* \theta = \exp (-t l) \circ \theta $  with respect to $t \in \C$, we obtain the Lie derivative
${\rm Lie}_{l^P} \theta = -\zeta(l) \circ \theta.$ Since $\theta (l^P) \equiv 0$, this implies the second equation in (i) by Cartan's formula
$${\rm Lie}_{l^P} \theta = {\rm d} (\theta (l^P)) + l^P \  \rfloor {\rm d} \theta.$$

In (ii), the existence of $a, b \in l$ is straightforward. Then
\begin{eqnarray*} \lefteqn{ \tau^{\sH'_y}(u,v) - \tau^{\sH_y}(u,v) } \\   && =
{\rm d} \theta (u^{\sH'_y}, v^{\sH'_y})- {\rm d} \theta(u^{\sH_y}, v^{\sH_y}) \\ && =
{\rm d} \theta ((u^{\sH'_y} - u^{\sH_y}), v^{\sH'_y}) + {\rm d} \theta ( u^{\sH_y}, ( v^{\sH'_y} - v^{\sH_y})) \\ && = {\rm d} \theta(a^P_y, v^{\sH'_y}) + {\rm d} \theta(u^{\sH_y}, b^P_y).
 \end{eqnarray*} Applying the second equation in (i), we obtain (ii).

For the first equation in (iii), we need to check for $y\in U, b = \beta(\psi(y))$ and $l \in O$ satisfying $R_{\exp(l)}(y) \in U$,
$$ {\rm d}_y R_{\exp(l)} (w^{\sH_y}) = (\exp(-l_b) (w))^{\sH_{R_{\exp(l)}(y)}}.$$
By our definition of $\sH$, the left hand side  is contained in $ \sH_{R_{\exp (l)}(y)}$. Thus it suffices to show that \begin{equation}\label{e.Sternberg} {\rm d}_{R_{\exp (l)}(y)} \psi
  ({\rm d}_{y}  R_{\exp (l)} ( (w^{\sH})_{y}) ) = R_{\exp(l)}(y) (\exp (-l_b) \cdot w),\end{equation}
  where $R_{\exp(l)}(y)$ on the right hand side is regarded as an element of $\F_{\psi(y)} M = {\rm Isom}(W, T_{\psi(y)} M)$. The left hand side of (\ref{e.Sternberg}) is $ {\rm d}_{y} \psi ((w^{\sH})_{y}) = y (w) $ by (\ref{e.trans}). This is equal to the right hand side of (\ref{e.Sternberg}) because  $y(w) =R_{\exp (l)}(y) ( \exp (-l_b) \cdot w)$ in terms of the ${\rm GL}(W)$-action on $\F M.$
This proves the first equation in (iii).
It follows that
  $$[l^P, w^{\sH}] = \lim_{t \to 0} \frac{1}{t} (w^{\sH} - {\rm d} R_{\exp (t l)} (w^{\sH}))
= \lim_{t \to 0} \frac{1}{t} (w^{\sH} - (\exp (-tl) (w))^{\sH}).$$
Since $\lim_{t \to 0} \frac{1}{t}( w - \exp (-t l) (w)) = \zeta(l) (w)$,
the above limit converges to $(\zeta(l) ( w))^{\sH}, $ proving the second equation in (iii).

Since $\theta (u^{\sH}) \equiv u$  and $\theta (w^{\sH}) \equiv w$ are constants for $u, w \in W$, the torsion
 $\tau^{\sH} (u, w) = {\rm d} \theta (u^{\sH}, w^{\sH})$ must satisfy (iv).

By ${\rm d}_{R_{\exp(l)}(y)} \psi = {\rm d}_y \psi \circ {\rm d}_{R_{\exp(l)}(y)} R_{\exp(-l)}$ from $\psi = \psi \circ R_{\exp(-l)}$ and (iii),
\begin{eqnarray*} {\rm d}_{R_{\exp(l)}(y)} \psi ([u^{\sH}, v^{\sH}]) &=& {\rm d}_y \psi ([{\rm d}R_{\exp(-l)} u^{\sH}, {\rm d} R_{\exp(-l)} v^{\sH}]) \\ &=& {\rm d}_y \psi ([(\exp(l_b) (u))^{\sH}, (\exp(l_b) (v))^{\sH}]). \end{eqnarray*} This combined with (iv) gives (v).
\end{proof}

\section{Generalized Tanaka prolongation: Statement}\label{s.Tanaka}
\begin{setup}\label{n.B} Throughout this section, we denote by $B$ a complex manifold  equipped with the following additional data.
Fix a graded vector space $\fv = \oplus_{i= -k}^{\ell} \fv^i$ with $\ell \geq 0$ and consider
 the trivial vector bundles on $B$ $$ \fv^- \times B \ \subset \  \fv^{< n} \times B \ \subset \    \fv \times B,$$ where  $n$ is an integer bigger than $ -k$.
Assume that we are given a $B$-Lie algebra structure on $\fv \times B$, namely, a vector bundle homomorphism $$\Lambda: (\wedge^2 \fv) \times B \ \longrightarrow \ \fv \times B,$$ with the following properties for each $b \in B$. \begin{itemize} \item[(i)]  The homomorphism $\Lambda_b: \wedge^2 \fv \to \fv$ determines a graded Lie algebra structure on $\fv = \oplus_{i= -k}^{\ell} \fv^i,$ which we denote by $\fg(b) = \oplus_{i=-k}^{\ell} \fg(b)^i$.
\item[(ii)]
 The graded Lie algebra
 $ \fg(b)^- \oplus \fg(b)^{0}$ is a fundamental graded Lie algebra and $\fg(b)$ is its universal prolongation.
\end{itemize} Then we have  the following $B$-groups, generalizing the groups introduced in Section 3.3 of \cite{AD}.
\begin{itemize} \item[(1)] From (ii), we have a  family of Lie subalgebras $\{ \fg(b)^0 \subset \fgl(\fv^-) \mid b \in B\}$ by the adjoint representation, which induces an inclusion $\fv^0 \times B \subset \fgl(\fv^-) \times B.$
Denote by  $\mathsf{G}^0 \subset  {\rm GL}(\fv^-) \times B$  the $B$-group consisting of  connected subgroups  $$ \{ \mathsf{G}^0(b) \subset {\rm GL}(\fv^-)  \mid b \in B \}$$ with the Lie algebras $\{ \fg^0(b) := \fg(b)^0 \subset \fgl(\fv^-) \mid b \in B\}$. We have a tautological isomorphism of vector bundles on $B$ $$\zeta^0: \fv^0 \times B \longrightarrow  {\rm Lie} \mathsf{G}^0.$$
\item[(2)] For each $n \geq 1$, denote by $\mathsf{G}^n {\rm GL}_{n+1}(\fv^{<n})$ the $B$-subgroup of the trivial bundle of groups $$\mathsf{G}^n {\rm GL}_{n+1}(\fv^{<n}) \  \subset \  {\rm GL}(\fv^{<n}) \times B $$ such that its fiber at $b \in B$  consists of elements in ${\rm GL}(\fv^{< n})$ of the form ${\rm Id}_{\fv^{<n }} + A^n + A_{n+1}$ where $A^n \in \fg(b)^n \subset \fgl^n(\fv^{<n})$ and $A_{n+1} \in \fgl_{n+1}(\fv^{< n})$.
    Then  $$\mathsf{G}^n {\rm GL}_{n+1}(\fv^{<n}) \  \subset \  {\rm GL}_n(\fv^{<n}) \times B  \ \subset \ H(\fv^{<n}) \times B,$$ where $H(\fv^{<n}) \subset {\rm GL}_1(\fv^{<n})$ is from Definition \ref{d.barG}.
    We have an induced isomorphism of vector bundles on $B$ $$\eta^n: (\fv^n \oplus \fgl_{n+1}(\fv^{<n}))  \times B \longrightarrow  {\rm Lie} \mathsf{G}^n{\rm GL}_{n+1}(\fv^{<n}).$$
\item[(3)] The quotient of the $B$-group  $\mathsf{G}^n {\rm GL}_{n+1}(\fv^{<n})$ in (2)   by the  normal $B$-subgroup $$ {\rm GL}_{n+1}(\fv^{<n}) \times B \ \subset  \ \mathsf{G}^n {\rm GL}_{n+1}(\fv^{<n})$$ is denoted by   $\mathsf{G}^n \to B$. We have an induced isomorphism of vector bundles on $B$ $$\zeta^n: \fv^n \times B \to {\rm Lie}\mathsf{G}^n.$$
     \item[(4)] For each $n \geq -1$ and each $b \in B$, let $\mathsf{S}^{<n+1}(b) \subset {\rm GL}(\fv^{<n+1})$ be the subgroup consisting of all graded vector space automorphisms of $\fv^{<n+1}$ which are graded Lie algebra automorphisms of $\fg(b)^-$. Then $\mathsf{S}^{<n+1} = \cup_{b \in B} \mathsf{S}^{<n+1}(b)$ is a $B$-subgroup of ${\rm GL}(\fv^{< n+1}) \times B$.
        The $B$-group $\mathsf{G}^0$ in (1) is a $B$-subgroup of $\mathsf{S}^{<0}.$
        \end{itemize}
        We also have the following objects over $B$. \begin{itemize}
        \item[(5)] As in Definition \ref{d.partial},
        write $${\rm Tor}^1(\fv) := \Hom^1 (\fv^- \wedge \fv^-, \fv)$$ and
$${\rm Tor}^{n+1}(\fv):=\Hom^{n+1}(\fv^{-1}\wedge \fv^{-}, \fv^{<n+1}) \oplus \Hom(\oplus_{i=0}^{n-1}(\fv^{-1} \wedge \fv^i), \fv^{n-1})$$
 for $n \geq 1$.         \item[(6)]
        For $n \geq 0$, define the vector bundle homomorphisms between trivial vector bundles on $B$
        $$\p^{n+1}:  \fh(\fv^{<n+1}) \times B \to {\rm Tor}^{n+1}(\fv) \times B,$$ such that its restriction on the fibers over $b \in B$ is $$\p^{n+1}_b : = \p^{n+1}_{\fg(b)}$$
     the  operator for the graded Lie algebra $\fg(b)$ as defined in Definition \ref{d.partial}. \item[(7)] By Lemma \ref{l.partial}, the vector bundle homomorphism $\p^{n+1}$ has a constant rank for any $n \geq 0$. We assume that there exists a vector subbundle $\sW^{n+1} \subset {\rm Tor}^{n+1}(\fv) \times B$ such that $${\rm Tor}^{n+1}(\fv) \times B = {\rm Im}(\p^{n+1}) \oplus \sW^{n+1}$$ for any $n \geq 0$.
                \end{itemize}
\end{setup}

\begin{definition}\label{d.betaframe}
Let $n \geq -1$ be an integer.
Let $M$ be a complex manifold with a submersion $\beta: M \to B$, where $B$ is as in Set-up \ref{n.B}.
\begin{itemize} \item[(i)] A Tanaka filtration $\sD_{\bullet}$ of height $n$ on $M$ is a {\em $\beta$-Tanaka filtration} if  ${\rm rank}(\sD_i/\sD_{i+1}) = \dim \fv^i$ for all $-k \leq i \leq n$ and the symbol algebra ${\rm symb}_x (\sD_{\bullet})$ at each $x \in M$ is isomorphic to the graded Lie algebra $\fg(b)^-$ for  $b= \beta(x)$.  \end{itemize}
Fix a $\beta$-Tanaka filtration  $\sD_{\bullet}$ of height $n$ on $M$. We have the graded vector bundle  $\gr (\sD_{\bullet})$ from Definition \ref{d.Lifts}.  Since $\dim M = \dim \fv^{<n+1}$, we can use the vector space $W = \fv^{<n+1}$ to define the frame bundle $\F M$ in Definition \ref{d.frame}, identifying $\F_x M$ with ${\rm Isom}(\fv^{<n+1}, T_x M)$ for each $x \in M$.
\begin{itemize}
\item[(ii)]  For each $x \in M$ and $b = \beta(x) \in B$, let  $\I^{<n+1}_x M $  be the set of graded vector space isomorphisms from $\fv^{<n+1}$ to $\gr(\sD_{\bullet})_x $ whose restrictions to $\fv^-$ is a
      Lie algebra isomorphism from $\fg (b)^-$ to ${\rm symb}_x(\sD_{\bullet})$.  The fiber bundle $\I^{<n+1} M := \cup_{x \in M} \I^{<n+1}_x M$ over $M$ is a $\beta$-principal bundle with the structure $B$-group $\mathsf{S}^{<n+1}$.
\item[(iii)] A holomorphic section $I$ of the $\beta$-principal bundle $\I^{<n+1} M$ is called a {\em  $\beta$-Tanaka parallelism} of $(M, \sD_{\bullet}).$
    More precisely, a $\beta$-Tanaka parallelism $I$ is a graded parallelism in the sense of Definition \ref{d.Lifts} whose
value $I_x$  at a point $x \in M$ is a graded vector space isomorphism $$ I_x: \fv^{< n+1} \to \oplus_{i= -k}^{n} (\sD_i/\sD_{i+1})_x $$
     such that its restriction on $\fv^-$, $$ I_x|_{\fv^-}: \oplus_{i= -k}^{-1} \fg(b)^i  \to {\rm symb}_x(\sD_{\bullet}) = \oplus_{i= -k}^{-1} (\sD_i/\sD_{i+1})_x$$ is a graded Lie algebra isomorphism. \end{itemize}
From Definition \ref{d.Lifts}, given a $\beta$-Tanaka parallelism $I$ of $(M, \sD_{\bullet})$, we have  \begin{itemize} \item   the bundle $\lambda^I: \bL(I) \to M$ of lifts of $I$, which is a principal subbundle of the frame bundle $\F M$ with the structure group ${\rm GL}_1(\fv^{<n+1})$; \item the bundle $\lambda_{(m)}^I: \bL_{(m)}(I) \to M$ of $m$-lifts of $I$; and \item the projection ${\rm pr}^I_{(m)}: \bL(I) \to \bL_{(m)}(I).$ \end{itemize} \end{definition}

\begin{definition}\label{d.betaTanaka}
Let $B$ be as in Set-up \ref{n.B}.
Let  $M$ be a complex manifold equipped with a submersion $\beta: M \to B.$
A {\em $\beta$-Tanaka structure} on $M$ is a  pair $(\sD_{\bullet}, \sP \subset \I^{<0} M)$ consisting of a  $\beta$-Tanaka filtration $\sD_{\bullet}$ of height $-1$ on $M$ and  a $\beta$-principal subbundle $\sP \subset \I^{<0} M$ with the structure $B$-group $\mathsf{G}^0 \subset \mathsf{S}^{<0}$. Recall that  for each $b \in B$, the fibers $$ \{ \sP_x \subset \I^{<0}_x M \mid x \in \beta^{-1} (b) \}$$  determine a principal subbundle of $\I^{<0} M|_{\beta^{-1}(b)}$ with the structure group $\mathsf{G}^0(b) \subset \mathsf{S}^{<0}(b)$.
\end{definition}

Now we can formulate the following generalization of Theorem 18 in \cite{AD}. Our formulation, especially (B) is more refined than that of \cite{AD}, even when $B$ is a point (for systems with constant symbols). This is to make it clear that the inductive procedure of Theorem \ref{t.Tanaka1} is compatible with the formal equivalence discussed in Section \ref{s.converge}.

\begin{theorem}\label{t.Tanaka1} Let $B$ be as in Set-up \ref{n.B}.
Let $M$ be a complex manifold equipped with a submersion $\beta: M \to B$.
Let $(\sD_{\bullet}, \sP \subset \I^{<0} M)$ be a $\beta$-Tanaka structure on $M$. Then for each $n \geq 1$, there is a canonically defined submersion $\bar{\pi}^n: \bar{P}^{(n)} \to \bar{P}^{(n-1)}$ between complex manifolds with the following properties. \begin{itemize}
\item[(A)] The base $\bar{P}^{(n-1)}$ is equipped with
\begin{itemize} \item[(A1)] a submersion $\beta^{(n-1)}: \bar{P}^{(n-1)} \to B$;  
\item[(A2)]  a $\beta^{(n-1)}$-Tanaka filtration of height $n-1$ $$\sD^{(n-1)}_{\bullet} = \{ T\bar{P}^{(n-1)} = \sD^{(n-1)}_{-k} \supset \cdots \supset \sD^{(n-1)}_{n-1}\}; \mbox{ and }$$
\item[(A3)] a  $\beta^{(n-1)}$-Tanaka parallelism $I^{(n-1)}$ of $(\bar{P}^{(n-1)}, \sD_{\bullet}^{(n-1)})$,  $$\{ I^{(n-1)}_y : \fv^{< n} \to \gr (\sD^{(n-1)}_{\bullet})_y \mid y \in \bar{P}^{(n-1)}\}.$$  \end{itemize}
\item[(B)] There is
a natural $\beta^{(n-1)}$-principal bundle $\widetilde{\pi}^n: \widetilde{P}^n \to \bar{P}^{(n-1)}$ with the structure $B$-group $\mathsf{G}^n {\rm GL}_{n+1}(\fv^{< n})$ such that
\begin{itemize} \item[(B1)] $\widetilde{P}^n$ is a $\beta^{(n-1)}$-principal subbundle of  $\bL(I^{(n-1)}) \subset \F \bar{P}^{(n-1)},$ where the frame bundle $\F \bar{P}^{(n-1)}$ is viewed as the bundle of isomorphisms from the vector space $\fv^{<n}$ to the tangent spaces of $\bar{P}^{(n-1)}$;
\item[(B2)]  the submersion $\widetilde{\pi}^n$ can be factored into two submersions
$$ \widetilde{P}^n \ \stackrel{\chi^n}{\longrightarrow} \  \bar{P}^{(n)} \ \stackrel{ \bar{\pi}^{(n)}}{\longrightarrow} \ \bar{P}^{(n-1)},$$ where $\chi^n$ is the quotient
by the normal subgroup ${\rm GL}_{n+1}(\fv^{<n})$ of $\mathsf{G}^n {\rm GL}_{n+1}(\fv^{<n})$ and $\bar{\pi}^{(n)}$ is
  a $\beta^{(n-1)}$-principal bundle with the structure $B$-group $\mathsf{G}^n$; and  \item[(B3)] there is a canonical embedding  $\xi^{(n)}: \bar{P}^{(n)} \to \bL_{(n+1)}(I^{(n-1)})$ with the commutative diagram
  $$ \begin{array}{ccc} \widetilde{P}^n & \subset & \bL(I^{(n-1)}) \\
  \chi^n \downarrow & & \downarrow {\rm pr}^{I^{(n-1)}}_{(n+1)} \\ \bar{P}^{(n)} & \stackrel{\xi^{(n)}}{\longrightarrow} & \bL_{(n+1)}(I^{(n-1)}) \\ \bar{\pi}^{(n)} \downarrow & & \downarrow \lambda^{I^{(n-1)}}_{(n+1)} \\ \bar{P}^{(n-1)} & = & \bar{P}^{(n-1)}, \end{array} $$
such that $\bar{\pi}^{(n)}$ is   a $\beta^{(n-1)}$-principal subbundle of $\lambda^{I^{(n-1)}}_{(n+1)}$ with the natural inclusion of  structure $B$-groups  $$\mathsf{G}^n \subset ({\rm GL}_1(\fv^{<n})/{\rm GL}_{(n+1)}(\fv^{<n})) \times B.$$ \end{itemize}
\item[(C)] For any point $z \in \widetilde{P}^n \subset \F \bar{P}^{(n-1)}$
     and a $\widetilde{\pi}^n$-horizontal subspace  $\sH_z \subset T_z \widetilde{P}^n$,
     the torsion $\tau^{\sH_z} \in \Hom(\wedge^2\fv^{<n }, \fv^{<n })$ satisfies the following.
\begin{itemize}
\item [(C1)] The restriction $\tau^{\sH_z}|_{ \fv^{-1}   \wedge \fv^{<n }}$ has only components of nonnegative  homogeneous degree, that is, for $-k \leq i <n$,
    $$\tau^{\sH_z}(\fv^{-1} \wedge \fv^i) \subset \oplus_{j=i-1}^{n-1}\fv^{j}.$$
\item [(C2)] If $u \wedge v \in \fv^{-1} \wedge \fv^{<n }$,  
    then the degree-zero component $(\tau^{\sH_z})^0$ of $\tau^{\sH_z}$ is given by
    $$(\tau^{\sH_z})^0(u,v) = -[u,v],$$
   where the Lie bracket is in the  Lie algebra $\mathfrak g(b)$ with $b=\beta^{(n-1)} \circ \widetilde{\pi}^{n }(z)$.
\end{itemize}
\end{itemize}
\end{theorem}

\begin{remark}\label{r.natural}
The notion of $\beta$-Tanaka structures and all the constructions in Theorem \ref{t.Tanaka1} are determined by the holomorphic map $\beta: M \to B$ and the isomorphism type of the $B$-Lie algebra in Set-up \ref{n.B}. They are independent of   the choices of the trivial bundle $\fv \times B$ (and the induced trivializations $\zeta^n$ and $\eta^n$).
\end{remark}

We give two immediate consequences of  Theorem \ref{t.Tanaka1}. The following is a generalization of Theorem 19 of \cite{AD}.

\begin{corollary}\label{c.Tanaka2} Let $B$ and $\ell \geq 0$ be as in Set-up \ref{n.B}. Let  $M$ be a complex manifold equipped with a submersion $\beta: M \to B$ and let $(\sD_{\bullet},  \sP \subset \I^{<0} M)$ be a $\beta$-Tanaka structure on $M$. Then the complex manifold $\bar{P}^{\ell}$ in Theorem \ref{t.Tanaka1} has a natural absolute parallelism, namely, a natural holomorphic section of $\F \bar{P}^{\ell}$.
\end{corollary}

\begin{proof}
For any $n \geq \ell+1$, the submersion $\bar{\pi}^{(n)}: \bar{P}^{(n)} \rightarrow \bar{P}^{(n-1)}$ is  biholomorphic because $\frak g^n(b)$ is zero.
Since $\bL_{(n+1)}(I^{(n-1)}) = \bL(I^{(n-1)})$ if   $n \geq \ell +k$,
the condition (B3) in Theorem \ref{t.Tanaka1} for  $n= \ell + k +1$ gives an embedding
$$\xi^{(\ell + k +1)}: \bar{P}^{(\ell + k)} = \bar{P}^{(\ell + k +1)} \to \bL_{(\ell + k +2)}(I^{(\ell+k)}) = \bL(I^{(\ell + k)}),$$
which is a section of $\bL(I^{(\ell +k)}) \to \bar{P}^{(\ell + k)}$, a natural absolute parallelism on $\bar{P}^{(\ell +k)}$.
Since the composition $$\bar{\pi} = \bar{\pi}^{(\ell +1)} \circ \cdots \circ \bar{\pi}^{(\ell + k)}: \bar{P}^{(\ell + k)} \to \bar{P}^{(\ell)}$$ is biholomorphic, this gives a natural absolute parallelism on $\bar{P}^{(\ell)}.$
\end{proof}

The following generalizes Theorem 8.4 of \cite{Ta70} and Proposition 51 of \cite{AD}.

\begin{corollary}\label{c.auto} The dimension of the Lie algebra of all infinitesimal automorphisms of a $\beta$-Tanaka structure in Definition \ref{d.betaTanaka} is  at most $\dim \fv$. \end{corollary}

\begin{proof}
In Corollary \ref{c.Tanaka2}, any automorphism of the $\beta$-Tanaka structure can be lifted to an automorphism of the absolute parallelism on $\bar{P}^{\ell}$. Thus Corollary \ref{c.auto} is an  immediate consequence of  Theorem 3.2 in Chapter I of \cite{Ko} on the infinitesimal automorphisms of an absolute parallelism. \end{proof}

\section{Generalized Tanaka prolongation: Proof of   the step $n=1$}\label{s.n=1}

We give the proof of Theorem \ref{t.Tanaka1} in this section and the next section.

\subsection{Definition of $\bar{P}^{(0)}$}\label{ss.P0}
We start with   a $\beta$-Tanaka structure $(\sD_{\bullet}, \sP \subset \I^{<0} M)$  on a complex manifold $M$ equipped with a submersion $\beta: M \to B$.
Let $\pi^{\sP}: \sP \to M$ be the natural projection, which is a $\beta$-principal bundle with the structure $B$-group $\mathsf{G}^0$.  Via  Example \ref{ex.vectorfield}, the isomorphism $\zeta^0: \fv^0 \times B \to {\rm Lie}\mathsf{G}^0$ of vector bundles on $B$ from Set-up \ref{n.B} gives rise to a $\fv^0$-exponential action on $\sP$ and the fundamental vector field $l^{\sP}$ on $\sP$ for each $l \in \fv^0$.

We set $\bar{P}^{(0)} = \sP$ and check the condition (A) as follows.
\begin{itemize} \item[(A1)] Define $\beta^{(0)}: \bar{P}^{(0)} \to B$ as the composition $$\bar{P}^{(0)} = \sP \stackrel{\pi^{\sP}}{\longrightarrow} M \stackrel{\beta}{\longrightarrow} B.$$
\item[(A2)] Applying Lemma \ref{l.vertical} (ii) to the $\beta$-Tanaka filtration $\sD_{\bullet}$ of height $-1$ on $M$, define the $\beta^{(0)}$-Tanaka filtration $\sD^{(0)}_{\bullet}$ of height $0$ on $\bar{P}^{(0)}$ as the $\pi^{\sP}$-lift $\sD^{\pi^{\sP}}_{\bullet}.$
    \item[(A3)] For $y \in \sP$, write $x = \pi^{\sP}(y)$ and $b = \beta(x)$.
     Since $\sP|_{\beta^{-1}(b)} = (\beta^{(0)})^{-1}(b) \to \beta^{-1}(b)$ is a principal bundle with the structure group $\mathsf{G}^0(b)$, we have a natural isomorphism \begin{equation}\label{e.I0} \fv^0= \fg(b)^0 \to {\rm Ker}({\rm d}_y \pi^{\sP}) = (\sD^{(0)}_0)_y, \end{equation} which sends $l \in \fv^0$ to $l^{\sP}_y$, the value of  the fundamental vector field at $y$. Denote by $(I^{(0)}_{y})^0: \fv^0 \to (\sD^{(0)}_0)_y$ the isomorphism (\ref{e.I0}).
     Furthermore, the inclusion $\sP_x  \subset \I^{<0}_x M$ determines a graded vector space isomorphism $(I^{(0)}_y)^{-}$ from $\fv^{-}$ to $$ (\sD_{-k}/\sD_{-k+1})_x \oplus \cdots \oplus (\sD_{-2}/\sD_{-1})_x \oplus (\sD_{-1})_x $$ $$ =
(\sD^{(0)}_{-k}/\sD^{(0)}_{-k+1})_y \oplus \cdots \oplus(\sD^{(0)}_{-2}/\sD^{(0)}_{-2})_y \oplus (\sD^{(0)}_{-1}/\sD^{(0)}_0)_y.$$
Then $I_y^{(0)} = (I_y^{(0)})^- + (I^{(0)}_y)^0$ for each $y \in P^{(0)}$ is a graded vector space isomorphism from $\fv^{<1}$ to $\gr(\sD^{(0)}_{\bullet})_y$ whose restriction to $\fv^-$ is a graded Lie algebra isomorphism.  Thus $I^{(0)} := \{ I^{(0)}_y \mid y \in \sP\}$ is a $\beta^{(0)}$-Tanaka parallelism of $(\bar{P}^{(0)}, \sD^{(0)}_{\bullet})$.
    \end{itemize}
     We need to construct bundles $\widetilde{\pi}^1: \widetilde{P}^1 \to \bar{P}^{(0)}$ and $\bar{\pi}^1: \bar{P}^{(1)} \to \bar{P}^{(0)}$ over $\bar{P}^{(0)}$ satisfying the properties (B) and (C) in Theorem \ref{t.Tanaka1}. This requires some preparatory work in Subsections \ref{ss.P1} -- \ref{ss.5.4}.

\subsection{$\fv^0$-exponential action on $ P^1 $}\label{ss.P1}
Associated with the $\beta^{(0)}$-Tanaka parallelism $I^{(0)}$ on $\bar{P}^{(0)}$, we have  from Definition \ref{d.Lifts} (iii) the bundle of lifts of $I^{(0)}$,  $$\lambda^{I^{(0)}} : \bL(I^{(0)}) \to \bar{P}^{(0)},$$ which is
 a principal subbundle of the frame bundle $\F \bar{P}^{(0)},$ with the structure group ${\rm GL}_1(\fv^{<1}) \subset {\rm GL}(\fv^{<1})$.
As mentioned in the subsection \ref{ss.P0},  we have a $\fv^0$-exponential action on $\bar{P}^{(0)} = \sP$ arising from $\zeta^0$.  We can lift it to a $\fv^0$-exponential action  on $\bL(I^{(0)})$, as described below.

\begin{definition}\label{d.G0action} Let us use the following notation. \begin{itemize}
\item[(1)] For each $l \in \fv^0$,
denote by  $R^{\bar{P}^{(0)}}_{\exp(l)}: \bar{P}^{(0)} \to \bar{P}^{(0)}$   the $\fv^0$-exponential action  on $\bar{P}^{(0)}$ arising from the isomorphism $\zeta^0: \fv^0 \times B \cong {\rm Lie}\mathsf{G}^0$ in Set-up \ref{n.B}.
\item[(2)] For any $b \in B$ and $g \in \mathsf{G}^0(b) \subset \mathsf{G}(b)$, let ${\rm Ad}_g \in {\rm GL}(\fv^{<1})$ be the linear automorphism induced by the adjoint action of  $g$ on ${\rm Lie}\mathsf{G}(b) = \fg(b) = \fv.$ \item[(3)] For $l \in \fv^0$ and $b \in B$, let $\exp_b(l) \in \mathsf{G}^0(b)$ be the image of $l$ under the composition of the isomorphism $\zeta^0:\fv^0 \times B \cong {\rm Lie}\mathsf{G}^0$ and the exponential map ${\rm Lie} \mathsf{G}^0(b) \to \mathsf{G}^0(b)$ of the Lie group $\mathsf{G}^0(b)$.
    \end{itemize}
    For any $y \in (\beta^{(0)})^{-1}(b) \subset \bar{P}^{(0)}$, $l \in \fv^0$ and an element $h \in {\rm Isom}(\fv^{<1}, T_y \bar{P}^{(0)}),$   define $$R_{\exp(l)}(h) := {\rm d}_y R^{\bar{P}^{(0)}}_{\exp(l)} \circ  h \circ {\rm Ad}_{\exp_b(l)} \ \in {\rm Isom}(\fv^{<1}, T_z \bar{P}^{(0)}),$$
    where $z := R^{\bar{P}^{(0)}}_{\exp(l)} (y)$.
    The operation $h \mapsto R_{\exp(l)}(h)$ determines a $\fv^0$-exponential action on $\F \bar{P}^{(0)}$, which descends under $\F \bar{P}^{(0)} \to \bar{P}^{(0)}$ to the given $\fv^0$-exponential action on $\bar{P}^{(0)}.$ \end{definition}

    \begin{lemma}\label{l.G0action}
    In Definition \ref{d.G0action},   if $h \in  \bL(I^{(0)}_y),$ then $R_{\exp(l)}(h) \in \F_z \bar{P}^{(0)}$  belongs to $\bL(I^{(0)}_z) \subset \F_z \bar{P}^{(0)}$.
    \end{lemma}

     \begin{proof}
     In the notation of Subsection \ref{ss.P0}, we need to check \begin{itemize} \item[(1)]
    $ {\rm d}_y R^{\bar{P}^{(0)}}_{\exp(l)} \circ h \circ {\rm Ad}_{\exp_b(l)} |_{\fv^0} = I^{(0)}_z |_{\fv^0} \in {\rm Isom}(\fv^0, \sD^{(0)}_z);$ and \item[(2)] ${\rm pr}^i_{(1)} \circ {\rm d}_y R^{\bar{P}^{(0)}}_{\exp(l)} \circ  h \circ {\rm Ad}_{\exp_b(l)} |_{\fv^i} = I^{(0)}_z|_{\fv^i}$ for $-k \leq i \leq -1$.   \end{itemize}

For $ b = \beta^{(0)} (y) \in B$, the fibration  $(\beta^{(0)})^{-1}(b) \to \beta^{-1}(b)$ is a principal bundle with the structure group $\mathsf{G}^0(b).$
Since the equality (1) is  concerned with only vectors tangent to fibers of  $(\beta^{(0)})^{-1}(b) \to \beta^{-1}(b)$,
    (1) follows from the standard result for principal bundles (e.g. \cite{KN} Chapter 1, Proposition 5.1).

For $-k \leq i \leq -1$,  we have $$I^{(0)}_z|_{\fv^i} = I^{(0)}_y \circ {\rm Ad}_{\exp_b(l)}|_{\fv^i}$$ by our definition of $(I^{(0)})^-$ from $\sP_x \subset \I^{<0}_x$.   The natural  identification $(\sD^{(0)}_i/\sD^{(0)}_{i+1})_y = (\sD^{(0)}_i/\sD^{(0)}_{i+1})_z = (\sD_i/\sD_{i+1})_x$ for $x = \pi^{\sP}(y) = \pi^{\sP}(z)$ shows  $${\rm pr}^i_{(1)} \circ {\rm d}_y R^{\bar{P}^{(0)}}_{\exp(l)} \circ h |_{\fv^i} = {\rm pr}^i_{(1)} \circ h|_{\fv^i}.$$   These two equations prove (2) because ${\rm pr}^i_{(1)} \circ h|_{\fv^i} = I^{(0)}_y|_{\fv^i}$ by $h \in \bL(I^{(0)}_y)$. \end{proof}

\begin{definition}\label{d.P1}
Write $P^1:= \bL(I^{(0)})$ and $\pi^1:= \lambda^{I^{(0)}} : P^1 \to \bar{P}^{(0)}$. By Lemma \ref{l.G0action}, we have a    $\fv^0$-exponential action on $P^1$, to be denoted by $R^{P^1}_{\exp (l)}: P^1 \to P^1$ for $l \in \fv^0$,  defined as   $$ \bL(I^{(0)}) \ni h \ \mapsto \  R^{P^1}_{\exp (l)}(h) := R_{\exp(l)}(h) \in \bL(I^{(0)}).$$
\end{definition}

The next lemma is immediate.

\begin{lemma}\label{l.23-1}
In Definition  \ref{d.P1}, the $\fv^0$-exponential action on $P^1$ descends via $\pi^1$ to the $\fv^0$-exponential action on $\bar{P}^{(0)}$ in Definition \ref{d.G0action}. In particular,  for any $l \in \fv^0$, $$ {\rm d} R^{\bar{P}^{(0)}}_{\exp (-l)} \circ {\rm d} \pi^1 \circ {\rm d} R^{P^1}_{\exp(l)} = {\rm d} \pi^1$$ and the fundamental vector field $l^{P^1}$ on $P^1$ and the fundamental vector field $l^{\bar{P}^{(0)}}$ on $ \bar{P}^{(0)}$ are related by ${\rm d} \pi^1(l^{P^1}) = l^{\bar{P}^{(0)}}$. \end{lemma}

Next, we need to describe the behavior of the soldering form on $P^1 \subset \F \bar{P}^{(0)}$ with respect to the $\fv^0$-exponential action  on $P^1$. Since $\mathsf{G}^0$ does not belong to the structure group of $\pi^1: P^1 \to \bar{P}^{(0)}$, Lemma \ref{l.Sternberg} (i) is not directly applicable. However, the following lemma says that a similar result still holds.

\begin{lemma}\label{l.23} In the setting of Definition \ref{d.P1}, let us use the following notation.
\begin{itemize} \item[(1)]
 For each $l \in \fv^0$, the isomorphism $\zeta^0: \fv^0 \times B \to {\rm Lie}\mathsf{G}^0$ induces an automorphism of the trivial vector bundle on $B$ $${\rm Ad}_{\exp (l)}: \fv^{<1} \times B \ \longrightarrow \ \fv^{<1} \times B$$  given by the family $\{ {\rm Ad}_{\exp_b(l)} \in {\rm GL}(\fv^{<1}) \mid  b \in B\}$. It can be lifted to an automorphism    of the trivial vector bundle on $P^1$ $${\rm Ad}_{\exp(l)}: \fv^{<1} \times P^1 \ \longrightarrow \ \fv^{<1} \times P^1.$$  Similarly, we have an  endomorphism $${\rm ad}_{l}: \fv^{<1} \times P^1 \ \longrightarrow \ \fv^{<1} \times P^1.$$
\item[(2)]
Let $\theta^1$ be the $\fv^{<1}$-valued 1-form   on $P^1$, the soldering form from $P^1=\bL(I^{(0)}) \subset \F \bar{P}^{(0)}.$
\end{itemize}
Then  for each $l \in \fv^0$, \begin{equation}\label{e.23-1} (R^{P^1}_{\exp (l)})^* \theta^1 = {\rm Ad}_{\exp (-l)} \circ \theta^1,\end{equation}
and for the fundamental vector field $l^{P^1}$ on $P^1$, \begin{equation}\label{e.23-1-2} {\rm Lie}_{l^{P^1}} \, \theta^1 = - {\rm ad}_{l} \circ \theta^1.\end{equation}
More precisely,  for any $h \in P^1$ and  $v \in T_h P^1$,
 $$ \theta^1({\rm d} R^{P^1}_{\exp(l)}(v)) = {\rm Ad}_{\exp (-l)} ( \theta^1(v)) \mbox{ and } ({\rm Lie}_{ l^{P^1}} \theta^1) (v) = -[ l, \theta^1(v)],$$ where the Lie bracket is that of  the Lie algebra $\fg(b)$ for $b = \beta^{(0)}(\pi^1(h)) \in B$. \end{lemma}

\begin{proof}
Regarding $R^{P^1}_{\exp (l)}(h)$ as an element of ${\rm Isom}(\fv^{<1}, T_{\pi^1(h)} \bar{P}^{(0)}),$ we have
\begin{eqnarray*}
 \theta^1 ({\rm d} R^{P^1}_{\exp (l)} (v))
& = & (R^{P^1}_{\exp (l)}(h))^{-1} ({\rm d} \pi^1 \circ {\rm d} R^{P^1}_{\exp (l)}(v)) \\ & = & (d R^{\bar{P}^{(0)}}_{\exp(l)} \circ h \circ {\rm Ad}_{\exp(l)})^{-1} \circ {\rm d} \pi^1 \circ {\rm d} R^{P^1}_{\exp (l)}(v) \\
&= & {\rm Ad}_{\exp (-l)} \circ h^{-1} \circ {\rm d} R^{\sP}_{\exp (- l)}
\circ {\rm d} \pi^1 \circ {\rm d} R^{P^1}_{\exp (l)}  (v) \\
& = & {\rm Ad}_{\exp (- l)} \circ h^{-1} \circ {\rm d} \pi^1 (v) \\
&=& {\rm Ad}_{\exp (-l)} \circ \theta^1 (v), \end{eqnarray*} where we use Definition \ref{d.G0action} in the second line and Lemma \ref{l.23-1} in the fourth line. This proves (\ref{e.23-1}). Taking derivative of (\ref{e.23-1}), we obtain (\ref{e.23-1-2}). \end{proof}

\subsection{Torsion of $\pi^1: P^1  \to \bar{P}^{(0)}$}

The following is a modification of  Proposition 24 and Theorem 25 of \cite{AD}.

\begin{proposition}\label{p.24}
In Definition \ref{d.P1}, for a point $z \in P^1 \subset \F \bar{P}^{(0)}$ and $b = \beta^{(0)} \circ \pi^1 (z) \in B,$ let $\sH_z \subset T_{z} P^1 \subset T_z \F \bar{P}^{(0)}$ be a $\pi^1$-horizontal subspace. Then the torsion $\tau^{\sH_z} \in \Hom(\wedge^2 \fv^{<1}, \fv^{<1})$ satisfies the following.
\begin{itemize} \item[(i)] The homomorphism $\tau^{\sH_z}$ has only components of nonnegative homogeneous degree with respect to the grading $\fv^{<1} = \oplus_{i=-k}^0 \fv^i$.
\item[(ii)] If $u, v \in \fv^0$, then $\tau^{\sH_z} (u, v) = - [u, v]$ in terms of the Lie algebra $\fg(b)^0$.
    \item[(iii)] If $u, v \in \fv^{<1}$, then $(\tau^{\sH_z}  )^0, $  the homogeneous component of  degree zero in $\tau^{\sH_z},$ is given by  $(\tau^{\sH_z} )^0(u, v) =- [u, v]$ in terms of the Lie algebra $\fg(b).$
\end{itemize}
\end{proposition}

\begin{proof}
Set $y = \pi^1(z) \in \bar{P}^{(0)}$ and write $z = h$ when we view it as an element of $\bL(I^{0}_y) \subset {\rm Isom}(\fv^{<1}, T_y \bar{P}^{(0)}).$:
\begin{eqnarray*}
\xymatrix{
\qquad \qquad \quad P^1=\bL(I^{(0)}) \ni z=h  \ar[d]^{\pi^1}\\
\bar{P}^{(0)} = \sP \ni y \ar[d]_{\beta^{(0)}}\\
 B
}
\end{eqnarray*}
To prove the proposition, we may replace $\bar{P}^{(0)}$ by a neighborhood of $y$ and assume that there exists
a  holomorphic section $\Sigma \subset P^1$ of $\pi^1$ through $z$  such that $T_z \Sigma = \sH_z$.  Translating $\Sigma$ by the right action of  ${\rm GL}_1(\fv^{<1})$ on $P^1$, we obtain a  principal connection $\sH$ on the principal subbundle $\pi^1: P^1 \to \bar{P}^{(0)}.$ Then for each $w \in \fv^{<1}$, we have the $\sH$-horizontal vector field $w^{\sH}$ on $P^1$ from Definition \ref{d.frame} (iii). This is equal to $w^{\sH}$ from Definition \ref{d.torsion} obtained through Example \ref{ex.vectorfield}.

 Let us  introduce a filtration on $P^1$ by  $\sD^{P^1}_i := ({\rm d} \pi^1)^{-1}(\sD^{(0)}_i)$ for $-k \leq i \leq 0.$ This is a $\beta^{(0)} \circ \pi^1$-Tanaka filtration of height $0$ on $P^1$. Since $I^{(0)}$ is a $\beta^{(0)}$-Tanaka parallelism and  $P^1 \subset \bL(I^{(0)})$, we see that $w^{\sH} \in \sD_i^{P^1}$ if $w \in \fv^i$.
 If $w \in \fv^i$ and $u \in \fv^j$ for $i, j  <1$, then $[w^{\sH}, u^{\sH}]$ is a section of $\sD^{P^1}_{i+j}$. Thus  $$\tau^{\sH_z}(w, u) = - h^{-1} ({\rm d} \pi^1 ([w^{\sH}, u^{\sH}]) ) \mbox{ from Lemma \ref{l.Sternberg} (iv)}$$ belongs to $ \oplus_{m = i+j}^0 \fv^{m}$. This proves (i).

To prove (ii) and (iii), we need the following two equalities.
For  $u \in \fv^0$, let $u^{\bar{P}^{(0)}}$ (resp. $u^{P^1}$) be the fundamental vector field on $\bar{P}^{(0)}$ (resp.  $P^1$) given by the $\fv^0$-exponential actions  in Definition \ref{d.P1}. Then
\begin{equation}\label{e.3.4a} {\rm d} \pi^1 (u^{\sH}) = u^{\bar{P}^{(0)}},\end{equation} namely, the vector field $u^{\sH}$ can be projected by $\pi^1$ to $u^{\bar{P}^{(0)}}$, and
\begin{equation}\label{e.3.4b} {\rm d} \pi^1 ( u^{\sH} - u^{P^1}) =0, \end{equation} namely, the vector field $u^{\sH} - u^{P^1}$ on $P^1$ is tangent to fibers of $\pi^1$.
  From   $h \in \bL(I^{(0)}) \subset {\rm Isom}(\fv^{<1}, T_y \bar{P}^{(0)})$ and our definition of $I^{(0)}$ in (A3) of Subsection \ref{ss.P0},  the image $h(u) \in T_y \bar{P}^{(0)}$ is equal to $u^{\bar{P}^{(0)}}_y$. This implies (\ref{e.3.4a}).  By  Lemma \ref{l.23-1}, (\ref{e.3.4a}) implies (\ref{e.3.4b}).

 If $u,v \in \fv^0=\mathfrak g  (b)^0$, then
  \begin{eqnarray*}
  \tau^{\sH_z}(u,v) &=& -h^{-1} \circ {\rm d} \pi^1([u^{\sH}, v^{\sH}]) \text{ by Lemma \ref{l.Sternberg} (iv)} \\
  &=& -h^{-1} ([  u^{\bar{P}^{(0)}},  v^{\bar{P}^{(0)}} ]_y) \text{ by (\ref{e.3.4a})} \\
  &=& -h^{-1} ([u,v]^{\bar{P}^{(0)}}_y) \text{ by Lemma \ref{l.fundamental}}\\
  &=& -[u,v] \text{ by (\ref{e.I0}).}
  \end{eqnarray*} This proves (ii).

For the proof of (iii), first assume  that  $u \in \fv^i$ and $v \in \fv^{j}$ for $i,j <0$.
Then $(\tau^{\sH_z})^0(u,v)$ is the $\fv^{i+j}$-component of   $$\tau^{\sH_z}(w, u) = - h^{-1} ({\rm d} \pi^1 ([w^{\sH}, u^{\sH}]) ).$$ Here, the bracket $[u^{\sH }, v^{\sH}]$ is a section of $\sD^{P^1}_{i+j}$ and  ${\rm d} \pi^1[u^{\sH }, v^{\sH}]_z$ is an element of $(\sD^{ (0) }_{i+j})_y$. Since $h$ is a lift of $I_y^{(0)}$, the   diagram
$$ \begin{array}{ccc} \fv^{i+j} & \stackrel{h|_{\fv^{i+j}}}{\longrightarrow} & (\sD_{i+j}^{ (0) })_y  \\
\| & & \downarrow {\rm pr}^{i+j}_{(1)} \\  \fv^{i+j} & \stackrel{I_y^{(0)}|_{\fv^{i+j}}}{\longrightarrow} & \gr^{i+j} (\sD_{\bullet}^{(0)})_y
\end{array} $$
  is commutative. Thus
  the $\fv^{i+j}$-component of $h^{-1} \circ {\rm d} \pi^1[u^{\sH }, v^{\sH}]_z$ is equal to $(I_y^{(0)})^{-1} \circ {\rm pr}^{i+j}_{(1)} {\rm d} \pi^1[u^{\sH }, v^{\sH}]_z$.
  From ${\rm d}\pi^1 (u^{\sH}_z) = h(u)$ and ${\rm d}\pi^1 (v^{\sH}_z) = h(v),$ it follows that
  \begin{eqnarray*}
  {\rm pr}^{i+j}_{(1)} {\rm d} \pi^1[u^{\sH }, v^{\sH}]_z &=&
  [{\rm pr}^{i }_{(1)}{\rm d}\pi^1 (u^{\sH})_z, {\rm pr}^{j}_{(1)}{\rm d}\pi^1 (v^{\sH})_z] \\
  &=& [{\rm pr}^{i }_{(1)}h(u), {\rm pr}^{j}_{(1)} h(v)] \\
  &=& [I_y^{(0)}(u), I_y^{(0)}(v)],
  \end{eqnarray*}
  where the Lie brackets on the right hand side are from the symbol algebra of $\sD^{(0)}_{\bullet}$.
  Since $I_y^{(0)}|_{\fv^-}$ preserves the Lie brackets, we see that $$(I_y^{(0)})^{-1} ({\rm pr}^{i+j}_{(1)} {\rm d} \pi^1[u^{\sH }, v^{\sH}]_z  ) = [u,v].$$ Hence the $\fv^{i+j}$-component of   $\tau^{\sH_z}(w, u)$ is $-[u,v]$.  We have proved (iii) when $u, v \in \fv^{<0}.$

If $u \in \fv^0$ and $v \in \fv^{j}$ for $j<0$, then in terms of the soldering form $\theta^1$ on $P^1 \subset \F \bar{P}^{(0)}$,
  \begin{eqnarray*}
  \tau^{\sH_z}(u,v) &=& {\rm d}\theta^1(u^{\sH_z}, v^{\sH _z}) \\
  &=& {\rm d}\theta^1( u^{P^1}, v^{\sH})_z + {\rm d}\theta^1( u^{\sH} - u^{P^1}, v^{\sH})_z
    \end{eqnarray*}
The first term ${\rm d}\theta^1( u^{P^1}, v^{\sH})_z$ is  equal to $({\rm Lie}_{u^{P^1}}\theta^1)_z(v^{\sH_z})$ because  of the Cartan formula   $${\rm Lie}_{u^{P^1}}\theta^1 = {\rm d}(\theta^1(u^{P^1})) + {\rm d} \theta^1 (u^{P^1}, \cdot) $$ and the constancy of $\theta^1(u^{P^1}) = \theta^1(u^{\sH}) =u $ from  (\ref{e.3.4b}).  Thus ${\rm d}\theta^1( u^{P^1}, v^{\sH})_z$ is equal to $-{\rm ad}_u(v) = -[u,v]$ by  (\ref{e.23-1-2}). It remains to check that the $\fv^j$-component of ${\rm d}\theta^1( u^{\sH} - u^{P^1}, v^{\sH})_z$ is zero.  By (\ref{e.3.4b}), this is equal to the $\fv^j$-component of $-\theta^1( [u^{\sH}-u^{P^1}, v^{\sH} ]_z  )$.

  Let $\{A_s \mid 1 \leq s \leq d:=\dim \fgl_1(\fv^{<1}) \}$ be a basis of $\mathfrak{gl}_1(\fv^{<1})$. On the ${\rm GL}_1(\fv^{<1})$-principal bundle $\pi^1: P^1 \to \bar{P}^{(0)}$, we have the fundamental vector fields $\{A_s^{P^1} \mid 1 \leq s \leq d\}.$  By (\ref{e.3.4b}), we can write  $$u^{\sH}-u^{P^1}= \sum_{s=1}^d  f_s \, A^{P^1}_s$$ in a neighborhood $U$ of $z \in P^1$,
  where $f_s$ is a holomorphic  function on $U$.  Then
  $$[u^{\sH}-u^{P^1}, v^{\sH}] =[\sum_{s=1}^d f_s \, A_s^{P^1}, v^{\sH}] = \sum_{s=1}^d (f_s[A_s^{P^1}, v^{\sH}] -v^{\sH}(f_s)A_s^{P^1}).$$  Since ${\rm d} \pi^1 (A_s^{P^1}) =0,$
  \begin{eqnarray*}
  \theta^1( [u^{\sH}-u^{P^1}, v^{\sH} ]_z  ) &=& \sum_{s=1}^d f_s(z) \theta^1( [A_s^{P^1}, v^{\sH}])
  \end{eqnarray*}
  By Lemma \ref{l.Sternberg} (iii), we have $[A_s^{P^1}, v^{\sH}]= (A_s(v))^{\sH}$. Since $A_s \in \fgl_1(\fv^{<1})$,  the $\fv^j$-component of $\theta^1( [A_s^{P^1}, v^{\sH}])$ must be zero. This completes the proof of (iii).
\end{proof}

\subsection{Variation of the torsion of $\pi^1$}\label{ss.5.4}
The following is a modified version of Proposition 26 of \cite{AD}.
\begin{proposition}\label{p.26} For $z \in P^1 \stackrel{\pi^1}{\to} \bar{P}^{(0)}$, let $\sH_z$ and $\sH'_z$  be any two $\pi^1$-horizontal subspaces of $T_z P^1$.
     For $u \in \fv^i$ and $v \in \fv^j$ with $i, j <0$, the components of degree $i+j +1$ of  $\tau^{\sH_z}(u, v)$ and $\tau^{\sH'_z}(u,v)$ coincide:
        $$ \tau^{\sH_z}(u, v)^{i+j+1} = \tau^{\sH'_z}(u,v)^{i+j+1}.$$
\end{proposition}

\begin{proof}
By Lemma \ref{l.Sternberg} (ii), there are $a, b \in \fgl_1(\fv^{<1})$ such that
\begin{equation}\label{e.p26} \tau^{\sH'_z}(u,v) - \tau^{\sH_z}(u,v) = -a(v) + b(u). \end{equation}
Since  $u \in \fv^i$ and $v \in \fv^j$, components of $a(v)$ have degrees at least $j+1$ and components of $b(u)$ have degrees at least $i+1$. From $i+j+1 <i+1, j+1$, we see that $-a(v)+ b(u)$ has no components of degree $i+j+1$. Thus (\ref{e.p26}) implies $\tau^{\sH_z}(u, v)^{i+j+1} = \tau^{\sH'_z}(u,v)^{i+j+1}.$ \end{proof}

\begin{definition}\label{d.27}
Recall ${\rm Tor}^1(\fv) = \oplus_{i, j <0} \Hom (\fv^i \wedge \fv^j, \fv^{i+j +1}).$
 For $z \in P^1$ and $u \in \fv^i, v \in \fv^j$ with $i, j <0$, define $\tau^1(z) \in {\rm Tor}^1(\fv)$ by
$$\tau^1 (z) (u, v):= \tau^{\sH_z}(u, v)^{i+j +1} \ \in \fv^{i+j+1},$$ where $\sH_z \subset T_z P^1$ is any $\pi^1$-horizontal subspace. This is independent of the choice of $\sH_z$ by Proposition \ref{p.26}. Then $\tau^1$ is a  ${\rm Tor}^1(\fv)$-valued holomorphic function on $ P^1.$  \end{definition}

\begin{proposition}\label{p.29}
For  $A \in \fgl_1(\fv^{<1})$ and $a:= {\rm Id}_{\fv^{<1}} + A \in {\rm GL}_1(\fv^{<1})$, let $R_a: P^1 \to P^1$ be the right action of $a$ on the principal bundle $\pi^1: P^1 \to \bar{P}^{(0)}$ with the structure group ${\rm GL}_1(\fv^{<1}).$ Then
the function $\tau^1$ in Definition \ref{d.27} satisfies $\tau^1_{R_a z} = \tau^1_z + \p^1 ( A \times b)$ for any $z \in P^1$ and $b= \beta^{(0)} \circ \pi^1(z)$,
in terms of  the vector bundle homomorphism  $\p^1$ in Set-up \ref{n.B} (6).
 \end{proposition}

\begin{proof} To simplify the notation, let us write $h(u) = hu$ for $h \in \End(\fv^{<1})$ and $u \in \fv^{<1}$ in this proof.  Let ${\rm Id}_{\fv^{<1}} + \check{A} \in {\rm GL}_1(\fv^{<1})$ be the inverse of $a= {\rm Id}_{\fv^{<1}} + A$. Then $\check{A}^1 = -A^1$ from Lemma \ref{l.barG}.
By Lemma \ref{l.Sternberg} (v), we have $\tau^{\sH_{R_az}}(u,v) = a^{-1} \tau^{\sH_z}(au, av)$.  Thus
\begin{eqnarray*}
\tau^{\sH_{R_az}}(u,v) &=& \tau^{\sH_{z }}(u,v) + \tau^{\sH_{z }}(u,Av) + \tau^{\sH_{z }}(Au,v) + \tau^{\sH_{z }}(Au,Av)  \\
& & \check{A} (\tau^{\sH_{z }}(u,v) + \tau^{\sH_{z }}(u,Av) + \tau^{\sH_{z }}(Au,v) + \tau^{\sH_{z }}(Au,Av)).
\end{eqnarray*} From Proposition \ref{p.24} (i),
when $u \in \fv^i$ and $v \in \fv^j$ with $i,j <0$, the $\fv^{i+j+1}$-component of $\tau^{\sH_{R_a z}}(u,v)$ is given by
\begin{eqnarray*}
\tau^{\sH_{z }}(u,v)^{i+j+1}   + (\tau^{\sH_{z }})^0(u,A^1v) + (\tau^{\sH_{z }})^0(A^1u,v) +\check{A}^1(\tau^{\sH_{y }})^0(u,v) \\
= \tau^{\sH_{z }}(u,v)^{i+j+1}   + (\tau^{\sH_{z  }})^0(u,A^1v) + (\tau^{\sH_{z }})^0(A^1u,v) -A^1(\tau^{\sH_{z }})^0(u,v),
\end{eqnarray*}
which is equal to
$$(\tau^{\sH_{z }})^1(u,v)   - [u,A^1v] - [A^1u,v] +A^1[u,v] $$  by Proposition \ref{p.24} (iii).
Since   for any $A \in \fgl_1(\fv^{<1})$,  $u, v \in \fv^-$ and  $b \in B$, $$\p_b^1 A(u,v) = A^1[u,v] - [A^1u, v] -[u, A^1v],$$  where the Lie bracket is that of $\fg(b),$
we have proved $\tau^1_{R_a z} = \tau^1_z + \p^1_b A.$
\end{proof}

\begin{proposition}\label{p.30}
Let us regard $\tau^1$ as a holomorphic map between fiber bundles $$\begin{array}{ccc}
P^1 & \stackrel{\tau^1}{\longrightarrow} &   {\rm Tor}^1(\fv) \times B \\ \pi^1 \downarrow & & \downarrow \\ \bar{P}^{(0)} & \stackrel{\beta^{(0)}}{\longrightarrow} & B.\end{array} $$
 Let $\sW^1 \subset {\rm Tor}^1(\fv) \times B$ be the vector subbundle from Set-up \ref{n.B} (7) satisfying ${\rm Tor}^1(\fv) \times B = {\rm Im}(\p^1) \oplus \sW^1.$ Define  $\widetilde{P}^1 := (\tau^1)^{-1}(\sW^1)$ and the natural projection $$\widetilde{\pi}^1 := \pi^1|_{\widetilde{P}^1}: \widetilde{P}^1  \longrightarrow \bar{P}^{(0)}.$$ Then \begin{itemize} \item[(i)] $\widetilde{\pi}^1$  is a $\beta^{(0)}$-principal subbundle of $\pi^1: P^1 \to \bar{P}^{(0)}$ with the structure $B$-group $\mathsf{G}^1 {\rm GL}_2(\fv^{<1}) \subset {\rm GL}_1(\fv^{<1})$; and \item[(ii)] for any $\widetilde{\pi}^1$-horizontal $\sH_z \subset T_z \widetilde{P}^1$ at a point $z \in \widetilde{P}^1$, its torsion
$\tau^{\sH_z}(u, v) \in \Hom(\wedge^2 \fv^{<1},\fv^{<1})$ has only components of nonnegative homogenous degree and satisfies
$ (\tau^{\sH_z} )^0(u,v)  = - [u, v] $ for $u \in \fv^{-1}$ and $v \in \fv^{<1}$.\end{itemize} \end{proposition}

 \begin{proof}
 For the fiber $P^1_y:=(\pi^1)^{-1}(y) $ over a point $y \in \bar{P}^{(0)}$ with  $b = \beta^{(0)}(y)$, Proposition \ref{p.29} says that  $\tau^1(P^1_y)$  is an affine subspace of ${\rm Tor}^1(\fv)$, a  translate of the vector subspace $${\rm Im}(\p^1_{\fg(b)}) \ \subset
  \ {\rm Tor}^1(\fv) = {\rm Im}(\p^1_{\fg(b)}) \oplus \sW^1_b.$$ Thus the intersection  $\tau^1(P^1_y) \cap \sW^1_b$ is a single point and  $$\widetilde{P}^1_y := (\widetilde{\pi}^1)^{-1}(y) \ = \  (\tau^1)^{-1}(\tau^1(P^1_y) \cap \sW^1_b).$$
By Lemma \ref{l.partial} and Proposition \ref{p.29}, the subgroup of ${\rm GL}_1(\fv^{<1})$ with Lie algebra $\fg(b)^1 + \fgl_2(\fv^{<1})$ acts simply transitively on the fiber $\widetilde{P}^1_y$.  This proves (i).

Regarding $\sH_z$ as a $\pi^1$-horizontal subspace of $T_z P^1$ and applying  Proposition \ref{p.24}, we obtain (ii).
 \end{proof}

\begin{proof}[Proof of  Theorem \ref{t.Tanaka1} in the case $n=1$] We have checked in Subsection \ref{ss.P0} that the base $\bar{P}^{(0)}$ satisfies condition (A).  The $\beta^{(0)}$-principal bundle $\widetilde{\pi}^1: \widetilde{P}^1 \rightarrow \bar{P}^{(0)}$ and its quotient $\bar{\pi}^1: \bar{P}^{(1)}:=\widetilde{P}^1/{\rm GL}_2(\fv^{<1}) \rightarrow \bar{P}^{(0)}$ satisfies conditions in (B) and (C) for $n=1$ by Proposition \ref{p.30}.
\end{proof}

\section{Generalized Tanaka prolongation: Proof of inductive steps}\label{s.ProofInduction}
Throughout this section, fix a positive integer $n \geq 1$ and assume that
 we have $\bar{\pi}^n: \bar{P}^{(n)} \to \bar{P}^{(n-1)}$ satisfying conditions (A), (B), and (C)   in Theorem \ref{t.Tanaka1}.
The goal of this section is to construct $\bar{\pi}^{n+1}: \bar{P}^{(n+1)} \rightarrow \bar{P}^{(n)}$ satisfying conditions (A), (B), and (C) with $n$ replaced by $n+1$.

\subsection{Definition of $(\bar{P}^{(n)}, \sD_{\bullet}^{(n)}, I^{(n)})$} \label{ss.Pn}
Since $\bar{\pi}^n: \bar{P}^{(n)} \to \bar{P}^{(n-1)}$ is a $B$-principal bundle with the
structure $B$-group $\mathsf{G}^n$, the isomorphism $\zeta^n: \fv^n \times B \to {\rm Lie}\mathsf{G}^n$ of vector bundles on $B$ from Set-up \ref{n.B} gives rise to a $\fv^n$-exponential action on $\bar{P}^{(n)}$ and the fundamental vector field $l^{\bar{P}^{(n)}}$ on $\bar{P}^{(n)}$ for each $l \in \fv^n$.

We  check the condition (A) for $\bar{P}^{(n)}$ as follows.
\begin{itemize} \item[(A1)] Define $\beta^{(n)}: \bar{P}^{(n)} \to B$ as the composition $$\bar{P}^{(n)}  \stackrel{\bar{\pi}^{n}}{\longrightarrow} \bar{P}^{(n-1)} \stackrel{\beta^{(n-1)}}{\longrightarrow} B.$$
\item[(A2)] Applying Lemma \ref{l.vertical} (ii) to the $\beta^{(n-1)}$-Tanaka filtration $\sD^{(n-1)}_{\bullet}$ of height $n-1$ on $\bar{P}^{(n-1)}$, define the $\beta^{(n)}$-Tanaka filtration $\sD^{(n)}_{\bullet}$ of height $n$ on $\bar{P}^{(n)}$ as the $\bar{\pi}^{n}$-lift $(\sD^{(n-1)}_{\bullet})^{\bar{\pi}^n}$.
    \item[(A3)] For $y \in \bar{P}^{(n)}$, write $x = \bar{\pi}^n(y)$ and $b = \beta^{(n)}(y)$.
     Since $(\beta^{(n)})^{-1}(b) \to (\beta^{(n-1)})^{-1}(b)$ is a principal bundle with the structure group $\mathsf{G}^n(b)$, we have a natural isomorphism \begin{equation}\label{e.In} \fv^n= \fg(b)^n \to {\rm Ker}({\rm d}_y \bar{\pi}^{n}) = (\sD^{(n)}_n)_y, \end{equation} which sends $l \in \fv^n$ to $l^{\bar{P}^{(n)}}_y$, the value of  the fundamental vector field at $y$. Denote by $(I^{(n)}_{y})^n: \fv^n \to (\sD^{(n)}_n)_y$ the isomorphism (\ref{e.In}).
     Furthermore, the $\beta^{(n-1)}$-Tanaka parallelism $I^{(n-1)}_x$  determines a graded vector space isomorphism $(I^{(n)}_y)^{<n}$ from $\fv^{<n}$ to $$ (\sD^{(n-1)}_{-k}/\sD^{(n-1)}_{-k+1})_x \oplus \cdots \oplus (\sD^{(n-1)}_{n-2}/\sD^{(n-1)}_{n-1})_x \oplus (\sD^{(n-1)}_{n-1})_x $$ $$ =
(\sD^{(n)}_{-k}/\sD^{(n)}_{-k+1})_y \oplus \cdots \oplus(\sD^{(n)}_{n-2}/\sD^{(n)}_{-2})_y \oplus (\sD^{(n)}_{n-1}/\sD^{(n)}_n)_y.$$
Then $I_y^{(n)} = (I^{(n)}_y)^{<n} + (I^{(n)}_y)^n$ for each $y \in P^{(n)}$ is a graded vector space isomorphism from $\fv^{<n+1}$ to $\gr(\sD^{(n)}_{\bullet})_y$ whose restriction to $\fv^-$ is a graded Lie algebra isomorphism.  Thus $I^{(n)} := \{ I^{(n)}_y \mid y \in \bar{P}^{(n)}\}$ is a $\beta^{(n)}$-Tanaka parallelism of $(\bar{P}^{(n)}, \sD^{(n)}_{\bullet})$.
    \end{itemize}
     We need to construct bundles $\widetilde{\pi}^{n+1}: \widetilde{P}^{n+1} \to \bar{P}^{(n)}$ and $\bar{\pi}^{n+1}: \bar{P}^{(n+1)} \to \bar{P}^{(n)}$ over $\bar{P}^{(n)}$ satisfying the properties (B) and (C) in Theorem \ref{t.Tanaka1}. As in Section \ref{s.n=1}, this requires  some preparatory work in Subsections \ref{ss.Pn+1} -- \ref{ss.4.6}.

\subsection{The principal bundle $\pi^{n+1}:P^{n+1} \to \bar{P}^{(n)}$}\label{ss.Pn+1}
We need to define an auxiliary principal bundle $\pi^{n+1}:P^{n+1} \to \bar{P}^{(n)}$ with the structure  group $H(\fv^{<n+1})$ from Definition \ref{d.barG}.

 \begin{definition}\label{d.pistar}
From the definitions of $\sD_{\bullet}^{(n)}$ and  $I^{(n)}$ in Subsection \ref{ss.Pn}, we see that for an element $h \in \bL(I^{(n)}_y) \subset {\rm Isom}(\fv^{<n+1}, T_y \bar{P}^{(n)})$ with $y \in \bar{P}^{(n)}$ and $x = \bar{\pi}^{(n)}(y) \in \bar{P}^{(n-1)},$ the composition $ {\rm d} \bar{\pi}^n \circ h|_{\fv^{<n}} \in {\rm Isom}(\fv^{<n}, T_x \bar{P}^{(n-1)})$ belongs to $\bL(I^{(n-1)}_x).$ Thus
 $$ \bar{\pi}^n_*(h):= {\rm d} \bar{\pi}^n \circ h|_{\fv^{<n}} \in \bL(I^{(n-1)}_x)$$ defines a holomorphic map $\bar{\pi}^n_*: \bL(I^{(n)}) \to \bL(I^{(n-1)}),$ satisfying the commutative diagram
  $$\begin{array}{ccc} \bL(I^{(n)}) & \stackrel{\bar{\pi}^n_*}{\longrightarrow} & \bL(I^{(n-1)}) \\ \lambda^{I^{(n)}} \downarrow & & \downarrow \lambda^{I^{(n-1)}} \\
  \bar{P}^{(n)} & \stackrel{\bar{\pi}^n}{\longrightarrow} & \bar{P}^{(n-1)}\end{array} $$
  such that the fibers of $\bar{\pi}^n_*$ are affine spaces isomorphic to $\Hom(\fv^{<n}, \fv^n).$\end{definition}

  \begin{proposition} \label{p.41}
Consider the two holomorphic maps
$$\bL(I^{(n)}) \stackrel{\bar{\pi}_*^n}{\longrightarrow} \bL(I^{(n-1)}) \stackrel{{\rm pr}^{I^{(n-1)}}_{(n+1)}}{\longrightarrow} \bL_{(n+1)}(I^{(n-1)}),$$ where the second map is from Definition \ref{d.Lifts} (iv). Then their composition $\bL(I^{(n)}) \to \bL_{(n+1)}(I^{(n-1)})$ is a principal bundle with the structure group $H(\fv^{< n+1})$ in Definition \ref{d.barG}.
\end{proposition}

\begin{proof}
As $\lambda^{I^{(n)}}: \bL(I^{(n)}) \to \bar{P}^{(n)}$ is a principal bundle with the structure group ${\rm GL}_1(\fv^{<n}),$ there is a natural right action of the subgroup $$H(\fv^{<n+1}) = {\rm Id}_{\fv^{<n+1}} + \fgl_{n+1}(\fv^{<n}) + \Hom (\fv^{<n}, \fv^n) \ \subset {\rm GL}_1(\fv^{<n+1})$$ on $\bL(I^{(n)}).$ Thus it suffices to show that $H(\fv^{<n+1})$ acts transitively on the fibers of $\bL(I^{(n)}) \to \bL_{(n+1)}(I^{(n-1)})$.
But this is clear because a fiber of $\bar{\pi}^n_*$
is isomorphic to $\Hom(\fv^{<n}, \fv^n)$ and the map ${\rm pr}^{I^{(n-1)}}_{(n+1)}$ is the quotient by the normal subgroup ${\rm GL}_{n+1}(\fv^{<n}) \subset {\rm GL}_1(\fv^{<n})$ from Lemma \ref{l.lift} (ii).
\end{proof}

\begin{definition}\label{d.Pn+1}
By the condition (B3) of Theorem \ref{t.Tanaka1},
we have a natural embedding $\xi^{(n)}: \bar{P}^{(n)} \subset \bL_{(n+1)}(I^{(n-1)}).$
Let $\pi^{n+1}: P^{n+1} \to \bar{P}^{(n)}$ be the fiber product:
$$ \begin{array}{ccc} P^{n+1} & \longrightarrow & \bL(I^{(n)}) \\ \pi^{n+1} \downarrow & & \downarrow {\rm pr}^{I^{(n-1)}}_{(n+1)} \circ \bar{\pi}^n_*\\ \bar{P}^{(n)} & \stackrel{\xi^{(n)}}{\longrightarrow} & \bL_{(n+1)}(I^{(n-1)}). \end{array} $$
Then $\pi^{n+1}: P^{n+1} \to \bar{P}^{(n)}$ is a principal fiber bundle with the structure group $H(\fv^{<n+1})$ by Proposition \ref{p.41}. \end{definition}

\subsection{$(\fv^n \oplus\fgl_{n+1}(\fv^{<n}))$-exponential action on $P^{n+1}$}

The next definition is analogous to Definition \ref{d.G0action}.

\begin{definition} \label{d.42} Let us use the following notation. \begin{itemize}
\item[(1)] For  $l \in \fv^n \oplus \fgl_{n+1}(\fv^{<n})$, let $l' \in \fv^n$ be its $\fv^n$-component. The isomorphism   $\zeta^n: \fv^n \times B \to {\rm Lie}\mathsf{G}^n$ from Set-up \ref{n.B}  induces a $\fv^n$-exponential action $$\{ R^{\bar{P}^{(n)}}_{\exp (l')}: \bar{P}^{(n)} \to \bar{P}^{(n)} \mid  l' \in \fv^n \}$$ on the $\beta^{(n-1)}$-principal bundle  $\bar{\pi}^n: \bar{P}^{(n)} \to \bar{P}^{(n-1)}$ with the structure $B$-group $\mathsf{G}^n.$
    \item[(2)] For each $l \in \fv^n \oplus \fgl_{n+1}(\fv^{<n})$ and $b \in B$, define
     an element $\rho (\exp_b(l)) \in {\rm GL}(\fv^{<n+1})$ via
     the composition  of the isomorphism from Set-up \ref{n.B} $$ \eta^n: (\fv^n \oplus \fgl_{n+1}(\fv^{<n})) \times B \ \to \ {\rm Lie} \mathsf{G}^n {\rm GL}_{n+1}(\fv^{<n})$$  and the exponential map  of the Lie group for each $b \in B$, $$   \fg^n(b) \oplus \fgl_{n+1}(\fv^{<n}) \stackrel{\exp_b}{\longrightarrow} \mathsf{G}^n(b) {\rm GL}_{n+1}(\fv^{<n}) \subset {\rm GL}_n(\fv^{<n}) \subset {\rm GL}(\fv^{<n+1}),$$ where the last inclusion is from  Lemma \ref{l.inclusion}.
 \end{itemize}
 Let us define a $(\fv^n \oplus \fgl_{n+1}(\fv^{<n}))$-exponential action on the frame bundle $\F \bar{P}^{(n)}$ in the following way.
For  $$ y \in \bar{P}^{(n)}, z = R^{\bar{P}^{(n)}}_{\exp(l')}(y) \in \bar{P}^{(n)}, b = \beta^{(n)}(y), h  \in {\rm Isom}(\fv^{<n+1}, T_y\bar{P}^{(n)}),$$ and $l \in \fv^n \oplus \fgl_{n+1}(\fv^{<n})$,   define $R_{\exp(l)}(h) \in {\rm Isom}(\fv^{<n+1}, T_z\bar{P}^{(n)})  $ by
$$R_{\exp(l)}(h):= {\rm d}_y R^{\bar{P}^{(n)}}_{\exp(l')} \circ h \circ \rho(\exp_b(l)) \ \in {\rm Isom}(\fv^{<n+1}, T_{z} \bar{P}^{(n)}).$$
It is easy to modify the proof of  Lemma \ref{l.G0action} to see that if $h \in \bL(I^{(n)}_y),$ then $R_{\exp(l)}(h) \in
\bL(I^{(n)}_{z}).$  The collection $\{R_{\exp(l)} \mid l \in \fv^n \oplus \fgl_{n+1}(\fv^{<n})\}$ defines a $(\fv^n \oplus \fgl_{n+1}(\fv^{<n}))$-exponential action on $\bL(I^{(n)})$.
\end{definition}

The following lemma corresponds to Proposition 37 of \cite{AD}.

\begin{lemma} \label{l.42}
The $(\fv^n \oplus \fgl_{n+1}(\fv^{<n}))$-exponential action on $\bL(I^{(n)})$ in Definition \ref{d.42} preserves $P^{n+1} \subset \bL(I^{(n)})$, inducing a  $(\fv^n \oplus \fgl_{n+1}(\fv^{<n}))$-exponential action $\{R^{P^{n+1}}_{\exp(l)} \mid l \in \fv^n \oplus \fgl_{n+1}(\fv^{<n})\}$ on $P^{n+1}$, which satisfies
\begin{equation}\label{e.420} R_{\exp(l')}^{ \bar{P}^{(n)}} \circ \pi^{n+1} = \pi^{n+1} \circ R_{\exp(l)}^{ P^{n+1} }\end{equation} for any $l \in \fv^n \oplus \fgl_{n+1}(\fv^{<n}).$
\end{lemma}

\begin{proof}
Using the notation in Definition \ref{d.Pn+1}, define for any $y \in \bar{P}^{(n)}$ with $x := \bar{\pi}^n(y) \in \bar{P}^{(n-1)}$,  $$ J_y := \xi^{(n)}(y) \in \Hom(\fv^{<n}, \gr_{(n+1)}(\sD^{(n-1)}_{\bullet})_x).$$
Then $h \in \bL(I^{(n)})$ belongs to $P^{n+1}$ if and only if for $y = \lambda^{I^{(n)}}(h), x = \bar{\pi}^n(y)$ and  any $v^i \in \fv^i, i<n$,
\begin{equation}\label{e.421} {\rm d} \bar{\pi}^n (h (v^i)) \equiv J_y (v^i) \mod (\sD^{(n-1)}_{i + n+1})_x.\end{equation}
To prove the lemma, we need to check that if $h \in P^{n+1}$, then
 $R_{\exp(l)}(h) \in P^{n+1}$  for any $l \in \fv^n \oplus \fgl_{n+1}(\fv^{<n}).$
Setting $ z = R^{\bar{P}^{(n)}}_{\exp(l')}(y)$ with $x = \bar{\pi}^n(y) = \bar{\pi}^n(z),$ we need to check
\begin{equation}\label{e.422}
{\rm d} \bar{\pi}^n  \circ {\rm d}_y R^{\bar{P}^{(n)}}_{\exp(l')} \circ h \circ \rho(\exp_b(l)) (v^i) \equiv J_z (v^i) \mod (\sD^{(n-1)}_{i + n+1})_x.\end{equation}
For $b = \beta^{(n)}(y)$, we can write $\exp_b(l) = {\rm Id}_{\fv^{<n}} + A^n + A_{n+1}$ for some $A^n \in \fg(b)^n$ and $A_{n+1} \in \fgl_{n+1}(\fv^{<n})$.
By Lemma \ref{l.lift} (iii), (\ref{e.421}) and  ${\rm d} \bar{\pi}^n \circ {\rm d}_y R^{\bar{P}^{(n)}}_{\exp(l')} = {\rm d} \bar{\pi}^n$,  the left hand side of (\ref{e.422}) is \begin{equation}\label{e.423} {\rm d} \bar{\pi}^n  \circ h (v^i + A^n v^i + A_{n+1}v^i) \equiv J_y(v^i) +  \overline{J_y (A^n v^i)} \end{equation} modulo $(\sD^{(n-1)}_{i  + n+1})_x.$
Note that from the condition (B3) in Theorem \ref{t.Tanaka1}, $$J_{z} (v^i) = J_y \circ \exp_b(l') (v^i).$$
Thus the right hand side of (\ref{e.422}) is
  $$J_z (v^i)= J_y (\exp_b(l') (v^i)) = J_y(v^i) +  \overline{J_y( A^nv^i)},$$ which is equal to (\ref{e.423}), proving (\ref{e.422}).  The proof of  (\ref{e.420}) is immediate.
    \end{proof}

\begin{lemma} \label{l.43} Let $\theta^{n+1}$ denote the restriction of the soldering form of $\mathbb F\bar{P}^{(n)}$ to $P^{n+1} \subset \bL(I^{(n)})$. Then $\theta^{n+1}$ is equivariant with respect to the $(\fv^n \oplus \fgl_{n+1}(\fv^{<n }))$-exponential action in Lemma \ref{l.42}, in the sense that  for any $l \in \fv^n \oplus \fgl_{n+1}(\fv^{<n })$,
\begin{equation}\label{e.431}
(R^{P^{n+1}}_{\exp(l)})^*\theta^{n+1} = \rho(\exp(l))^{-1}\circ \theta^{n+1},
\end{equation} where $\rho(\exp(l))$ is the automorphism of the trivial vector bundle $\fv^{<n+1} \times B$ defined by $\rho(\exp_b(l)) \in {\rm GL}(\fv^{<n+1})$ for $b \in B$ from Definition \ref{d.42} (2).
If $l^{P^{n+1}}$ is the fundamental  vector field corresponding to $l$,  then
\begin{equation}\label{e.432}
{\rm Lie}_{l^{P^{n+1}}}\theta^{n+1} = - \rho(l) \circ \theta^{n+1},
\end{equation}
where $\rho(l)$ is the endomorphism of the trivial vector bundle $\fv^{<n+1} \times B$ whose value at $b \in B$ comes from the inclusions $$\fg^n(b) + \fgl_{n+1}(\fv^{<n})
\subset \fgl_n(\fv^{<n}) \subset \fgl(\fv^{<n+1}),$$ the first inclusion by (\ref{e.prolong}) and the second inclusion by Lemma \ref{l.inclusion}.
\end{lemma}

\begin{proof}
Regarding a point $h \in P^{n+1}$ as an element of ${\rm Isom}(\fv^{<n+1}, T_y(\bar{P}^{(n)}))$ with $ y = \pi^{n+1}(h),$  we have $\theta^{n+1}(v) = h^{-1}({\rm d} \pi^{n+1}(v))$ for any $ v \in T_{h}P^{n+1}.  $
Thus
\begin{eqnarray*}
\lefteqn{(R^{P^{n+1}}_{\exp(l)})^*\theta^{n+1}(v) } \\ &=& \theta^{n+1}({\rm d} R^{P^{n+1}}_{\exp(l)}(v)) \\
&=& (R^{P^{n+1}}_{\exp(l)}(h))^{-1}({\rm d}\pi^{n+1} \circ {\rm d} R^{P^{n+1}}_{\exp(l)} (v)) \\
&=& (\rho(\exp(l)))^{-1} \circ h^{-1} \circ ({\rm d}_y R^{\bar{P}^{(n)}}_{\exp(l')})^{-1} \circ {\rm d}\pi^{n+1} \circ {\rm d}R^{P^{n+1}}_{\exp(l)} (v)  \\
&=& (\rho(\exp(l)))^{-1} \circ h^{-1} \circ {\rm d} \pi^{n+1}(v) \text{ by (\ref{e.420})} \\
&=& (\rho(\exp(l)))^{-1} \circ \theta^{n+1}(v).
\end{eqnarray*} This proves (\ref{e.431}). By taking derivative, we obtain (\ref{e.432}).
\end{proof}

\subsection{Auxiliary vector fields on $\bar{P}^{(n)}$}

From the inductive assumption (B), we have $\widetilde{\pi}^{n}: \widetilde{P}^{n} \rightarrow  \bar{P}^{(n-1)},$ a  $\beta^{(n-1)}$-principal subbundle of $\bL(I^{(n-1)}) \subset \mathbb F \bar{P}^{(n-1)} \rightarrow \bar{P}^{(n-1)}$ with the structure $B$-group $\mathsf{G}^n{\rm GL}_{n+1}(\fv^{<n})$ such that  the $\beta^{(n-1)}$-principal bundle $\bar{P}^{(n)} \rightarrow \bar{P}^{(n-1)}$ with the structure $B$-group $\mathsf{G}^n$  is obtained from $\widetilde{\pi}^n$ by taking the quotient by the normal subgroup ${\rm GL}_{n+1}(\fv^{<n})$.

\begin{definition}\label{d.hatu}
Assume that we have a section $\w{\Sigma} \subset \widetilde{P}^n$ of $\widetilde{\pi}^n$. From Definition \ref{d.connection}, the isomorphism $\eta^n: (\fv^n \oplus \fgl_{n+1}(\fv^{<n}))\times B \to {\rm Lie}\mathsf{G}^n {\rm GL}_{n+1}(\fv^{<n})$ and the section $\w{\Sigma}$ determine a $\widetilde{\pi}^n$-connection  $\widetilde{\sH} \subset T \widetilde{U}$ in a neighborhood $\widetilde{U} \subset \widetilde{P}^n$ of $\w{\Sigma}$: $$ \begin{array}{cccc} \widetilde{U}\subset & \widetilde{P}^n & \subset & \F \bar{P}^{(n-1)} \\ &  \widetilde{\pi}^n \downarrow & & \downarrow \\ & \bar{P}^{(n-1)} & = & \bar{P}^{(n-1)}.\end{array} $$
Moreover, we have the $\widetilde{\sH}$-horizontal vector field $w^{\widetilde{\sH}}$ on $\widetilde{U}$ corresponding to each $w \in \fv^{<n}$.
For $w \in \oplus_{j=-1}^{n-1} \fv^{j}$ and $A \in {\rm GL}_{n+1}(\fv^{<n})$,  we have $A(w) =w$. Thus by Lemma \ref{l.Sternberg} (iii), the $\widetilde{\sH}$-horizontal vector field $w^{\widetilde{\sH} }$ on $\widetilde{U}$ is ${\rm GL}_{n+1}(\fv^{<n})$-invariant and descends to a vector field on $\bar{P}^{(n)}$, which we denote by $\widehat{w}$. \end{definition}

The following lemma corresponds to Lemma 41 (i) and (ii) of \cite{AD}.

\begin{lemma} \label{l.56} As $\bar{\pi}^{(n)}: \bar{P}^{(n)} \to \bar{P}^{(n-1)}$ is  a $\beta^{(n-1)}$-principal bundle with the structure $B$-group $\mathsf{G}^n$,  the isomorphism  $\zeta^n: \fv^n \times B \cong {\rm Lie} \mathsf{G}^n$ determines  the fundamental vector field $A^{\bar{P}^{(n)}}$ on $\bar{P}^{(n)}$ for  $A \in \fv^n$.  Let $u \in \fv^{-1}$ and $v \in \fv^{i}$ with $0 \leq i \leq n-1$. Since $A(u) \in \fv^{n-1}$, we have vector fields $\widehat{u}, \widehat{v}$ and $\widehat{A(u)}$ on $\bar{P}^{(n)}$ from Definition \ref{d.hatu}. Then \begin{equation} \label{e.54}
    [A^{\bar{P}^{(n)}}, \widehat{u}] =\widehat{A(u)}, \qquad [A^{\bar{P}^{(n)}}, \widehat{v}] =0
    \end{equation}
    and
\begin{equation} \label{e.53}
[\widehat{u}, \widehat{v}] = \widehat{[u,v]} \mod \sD_i^{(n)}.
\end{equation}
\end{lemma}

\begin{proof}
By the $(\fv^n \oplus \fgl_{n+1}(\fv^{<n}))$-exponential action on $\widetilde{P}^n$ associated with the isomorphism $$\eta^n: (\fv^n \oplus \fgl_{n+1}(\fv^{<n})) \times B \ \longrightarrow \ {\rm Lie}\mathsf{G}^n {\rm GL}_{n+1}(\fv^{<n}),$$ we have the fundamental vector field $A^{\widetilde{P}^n}$ on $\widetilde{P}^n$ for each $ A \in \fv^n$ such that for $w \in \oplus_{j=-1}^{n-1} \fv^{j},$
  $$ [A^{\widetilde{P}^n}, w^{\widetilde{\sH}}] = (A(w))^{\widetilde{\sH}}$$   by Lemma \ref{l.Sternberg} (iii). By projecting this equality to  $\bar{P}^{(n)},$ we obtain  $[A^{\bar{P}^{(n)}}, \widehat{w} ] = \widehat{A(w) }, $ proving (\ref{e.54}).

  By the inductive assumption (C) in Theorem \ref{t.Tanaka1}, for $h \in \widetilde{U} \subset \widetilde{P}^n$,
     $$\tau^{\widetilde{\sH}_{h}}(u ,v) \in \oplus_{j= i-1}^{n-1} \fv^j \ \text{ and } \
     (\tau^{\widetilde{\sH}_{h}})^0(u,v) = -[u,v]$$
    in terms of the  Lie algebra $\mathfrak g(b)$ with $b=\beta^{(n-1)} \circ \widetilde{\pi}^{n }(z).$
 Thus $$-\tau^{\widetilde{\sH}_{h}}(u ,v) -[u,v]\in \oplus_{j=i}^{n-1}\fv^{j }.$$
On the other hand, regarding $h \in \widetilde{P}^n \subset \bL(I^{(n-1)})$ as an element of ${\rm Isom}(\fv^{<n}, T_x \bar{P}^{(n-1)})$ for $x = \widetilde{\pi}^n(h)$, we have by Lemma \ref{l.Sternberg} (iv),
\begin{eqnarray*}
-\tau^{\widetilde{\sH}_{z}}(u ,v) -[u,v]
&=& h^{-1} \circ  {\rm d}\widetilde{\pi}^n ([u^{\widetilde{\sH}}, v^{\widetilde {\sH}}] - [u,v]^{\widetilde{\sH}})\\
&=& h^{-1} \circ {\rm d}\bar{\pi}^n ([\widehat{u}  , \widehat{v} ] - \widehat{[u,v]} ).
\end{eqnarray*}
Therefore, $[\widehat{u}  , \widehat{v} ] - \widehat{[u,v]}$ is contained in $$   ( {\rm d} \bar{\pi}^{(n)})^{-1} \circ h  (\oplus_{j=i}^{n-1} \fv^j) =( {\rm d} \bar{\pi}^{(n)})^{-1}(\sD_i^{(n-1)}) = \sD_i^{(n)},$$ which proves (\ref{e.53}).
\end{proof}

The following lemma corresponds to Lemma 41 (iii) of \cite{AD}.

\begin{lemma} \label{l.57} Let $u \in \oplus_{j=-1}^{n-1}\fv^{j}$ and $\widehat{u}$ be as in Definition \ref{d.hatu}.  Then for any $h\in P^{n+1} \subset \bL(I^{(n)})$ and $y = \pi^{n+1}(h) \in \bar{P}^{(n)}$, regarding $h$ as an element of ${\rm Isom}(\fv^{<n+1}, T_y \bar{P}^{(n)}),$
\begin{equation} \label{e.55}
{\rm d} \bar{\pi}^{(n)} (h(u)) ={\rm d} \bar{\pi}^{(n)} (\widehat{u}_y).
\end{equation}
\end{lemma}

\begin{eqnarray*}
\xymatrix{
& \qquad P^{n+1} \ar[dd]^{\pi^{n+1}}  \ni h \\
f \in \widetilde{P}^n  \qquad \qquad \ar[dd]_{\widetilde{\pi}^{n}} \ar[rd]^{{\rm GL}_{n+1}(\fv^{<n})} & \\
& \qquad  \bar{P}^{(n)} \ni y  \ar[ld]_{\bar{\pi}^{(n)}}^{\mathsf{G}^n} \\
 \bar{P}^{(n-1)} &
}
\end{eqnarray*}

\begin{proof}
Fix $i, -1 \leq i \leq n-1.$
From Definition \ref{d.Pn+1},  \begin{equation}\label{e.571} {\rm pr}^i_{(n+1)} \circ {\rm d} \bar{\pi}^{(n)} \circ h|_{\fv^i} = \xi^{(n)}(y)|_{\fv^i}. \end{equation}
From the commutative diagram in (B3) of Theorem \ref{t.Tanaka1},
for any $f \in \widetilde{P}^n \subset \bL(I^{(n-1)})$ which projects to $y \in \bar{P}^{(n)}$,
\begin{equation}\label{e.572} {\rm pr}^i_{(n+1)} \circ f|_{\fv^i} = \xi^{(n)}(y)|_{\fv^i}.\end{equation}
But ${\rm pr}^i_{(n+1)}: \sD_i  ^{(n-1)}  \to \gr^i_{(n+1)}(\sD_{\bullet}^{(n-1)})$ is an isomorphism because  $\sD_{i+n+1}^{(n-1)}=0$ for $-1 \leq i \leq n-1$. Thus (\ref{e.571}) and (\ref{e.572}) imply
${\rm d} \bar{\pi}^{(n)} (h(u)) =f (u)$  for any $ u \in \fv^i.$ Since $${\rm d} \bar{\pi}^{(n)} (\widehat{u}_y) = {\rm d} \widetilde{\pi}^n (u^{\widetilde{\sH}}_f) =f(u)$$ by the definition of $u^{\widetilde{\sH}}$,  we obtain (\ref{e.55}).
\end{proof}

\subsection{Torsion of $\pi^{n+1}: P^{n+1} \rightarrow \bar{P}^{(n)}$}

The following is a modification of  Theorem 39 of \cite{AD}.

\begin{proposition} \label{p.44}  For a point $z \in P^{n+1} \subset \mathbb F\bar{P}^{(n)}$, let $\sH_z \subset T_zP^{n+1}$ be a $\pi^{n+1}$-horizontal subspace.  Then the torsion $\tau^{\sH_z} \in \Hom(\wedge^2\fv^{<n+1}, \fv^{<n+1})$ and its  degree-zero component $(\tau^{\sH_z})^0$ satisfy the following.
\begin{itemize}
\item [(i)] If $u \in \fv^n$ and $v \in \fv^i, -k \leq i \leq n$,
$$\tau^{\sH_z}(u, v) \in \left\{ \begin{array}{ll} \oplus_{j=i+n}^n \fv^j & \text{ if } i \leq 0 \\  \fv^n & \text{ if } i > 0. \end{array} \right. $$  \item[(ii)]  If $u \in \fv^n$ and $v \in \fv^i, -k \leq i \leq -1$, $$(\tau^{\sH_z})^0(u,v) = -[u,v].$$
\item[(iii)] If $u \in \fv^{-1}$ and $v \in \fv^i, -k \leq i \leq n$,
$$ \tau^{\sH_z}(\fv^{-1} \wedge \fv^i) \subset \oplus_{j= i-1}^n \fv^{j} \ \mbox{  and } \ (\tau^{\sH_z})^0(u,v) = -[u,v].$$
\end{itemize}
Here, the Lie bracket is that of $\fg(b)$ for $b = \beta^{(n)} \circ \pi^{n+1}(z).$
\end{proposition}

\begin{proof} We may replace $P^{n+1}$ by a neighborhood of $z$ and
use Definition \ref{d.connection} to fix  a $\pi^{n+1}$-connection $\sH$ in a neighborhood $z \in U \subset P^{n+1}$ such that it is locally invariant under the right action of $H(\fv^{<n+1})$ on $P^{n+1}$ and its value at $z$ is the given subspace $\sH_z$. Then for each $w \in \fv^{<n+1}$, we have the $\sH$-horizontal vector field $w^{\sH}$ on $U$.

For  $u \in \fv^n$, we have   the fundamental vector field $u^{P^{n+1}}$ on $P^{n+1}$ from Lemma \ref{l.42}. By the same argument for  (\ref{e.3.4b}), we see that \begin{equation}\label{e.4.12b} {\rm d} \pi^{n+1}(u^{\sH} - u^{P^{n+1}}) =0 \mbox{ on } U \subset P^{n+1}.\end{equation}  Since $\pi^{n+1}: P^{n+1} \to \bar{P}^{(n)}$ is a principal bundle with the structure group $H(\fv^{<n+1})$, for a basis $\{ A_s \mid 1 \leq s \leq d:= \dim \fh(\fv^{<n+1})\}$ of $\fh(\fv^{<n+1})$, we can write, after shrinking the neighborhood $U \subset P^{n+1}$ of $z$ if necessary,
$$u^{\sH} - u^{P^{n+1}} = \sum_{s=1}^d f_s A_s^{P^{n+1}},$$ where $f_s$ is a holomorphic function on $U$ and $A_s^{P^{n+1}}$ is the fundamental vector field on $P^{n+1}$ corresponding to $A_s$. Then for any $v \in \fv^i, -k \leq i \leq n,$
\begin{eqnarray*} \theta^{n+1}([u^{\sH} - u^{P^{n+1}}, v^{\sH}]_z) &=& \sum_{s=1}^d f_s(z) \theta^{n+1}([A^{P^{n+1}}_s, v^{\sH}]_z)  \\
& = & \sum_{s=1}^d f_s(z) \theta^{n+1}(A_s(v))^{\sH_z} \\ & = & \sum_{s=1}^d f_s(z) A_s(v),
\end{eqnarray*} where the second equality is from Lemma \ref{l.Sternberg} (iii). Thus we have
 $A \in \fh(\fv^{<n+1})$ such that
 \begin{equation}\label{e.Av}  \theta^{n+1}([u^{\sH} - u^{P^{n+1}}, v^{\sH}]_z) = A(v). \end{equation}
Since $\theta^{n+1}(u^{P^{n+1}}) = \theta^{n+1}(u^{\sH}) = u$ by (\ref{e.4.12b}), we have $${\rm Lie}_{u^{P^{n+1}}} \theta^{n+1} = {\rm d}(\theta^{n+1}(u^{P^{n+1}})) + {\rm d}  \theta^{n+1}(u^{P^{n+1}}, \cdot) = {\rm d} \theta^{n+1} (u^{P^{n+1}}, \cdot).$$ Consequently,
$$ ({\rm Lie}_{u^{P^{n+1}}}\theta^{n+1})_z(v^{\sH_z}) = {\rm d} \theta^{n+1}(u^{P^{n+1}}_z, v^{\sH_z}) =   -\theta^{n+1}([u^{P^{n+1}}, v^{\sH}]_z).$$ Thus Lemma \ref{l.43} implies \begin{equation}\label{e.4.12c}- \theta^{n+1}([u^{P^{n+1}}, v^{\sH}]_z) = -{\rm ad}_u(v) = -[u, v]. \end{equation}
It follows that
\begin{eqnarray*} \tau^{\sH_z}(u, v) & = & {\rm d} \theta^{n+1} (u^{\sH_z}, v^{\sH_z}) \\
& =& - \theta^{n+1}([u^{\sH}, v^{\sH}]_z) \\ & = & - \theta^{n+1}([u^{P^{n+1}}, v^{\sH}]_z) - \theta^{n+1}([u^{\sH} - u^{P^{n+1}}, v^{\sH}]_z) \\ & = & - [u, v] - A(v),\end{eqnarray*}
where the last equality  uses (\ref{e.Av}) and (\ref{e.4.12c}).
If $v \in \fv^i$ for $i<0$, then $[u, v] \in \fv^{i+n}$ and $A(v) \in \oplus_{j=i+n+1}^n\fv^{j}$. Thus $\tau^{\sH_z}(u,v) \in \oplus_{j=i+n}^n \fv^{j}$ and $(\tau^{\sH_z})^0(u,v) = -[u,v]$.
If $v \in \fv^i$ for $0 \leq i \leq n$, then $[u, v]=0$ and $A(v) \in \fv^n$.
This proves (i) and (ii).

For (iii), we skip the proof when $-k \leq i \leq -1,$ for which the arguments in the proof of  Proposition \ref{p.24} (i) and (ii)  can be applied with obvious changes of indices.
We prove the case $0 \leq i \leq n$.
Let us use the following notation.
For a vector field $\alpha$ in a neighborhood of $\pi^{n+1}(U)$ in $\bar{P}^{(n)}$, denote by $\alpha^{\sH}$ the unique vector field on $U$ tangent to $\sH$ such that ${\rm d}_{z'} \pi^{n+1} (\alpha^{\sH_{z'}}) = \alpha_{\pi^{n+1}(z')}$ for $z' \in U$. Also, denote by $\alpha^{\sH_{z'}} \in T_{z'} U$ its value at $z' \in U$. We can use the vector fields on $\bar{P}^{(n)}$ defined in Definition \ref{d.hatu} by choosing a section $\w{\Sigma}$ locally.

For $u, v \in \oplus_{j=-1}^{n-1} \fv^j$, there are $a, b \in \fv^n$ from Lemma \ref{l.57} satisfying $$u^{\sH_z} = \widehat{u}^{\sH_z} + (a^{\bar{P}^{(n)}})^{\sH_z} \mbox{ and } v^{\sH_z} = \widehat{v}^{\sH_z} + (b^{\bar{P}^{(n)}})^{\sH_z},$$
where $a^{\bar{P}^{(n)}}$ (resp. $b^{\bar{P}^{(n)}}$) is the fundamental vector field on $\bar{P}^{(n)}$ corresponding to $a \in \fv^n$ (resp. $b \in \fv^n$) for the $\mathsf{G}^n$-principal bundle $\bar{\pi}^{(n)}$. Then \begin{eqnarray*}
\tau^{\sH_z}(u,v) & = & {\rm d} \theta^{n+1}( u^{\sH_z}, v^{\sH_z}) \\ & = & {\rm d} \theta^{n+1}(\widehat{u}^{\sH_z} + (a^{\bar{P}^{(n)}})^{\sH_z}, \widehat{v}^{\sH_z} + (b^{\bar{P}^{(n)}})^{\sH_z}), \end{eqnarray*} which is equal to the value at the point $z \in U$ of the function on $U$ \begin{equation}\label{e.dtheta} {\rm d} \theta^{n+1}(\widehat{u}^{\sH} + (a^{\bar{P}^{(n)}})^{\sH}, \widehat{v}^{\sH} + (b^{\bar{P}^{(n)}})^{\sH} ).\end{equation} For any point $z' \in U \subset P^{n+1}$, denote by $$h_{z'} \in {\rm Isom}(\fv^{<n+1}, T_{y'} \bar{P}^{(n)}), \  y':= \pi^{n+1}(z')$$  the element of $\bL(I^{(n)})$ representing $z' \in P^{n+1} \subset \bL(I^{(n)}).$ Then $$\theta^{n+1}_{z'}(\widehat{v}^{\sH} + (b^{\bar{P}^{(n)}})^{\sH}) = h_{z'}^{-1}({\rm d} \pi^{n+1}(\widehat{v}^{\sH} +(b^{\bar{P}^{(n)}})^{\sH})) = h_{z'}^{-1}(\widehat{v} + b^{\bar{P}^{(n)}}).$$
Since $h_{z'}^{-1}(\widehat{v}) \equiv v \mod \fv^n$ by  Lemma \ref{l.57}, we see that the function $\theta^{n+1}(\widehat{v}^{\sH} + (b^{\bar{P}^{(n)}})^{\sH})$ on $U$ has the constant value $v$ modulo $\fv^n$.
It follows that the function on $U$ $$(\widehat{u}^{\sH} + (a^{\bar{P}^{(n)}})^{\sH}) \theta^{n+1}(\widehat{v}^{\sH} + (b^{\bar{P}^{(n)}})^{\sH}),$$ namely, the derivative of the function  $\theta^{n+1}(\widehat{v}^{\sH} + (b^{\bar{P}^{(n)}})^{\sH})$ by the vector field
$\widehat{u}^{\sH} + (a^{\bar{P}^{(n)}})^{\sH}$,
 has values in $\fv^n.$
By the same argument, so does the function on $U$ $$ (\widehat{v}^{\sH} + (b^{\bar{P}^{(n)}})^{\sH}) \theta^{n+1}(\widehat{u}^{\sH} + (a^{\bar{P}^{(n)}})^{\sH}).$$ Thus  the $\fv^{<n+1}$-valued function (\ref{e.dtheta})  is equal to $$ - \theta^{n+1}( [ \widehat{u}^{\sH} + (a^{\bar{P}^{(n)}})^{\sH}, \widehat{v}^{\sH} + (b^{\bar{P}^{(n)}})^{\sH}])$$ modulo $\fv^n$. By  (\ref{e.54}), the value of this function at $z$ is equal to $$-h_z^{-1}([\widehat{u} + a^{\bar{P}^{(n)}}, \widehat{v} + b^{\bar{P}^{(n)}}]_y) = - h_z^{-1}([\widehat{u}, \widehat{v}]_y + \widehat{a(v)}_y - \widehat{b(u)}_y)$$ modulo $\fv^n$.
If $u \in \fv^{-1}$ and $v \in \fv^i$ with $0 \leq i \leq n-1$, we obtain $$\tau^{\sH_z}(u, v)  \equiv -h_z^{-1}(\widehat{[u,v]} - \widehat{b(u)}) \mod \fv^n$$ by (\ref{e.53}). This proves (iii) when $0 \leq i \leq n-1$. \end{proof}

\subsection{Variation of the torsion of $\pi^{n+1}$}\label{ss.4.6}

\begin{proposition}  \label{p.45} For a point $z \in P^{n+1} \subset \F \bar{P}^{(n)}$,
 let $\sH_z$ and $\sH_z'$ be two horizontal subspaces of $T_zP^{n+1}$.
 \begin{itemize}
 \item[(i)] For $u \in \fv^{-1}$ and $v \in \fv^i$ with $i<0$, the components of degree $m \leq n+i$    of $\tau^{\sH_z}(u,v)$ and $\tau^{\sH'_z}(u,v)$ coincide:
      $$\tau^{\sH_z}(u,v)^{m} = \tau^{\sH_z'}(u,v)^{m} \mbox{ for any } m \leq n+i.$$
  \item[(ii)]
 For $u \in \fv^{-1}$ and $v \in \fv^i$ with $0 \leq i \leq n-1$, the components of degree $m \leq n-1$    of $\tau^{\sH_z}(u,v)$ and $\tau^{\sH'_z}(u,v)$ coincide:
     $$\tau^{\sH_z}(u,v)^m = \tau^{\sH_z'}(u,v)^m \mbox{ for any } m \leq n-1.$$

 \end{itemize}
\end{proposition}

\begin{proof}
Applying  Lemma \ref{l.Sternberg} (ii) to the principal bundle $\pi^{n+1}: P^{n+1} \to \bar{P}^{(n)}$ with the structure group $H(\fv^{<n+1})$, we obtain $$a,b \in \fh(\fv^{<n+1}) = \fgl_{n+1}(\fv^{<n+1}) + \Hom(\oplus_{i=0}^{n-1}\mathfrak \fv^i, \mathfrak \fv^n)$$ such that $\tau^{\sH'_z}(u,v) -  \tau^{\sH_z}(u,v) = -a(v) + b(u)$ for any $u, v \in \fv^{<n+1}$.

If $u \in \fv^{-1}$ and $v \in \fv^i$ for $i<0$, then $a(v) \in \oplus_{j= n+1+i}^n \fv^{j}$   and $b(u) \in \fv^n$. Thus $\tau^{\sH_z}(u,v)^{m} = \tau^{\sH_z'}(u,v)^{m}$ if $m \leq n+i$, proving (i).

If $u \in \fv^{-1}$ and $v \in \fv^i$ for $0 \leq i \leq n-1$, then both $a(v)$ and $b(u)$ are contained in $\fv^n$, and thus $\tau^{\sH_z}(u,v)^m = \tau^{\sH_z'}(u,v)^m \mbox{ for any } m \leq n-1,$ proving (ii).
\end{proof}

\begin{definition}\label{d.46}
Recall from Definition \ref{d.partial},
$${\rm Tor}^{n+1}(\fv)=\Hom^{n+1}(\fv^{-1}\wedge \fv^{-}, \fv^{<n+1}) \oplus \Hom(\oplus_{i=0}^{n-1}(\fv^{-1} \wedge \fv ^i), \fv^{n-1}) .$$
For $z \in P^{n+1}$, define $\tau^{n+ 1}(z) \in {\rm Tor}^{n+1}(\fv)$ by
    \begin{eqnarray*}
    \tau^{n+1}(z)(u,v):=\left\{
    \begin{array}{l}
    \tau^{\sH_z}(u,v)^{n+i}  \text{ for  } u \in \fv^{-1} \text{ and } v \in \fv^i \text{ with } i<0 \\[5pt]
    \tau^{\sH_z}(u,v)^{n-1} \text{ for } u \in \fv^{-1} \text{ and } v \in  \fv^i  \text{ with }  0 \leq i \leq n-1,
    \end{array}
    \right.
    \end{eqnarray*}
    where $\sH_z$ is any $\pi^{n+1}$-horizontal subspace. This is
 is independent of the choice of $\sH_z$ by Proposition \ref{p.45}. Then $\tau^{n+1}$ is a ${\rm Tor}^{n+1}(\fv)$-valued holomorphic function on $P^{n+1}$.
\end{definition}

\begin{proposition} \label{p.48}  For  $A \in \fh(\fv^{<n+1}) = \mathfrak{gl}_{n+1}(\fv^{<n+1}) + \Hom(\oplus_{i=0}^{n-1} \fv^i, \fv^n)$ and $a:={\rm Id} + A \in H(\fv^{<n+1})$,
let $R_a: P^{n+1} \to P^{n+1}$ be the right action of $a$ on the principal  bundle $\pi^{n+1}: P^{n+1}  \to \bar{P}^{(n)}$ with the structure group $H(\fv^{<n+1}).$
Then  the function $\tau^{n+1}$ on $P^{n+1}$ defined in Definition \ref{d.46}  satisfies
$$\tau^{n+1}_{R_a z} = \tau^{n+1}_z + \partial^{n+1} (A \times b)$$ for any $z \in P^{n+1}$ and $b = \beta^{(n)} \circ \pi^{n+1}(z)$,  where $\p^{n+1}$ is the vector bundle homomorphism in Set-up \ref{n.B} (6).
\end{proposition}

\begin{proof} To simplify the notation, let us write $h(u) = hu$ for $h \in \End(\fv^{<n+1})$ and $u \in \fv^{<n+1}$ in this proof.  Let ${\rm Id} + \check{A} \in H(\fv^{<n+1})$ be the inverse of $a= {\rm Id} + A$ with $\check{A} \in \fh(\fv^{<n+1})$.
Write $A = \sum_{m\geq 1} A^m$ and $\check{A}= \sum_{m\geq 1} \check{A}^m$ with $\check{A}^m=-A^m$ for $m \leq n+1$ from Lemma \ref{l.barG}.
By Lemma \ref{l.Sternberg} (v),
we have $\tau^{\sH_{R_a z}}(u,v) = a^{-1} \tau^{\sH_z}( au, av).$ Thus
\begin{eqnarray}\label{e.48}
\lefteqn{\tau^{\sH_{R_a z}}(u,v) =} \\  \nonumber & &  \tau^{\sH_{z }}(u,v) + \tau^{\sH_{z }}(u,Av) + \tau^{\sH_{z }}(Au,v) + \tau^{\sH_{z }}(Au,Av)  \\ \nonumber
& & + \check{A} (\tau^{\sH_{z }}(u,v) + \tau^{\sH_{z }}(u,Av) + \tau^{\sH_{z }}(Au,v) + \tau^{\sH_{z }}(Au,Av)).
\end{eqnarray}

First assume that $u \in \fv^{-1}$ and $v \in \fv^i$ with $i<0$.  Then  $$Au = A^{n+1}u \in \fv^n \mbox{ and } Av \in \oplus_{j= i+n+1}^n \fv^j  \mbox{ with } (Av)^{i + n +1} = A^{n+1} v.$$  By Proposition \ref{p.44}, we have
 $$\tau^{\sH_{z }}(Au, v) \in \oplus_{j = i+n}^n \fv^{j} \mbox{ with } \tau^{\sH_z}(Au, v)^{i+n} = (\tau^{\sH_z})^0(A^{n+1} u, v),$$
$$  \tau^{\sH_{z }}(u, Av) \in \oplus_{j= i+n}^n \fv^{j} \mbox{ with } \tau^{\sH}_z(u, Av)^{i+n} = (\tau^{\sH_z})^0(u, A^{n+1} v), \mbox{ and }$$
$$ \tau^{\sH_{z }} (Au, Av)\in \oplus_{j = \min\{n, 2n+i+1\}}^n \fv^j.$$
 Putting these into (\ref{e.48}), we see that the $\fv^{i+n}$-component of $ \tau^{\sH_{R_a z}} (u,v)$ should satisfy
\begin{eqnarray*}
\lefteqn{\tau^{\sH_{R_a z}}(u,v)^{i+n} =}    \\ & & \tau^{\sH_{z }}(u,v)^{i+n} + (\tau^{\sH_{z }})^{0}(u,A^{n+1}v) + (\tau^{\sH_{z }})^{0}(A^{n+1}u,v)  \\
&&+ \text{ the } \fv^{i+n}\text{-component of } \check{A} (\tau^{\sH_{z }} (u,v)).
\end{eqnarray*}
Since $\tau^{\sH_{z }} (u,v) \in \oplus_{j=i-1}^n \fv^{j}$ from Proposition \ref{p.44}   and $$\check{A} (\oplus_{j \geq i} \fv^j) \subset \oplus_{m \geq i+n+1}\fv^m,$$   the $\fv^{i+n}$-component of $\check{A} (\tau^{\sH_{z }} (u,v))$ should be $$\check{A}^{n+1}( (\tau^{\sH_z})^0(u,v))=  -A^{n+1}((\tau^{\sH_z})^0(u,v)).$$ Combining this with  Proposition \ref{p.44} (ii), we obtain
$$\tau^{\sH_{R_a z}}(u,v)^{i+n} = \tau^{\sH_{z }}(u,v)^{i+n} -[ u,A^{n+1}v] -[A^{n+1}u,v] + A^{n+1}[u, v].$$
This proves $\tau^{n+1}_{R_a z}(u,v) = \tau^{n+1}_z(u,v) + \partial^{n+1}_{\fg(b)} A (u,v)$.

Next assume that $u \in \fv^{-1}$ and $v \in \fv^i$ with $0 \leq i \leq n-1$.
Then $$Au = A^{n+1}u \in \fv^n \mbox{ and } Av = A^{n-i}v \in \fv^n.$$ Thus by Proposition \ref{p.44} (i),
$$\tau^{\sH_z}(Au, v) \in \fv^n, \ \tau^{\sH_z}(u, Av) \in \fv^{n-1} + \fv^n,$$
$$\tau^{\sH_z}(Au, Av) \in \fv^n, \ \check{A}(\tau^{\sH_z}(u, v)) \subset  \check{A} (\oplus_{j=i-1}^n \fv^j) \subset \fv^n.$$ Putting these into (\ref{e.48}), we see that the $\fv^{n-1}$-component of $\tau^{\sH_{R_a z}}(u,v)$ should be
\begin{eqnarray*}
\tau^{\sH_{R_a z}}(u,v)^{n-1} & = &\tau^{\sH_{z }}(u,v)^{n-1} + (\tau^{\sH_{z }})^{0}(u,A^{n-i}v) \\ &=& \tau^{\sH_{z }}(u,v)^{n-1} - [u, A^{n-i}v],
\end{eqnarray*} where the second equality uses Proposition \ref{p.44} (iii).
This proves $\tau^{n+1}_{R_a z}(u,v) = \tau^{n+1}_z(u,v) + \partial^{n+1}_{\fg(b)} A (u,v)$.
\end{proof}

\begin{proposition} \label{p.49}
Let us regard $\tau^{n+1}$ as a holomorphic map between fiber bundles $$\begin{array}{ccc}
P^{n+1} & \stackrel{\tau^{n+1}}{\longrightarrow} & {\rm Tor}^{n+1}(\fv) \times B \\
\pi^{n+1} \downarrow & & \downarrow \\\bar{P}^{(n)} & \stackrel{\beta^{(n)}}{\longrightarrow} & B. \end{array} $$
Let  $\sW^{n+1} \subset {\rm Tor}^{n+1}(\fv) \times B$ be the vector subbundle from Set-up \ref{n.B} (7) satisfying ${\rm Tor}^{n+1}(\fv) \times B = {\rm Im}(\p^{n+1}) \oplus \sW^{n+1}$. Define $\widetilde{P}^{n+1} := (\tau^{n+1})^{-1}(\sW)$ and the natural projection $$\widetilde{\pi}^{n+1}:= \pi^{n+1}|_{ \widetilde{P}^{n+1}}  : \widetilde{P}^{n+1}  \to \bar{P}^{(n)}.$$ Then \begin{itemize} \item[(i)] $\widetilde{\pi}^{n+1}$  is a $\beta^{(n)}$-principal subbundle of $\pi^{n+1}: P^{n+1} \to \bar{P}^{(n)}$ with the structure $B$-group $\mathsf{G}^{n+1} {\rm GL}_{n+2}(\fv^{<n+1}) \subset H(\fv^{<n+1})$; and \item[(ii)]  for any $\widetilde{\pi}^{n+1}$-horizontal subspace $\sH_z \subset T_z \widetilde{P}^{n+1}$ at $z \in \widetilde{P}^{n+1}$, the restriction
$\tau^{\sH_z}|_{\fv^{-1} \wedge \fv^{<n+1}}$ has only components of nonnegative homogenous degree and satisfies
$ (\tau^{\sH_z} )^0(u,v)  = - [u, v] $ for $u \in \fv^{-1}$ and $v \in \fv^{<n+1}.$ \end{itemize} \end{proposition}

\begin{proof}
For the fiber  $P^{n+1}_y:=(\pi^{n+1})^{-1}(y) $ over a point  $y \in \bar{P}^{(n)}$ with $b= \beta^{(n)}(y)$, Proposition \ref{p.48} says that the $\tau^{n+1}(P^{n+1}_y)$  is an affine subspace of ${\rm Tor}^{n+1}(\fv),$  a translate of the vector subspace $${\rm Im}(\p^{n+1}_{\fg(b)}) \ \subset \ {\rm Tor}^{n+1}(\fv) = {\rm Im}(\p^{n+1}_{\fg(b)}) \oplus \sW^{n+1}_b.$$ Thus the intersection $\tau^{n+1}(P^{n+1}_y) \cap \sW^{n+1}_b$ is a single point and $$\widetilde{P}^{n+1}_y := (\widetilde{\pi}^{n+1})^{-1}(y) \ = \ (\tau^{n+1})^{-1}(\tau^{n+1}(P^{n+1}_y) \cap \sW^{n+1}_b).$$
By Lemma \ref{l.partial} and Proposition \ref{p.48}, the subgroup of $H(\fv^{<n+1})$ with Lie algebra $\fg(b)^{n+1} + \fgl_{n+2}(\fv^{<n+1})$ acts simply  transitively on the fiber $\widetilde{P}^{n+1}_y$.  This proves (i).

Regarding $\sH_z$ as a $\pi^{n+1}$-horizontal subspace of $T_z P^{n+1}$ and applying  Proposition \ref{p.44}, we obtain (ii).
 \end{proof}

\begin{proof}
[Proof of  Theorem \ref{t.Tanaka1}: From $\bar{P}^{(n)}$ to $\bar{P}^{(n+1)}$] The base $\bar{P}^{(n)}$ satisfies condition (A) as checked in Section \ref{ss.Pn}.  The $\beta^{(n)}$-principal bundle $\widetilde{P}^{n+1} \rightarrow \bar{P}^{(n)}$ and its quotient $\bar{P}^{(n+1)}:=\widetilde{P}^{n+1}/{\rm GL}_{n+2}(\fv^{<{n+1}}) \rightarrow \bar{P}^{(n)}$ equipped with a natural embedding $\zeta^{(n+1)}:  \bar{P}^{(n+1)} \subset \bL_{(n+2)}(I^{(n)})$ induced from $\widetilde{P}^{n+1} \subset P^{n+1} \subset \bL(I^{(n)})$ satisfy conditions in (B) and (C) with  $n $ replaced by $n+1$ by Proposition \ref{p.49}.
\end{proof}

\section{Pseudo-product structures and Levi-nondegeneracy}\label{s.PP}

\begin{definition}\label{d.distribution}
A {\em distribution} on a complex manifold $M$ means a vector subbundle $D \subset TM$ of the tangent bundle. \begin{itemize} \item[(1)]
The Lie brackets of local vector fields define a vector bundle homomorphism $${\rm Levi}^D: D \otimes D \to TM/D,$$ called the {\em Levi tensor} of $D$. The distribution $D$ is {\em integrable} if ${\rm Levi}^D$ is identically zero, which is equivalent to saying that $D$ is tangent spaces of leaves of a holomorphic foliation on $M$. Let $\chi: D \to \Hom(D, TM/D)$ be the vector bundle homomorphism defined by  $$ \chi(u) (v) := {\rm Levi}^D(u,v) \in T_x M/D_x \mbox{ for } u, v \in D_x, x \in M.$$  Let $M_{\chi} \subset M$ be the nonempty Zariski-open subset where the homomorphism $\chi$ has constant rank and let ${\rm Ch}(D) \subset T M_{\chi}$ be the distribution defined by ${\rm Ker}(\chi)$, called the {\em Cauchy characteristic} of $D$. This means that for  any local holomorphic sections $u$ of ${\rm Ch}(D)$ and $v$ of $D$, the Lie bracket $[u, v]$ is a local holomorphic section of $D$. It is easy to check that the Cauchy characteristic is integrable. We say that the distribution $D$ is {\em Levi-nondegenerate} if ${\rm Ch}(D) =0$.
\item[(2)]
Set $D_{-1} := D,  M_{-1}:= M$ and  $\delta_{-2}:= {\rm Levi}^D$.  Let $M_{-2} \subset M_{-1}$ be the nonempty Zariski-open subset where $\delta_{-2}$ has constant rank and let $D_{-2} \subset TM_{-2}$ be the distribution on $M_{-2}$ satisfying $D_{-2}/D_{-1} = {\rm Im}(\delta_{-2})$ on $M_{-2}$. Inductively define a sequence of Zariski-open subsets $$M_{-k} \subset M_{-k+1} \subset \cdots \subset M_{-2} \subset M_{-1} = M$$
and a distribution $D_{-i} \subset T M_{-i}$ for each $1 \leq i \leq k$ as follows. \begin{itemize} \item[(i)] For $2 \leq i \leq k$, define
the vector bundle homomorphism $$\delta_{-i}: D_{-1} \otimes D_{-i+1} \to TM_{-i+1}/D_{-i+1}$$ by Lie brackets of local vector fields given by local holomorphic sections of $D_{-1}$ and $D_{-i+1}.$ Let $M_{-i} \subset M_{-i+1}$ be the Zariski-open subset where $\delta_{-i}$ has constant rank and let $D_{-i} \subset T M_{-i}$ be the distribution satisfying $D_{-i}/D_{-i+1} = {\rm Im}(\delta_{-i})$ on $M_{-i}$.
\item[(ii)] The integer $k$ is the smallest positive integer such that $\delta_{-k-1} =0$.
\end{itemize}  We call $k$ the {\em depth} of the distribution $D$. We say that the distribution $D$ is {\em bracket-generating} if $D_{-k} = T M_{-k}$.
\item[(3)] Suppose that $D$ is bracket-generating. The filtration $D_{-1} \subset D_{-2} \subset \cdots \subset D_{-k} =  TM_{-k}$ is called the {\em weak derived system} of $D$.  It is a Tanaka filtration in the sense of Definition \ref{d.filtered}. The symbol algebra ${\rm symb}_x(D_{\bullet})$ of this filtration at $x \in M_{-k}$ is called the {\em symbol algebra} of the distribution $D$ at $x \in M_{-k}$, to be denoted by $\fg_x^-= \oplus_{i=-k}^{-1} \fg_x^{i}.$  It is a graded nilpotent Lie algebra generated by $\fg_x^{-1}$.  \item[(4)] A point $x$ in the intersection $M_{\chi} \cap M_{-k}$ is called a {\em regular point } of the distribution $D$. \end{itemize} \end{definition}

We recall the following notion from Section 1.5 of \cite{Ta85}.

\begin{definition}\label{d.PP}
A {\em pseudo-product structure} on a complex manifold $M$ is a pair $(E,F)$ of vector subbundles of $TM$ of positive rank such that  $E \cap F =0$ and both $E$ and $F$ are integrable. We say that $(E, F)$ is {\em bracket-generating} if the  distribution  $D:= E+ F \subset TM$ is bracket-generating.
Fix a bracket-generating pseudo-product structure $(E,F)$ on $M$.
\begin{itemize}
\item[(i)] At a regular point $x \in M$ of the distribution $D,$ its symbol algebra $ \fg_x^- = \oplus_{i=-k}^{-1} \fg_x^{i}$ contains two distinguished abelian subalgebras $\fe_x, \ff_x \subset \fg_x^{-1}$ corresponding to the integrable distributions $E, F \subset D$.
Define $\fg_x^0 \subset \fgl(\fg_x^-)$ as the Lie subalgebra consisting of gradation-preserving endomorphisms  $h \in \Hom^0(\fg_x^-, \fg_x^-) \subset \End(\fg_x^-)$ satisfying  $$h(\fe_x) \subset \fe_x \mbox{ and } h(\ff_x) \subset \ff_x.$$
Then $\fg_x^- \oplus \fg_x^0$ is a fundamental graded Lie algebra in the sense of Definition \ref{d.prolongLie}.
\item[(ii)] Let $\fg_x^{\bullet} = \oplus_{j = -k}^{\infty} \fg^j_x$ be the universal prolongation of $\fg_x^- \oplus \fg_x^0$ in the sense of Definition \ref{d.prolongLie}.
    We say that the pseudo-product structure is {\em  of finite height} at $x$ if $\fg_x^{\ell +1} =0$ for some $\ell \geq 0$ and it  has {\em height $\ell$} at $x$ if $\fg_x^{\ell +1} =0$ and $\fg_x^{\ell} \neq 0$.
In this case, the graded Lie algebra $$\fg_x^{\bullet} := \bigoplus_{i=-k}^{\ell} \fg_x^i$$ is called the {\em prolongation of the symbol algebra} at $x \in M$ of the pseudo-product structure $(E, F)$.
\item[(iii)] The pseudo-product structure  {\em is of finite height} if it is so at a general point $x \in M$. In this case, we have an integer $\ell \geq 0$ and a Zariski-open subset $M_* \subset M$ consisting of regular points of $D$ such that the pseudo-product structure has height $\ell$ at any point of  $M_*$ and  $\dim \fg_x^i,$  for each $-k \leq i \leq \ell,$ is constant for all $x \in M_*$.  \end{itemize}
\end{definition}

We recall the following  fundamental result due to Tanaka, which was stated as  Lemma 1.14 of \cite{Ta85}, without an explicit proof. It is stated as Lemma 8.1 in \cite{Yat} with a proof, an  algebraic argument reducing it to the deep result, Corollary 1 to Theorem 11.1 in \cite{Ta70}.

\begin{theorem}\label{t.TanakaPP}
Let $(E,F)$ be a bracket-generating pseudo-product structure on a complex manifold $M$ such that the distribution $D = E+ F \subset TM$ is Levi-nondegenerate. Then the prolongation of the symbol algebra of $(E,F)$ at any regular point $x \in M$ of the distribution $D$ is finite-dimensional, namely, the pseudo-product structure is of finite height at $x$. \end{theorem}

Because of Theorem \ref{t.TanakaPP}, it is important to study the Levi-nondegeneracy of $D = E+F$. For this, we need to consider the following analogue of Definition \ref{d.tau}.

\begin{definition}\label{d.sharp}
In Definition \ref{d.PP}, let  $r$ (resp. $m$) be the rank of $E$ (resp. $F$).
Choose a connected open subset $\sU \subset M$ equipped with a submersion $\mu: \sU \to X$ (resp. $\rho: \sU \to \sK$) to a complex manifold $X$ (resp. $\sK$) such that fibers of $\mu$ (resp. $\rho$) are leaves of $E|_{\sU}$ (resp. $F|_{\sU}$).
  For each $y \in \sU$, set $x := \mu(y)$ (resp. $z := \rho(y)$) and  let $\mu^{\sharp} (y) \in {\rm Gr}(m; T_x X)$ (resp.   $\rho^{\sharp} (y) \in {\rm Gr}(r; T_z \sK)$) be the point in the Grassmannian of $m$-dimensional (resp. $r$-dimensional)  subspace in $T_x X$ (resp. $T_z \sK$) corresponding to ${\rm d} \mu ({\rm Ker}({\rm d}_y \rho)) \subset T_x X$ (resp. ${\rm d} \rho ({\rm Ker}({\rm d}_y \mu)) \subset T_z \sK$). This defines a holomorphic map $\mu^{\sharp}: \sU \to {\rm Gr}(m; TX)$ (resp. $\rho^{\sharp}: \sU' \to {\rm Gr}(r; T\sK)$) to the Grassmannian bundle over $X$ (resp. $\sK$):
$$ \begin{array}{ccccc}
{\rm Gr}(r; T\sK) & \stackrel{\rho^{\sharp}}{\longleftarrow} & \sU & \stackrel{\mu^{\sharp}}{\longrightarrow} & {\rm Gr}(m; TX) \\
\downarrow & & \| & & \downarrow \\ \sK & \stackrel{ \rho}{\longleftarrow} & \sU & \stackrel{ \mu }{\longrightarrow} & X. \end{array} $$
\end{definition}

The main result of this section is the following.

\begin{theorem}\label{t.Cauchy}
In Definition \ref{d.sharp}, assume that both $\rho^{\sharp}$ and $\mu^{\sharp}$ are generically immersive. Then $D = E+ F$ is Levi-nondegenerate. \end{theorem}

The following is an immediate corollary of Theorem \ref{t.TanakaPP} and Theorem \ref{t.Cauchy}.

\begin{corollary}\label{c.Cauchy}
Let $(E,F)$ be a bracket-generating pseudo-product structure on a complex manifold.
Assume that the maps $\rho^{\sharp}$ and $\mu^{\sharp}$ in Definition \ref{d.sharp} are generically immersive. Then $(E,F)$ is of finite height. \end{corollary}

For the proof of Theorem \ref{t.Cauchy}, we need some preparation.

\begin{definition}\label{d.sJ}
Let $X$ be a complex manifold. Fix a positive integer $m < \dim X.$ \begin{itemize} \item[(i)] The {\em tautological distribution} $J$ on the Grassmannian bundle $\pi: {\rm Gr}(m; TX) \to X$ is a vector subbundle $J \subset T {\rm Gr}(m; TX)$ defined as follows.
For each point $\alpha \in {\rm Gr}(m; TX)$ and $x = \pi (\alpha),$ let  $\widehat{\alpha} \subset T_x X$ be the $m$-dimensional subspace corresponding to $\alpha$ and let
 $${\rm d}_{\alpha} \pi : T_{\alpha} {\rm Gr}(m; TX) \ \longrightarrow \ T_x X$$ be the differential of the projection. Then the fiber of $J$  at $\alpha$ is $$J_{\alpha} := ({\rm d}_{\alpha} \pi)^{-1} (\widehat{\alpha}) \ \subset T_{\alpha} {\rm Gr}(m; TX).$$
\item[(ii)] Let $\sC \subset {\rm Gr}(m; TX)$ be a locally closed submanifold such that the image $\pi(\sC)$ is open subset in $X$ and the  projection   $\pi^{\sC}= \pi|_{\sC} :  \sC \to \pi(\sC)$ is  submersive.  The  subbundle $J^{\sC} :=J|_{\sC} \cap T\sC$ is called the {\em tautological distribution} on $\sC$. For $\alpha \in \sC, x = \pi^{\sC}(\alpha)$ and the differential of the projection
 $${\rm d}_{\alpha} \pi^{\sC}: T_{\alpha} \sC \ \longrightarrow \ T_x X,$$ the fiber $J^{\sC}$  at $\alpha \in \sC$ is $J^{\sC}_{\alpha} = ({\rm d}_{\alpha} \pi^{\sC})^{-1} (\widehat{\alpha})$, which has dimension  $\dim \sC - \dim X +m.$ \end{itemize} \end{definition}

 \begin{proposition}\label{p.sJsC}
 The Cauchy characteristic of the tautological distribution $J^{\sC}$ in Definition \ref{d.sJ} satisfies ${\rm Ch}(J^{\sC})_{\alpha} \cap {\rm Ker}({\rm d}_{\alpha} \pi^{\sC}) =0$ for a general point $\alpha \in \sC$.  \end{proposition}

For the proof, we use  the following well-known fact (for example, see page 122 of \cite{Ya82}).

 \begin{lemma}\label{l.sJ} In the setting of Definition \ref{d.sJ},
in a neighborhood of any point $\alpha \in {\rm Gr}(m; TX)$,  we have a holomorphic coordinate system $$(z^i, w^k, p^k_i; 1 \leq i \leq m, 1 \leq k \leq c), \ m + c = \dim X,$$ such that \begin{itemize} \item[(i)] $(z^i, w^k; 1 \leq i \leq m, 1 \leq k \leq c)$ are the $\pi$-pullback of holomorphic coordinates  on a neighborhood of $x = \pi(\alpha)$ in $X$; and
\item[(ii)] the distribution $J \subset T {\rm Gr}(m;TX)$ is given by the vanishing of the collection of 1-forms
$${\rm d} w^{k} - \sum_{i=1}^m
p^k_i {\rm d} z^i, \ 1 \leq k \leq c.$$\end{itemize} \end{lemma}

 \begin{proof}[Proof of Proposition \ref{p.sJsC}]
 Let us use the coordinates in a neighborhood of $\alpha$ in ${\rm Gr}(m; TX)$ given in Lemma \ref{l.sJ} and use Einstein summation convention on indices $1 \leq i, j \leq m$ and $1 \leq k, l \leq c$.  A holomorphic section of $J^{\sC}$ in a sufficiently small neighborhood of $\alpha$ in $\sC$ can be extended to a holomorphic section of $J$ in a neighborhood of $\alpha$ in ${\rm Gr}(m;TX)$, which can be written as \begin{equation}\label{e.J}  b^i(\frac{\p}{\p z^i} + p^k_i \frac{\p}{\p w^k}) + f^k_i \frac{\p}{\p p^k_i} \end{equation} for suitable local holomorphic functions $b^i$ and $f^k_i$ in a neighborhood of $\alpha$ in ${\rm Gr}(m; TX)$.  A holomorphic section of ${\rm Ker}({\rm d}\pi^{\sC})$ can be extended locally to a holomorphic section of ${\rm Ker}({\rm d} \pi)$ and can be written as $ v^l_j \frac{\p}{\p p^l_j}$. For this to be in ${\rm Ch}(J^{\sC})$, its bracket with (\ref{e.J}) modulo ${\rm Ker}({\rm d} \pi)$, \begin{eqnarray*} \lefteqn{ [ v^l_j \frac{\p}{\p p^l_j},  b^i(\frac{\p}{\p z^i} + p^k_i  \frac{\p }{\p w^k})]= } \\ & &  v^l_j \frac{\p b^i}{\p p^l_j} ( \frac{\p}{\p z^i} +p^k_i \frac{\p}{\p w^k}) + v^k_i  b^i\frac{\p}{\p w^k} - b^i(\frac{\p v^l_j}{\p z^i} + p^k_i \frac{\p v^l_j}{\p w^k}) \frac{\p}{\p p^l_j},   \end{eqnarray*}   should belong to $J^{\sC}$. By Lemma \ref{l.sJ}, it has to be annihilated by ${\rm d} w^k- p^k_i {\rm d} z^i$ for all $1 \leq k \leq c$.  It follows that $b^i v^k_i =0$ for all $1 \leq k \leq c$. Since $(b^1, \ldots, b^m)$ can be chosen to represent $m$ independent vectors, we conclude $v^k_i =0$ for all $1 \leq i \leq m$ and $1 \leq k \leq c$, proving Proposition \ref{p.sJsC}.
 \end{proof}

\begin{proof}[Proof of Theorem \ref{t.Cauchy}]
For a general point $y \in \sU$, we claim  \begin{equation}\label{e.Cauchy1} {\rm Ch}(D)_y = ({\rm Ch}(D)_y \cap E_y) + ({\rm Ch}(D)_y \cap F_y). \end{equation}
To see this, let $\fg_y^- = \oplus_{i=-k}^{-1} \fg_y^{i}$ be the symbol algebra of $D$ at $y$ and identify $D_y$ with $\fg_y^{-1}$.
    For $w \in  \fg_y^{-1} = E_y \oplus F_y$, write $w= w_{\fe} + w_{\ff}$ with $w_{\fe} \in E_y$ and $w_{\ff} \in F_y$. If $w \in {\rm Ch}(D)_y$, then  $[w, \fg_y^{-1}] =0$ in the Lie algebra $\fg_y^-$. From  $[F_y, F_y] =0$ and $$ [w_{\ff}, e] = [w_{\fe} +w_{\ff}, e] = 0 \mbox{ for all } e \in E_y,$$
    we have $w_{\ff} \in {\rm Ch}(D)_y \cap F_y$. By the same argument, we have $w_{\fe} \in {\rm Ch}(D)_y \cap E_y.$ This verifies (\ref{e.Cauchy1}).

    By (\ref{e.Cauchy1}), it suffices to prove ${\rm Ch}(D)_y \cap E_y = 0 = {\rm Ch}(D)_y \cap F_y$. Let us prove ${\rm Ch}(D)_y \cap E_y = 0$. The proof of  ${\rm Ch}(D)_y \cap F_y$
    follows from the same argument by replacing $E$ with $F$ (also $X$ with $\sK$, etc.).

 By the general choice of $y$, we can choose a neighborhood $U_y \subset \sU$ of $y$ such that  the image  $\sC := \mu^{\sharp}(U_y) \subset {\rm Gr}(m; TX)$   is    a locally closed complex submanifold in ${\rm Gr}(m; TX)$ and  $\mu^{\sharp}|_{U_y}: U_y \to \sC$ is a  submersion.

We claim  \begin{equation}\label{e.Cauchy2} D|_{U_y} = ({\rm d}\mu^{\sharp})^{-1} J^{\sC}, \end{equation} where  $J^{\sC}$ is the tautological distribution on $\sC$ from Definition \ref{d.sJ}.
It suffices to check $ ({\rm d}_u \mu^{\sharp})^{-1}(J^{\sC}_{\alpha}) = D_u$ for $u \in U_y$ and $\alpha = \mu^{\sharp}(u)$. Note that $\dim D_u = \dim  ({\rm d}_u \mu^{\sharp})^{-1}(J^{\sC}_{\alpha})$ because \begin{eqnarray*} \dim ({\rm d}_u \mu^{\sharp})^{-1}(J^{\sC}_{\alpha}) - \dim {\rm Ker}({\rm d}_u \mu^{\sharp}) & = & \dim J^{\sC}_{\alpha} \\  & = & \dim \sC - \dim X + m \\ & = & \dim \sC - (\dim \sU -r) + m \\ &=& m + r - \dim {\rm Ker} ({\rm d}_u \mu^{\sharp}). \end{eqnarray*}   But $D_u \subset ({\rm d}_u \mu^{\sharp})^{-1}(J^{\sC}_{\alpha})$ because
    ${\rm d}_u \mu^{\sharp} (E_u) \subset J^{\sC}_{\alpha}$  from $ {\rm d}_u \mu (E_u) = 0 $ and $ {\rm d}_u \mu^{\sharp} (F_u) \subset J^{\sC}_{\alpha} $ from  $    {\rm d}_u \mu (F_u) = \widehat{\alpha}.$   This proves (\ref{e.Cauchy2}).

  Since $\mu^{\sharp}$ is generically immersive, we may identify the neighborhood $U_y$ with  $\sC$ and $D|_{U_y}$ with $J^{\sC}$ by (\ref{e.Cauchy2}). Then  Proposition \ref{p.sJsC} implies  ${\rm Ch}(D)_y \cap E_y = 0.$
    \end{proof}

\section{$\beta$-Tanaka structures arising from pseudo-product structures}\label{s.PPsK}

We skip the proof of the next lemma, which is just a direct consequence of a classical result of Rosenlicht (for example, see Theorem 6.2 of \cite{Do}).

\begin{lemma}\label{l.Rosenlicht}
Let $G$ be a connected  linear algebraic group acting on an irreducible quasi-projective algebraic variety $Z$.
Then there is a decomposition into a disjoint union $$ Z = Z^1 \cup Z^2 \cup \cdots \cup Z^d$$ for some positive integer $d$ such that for each $ 1 \leq i \leq d,$ \begin{itemize} \item[(i)] $Z^i$  is a nonsingular quasi-projective variety preserved by the $G$-action; and \item[(ii)] there is a submersion $Z^i \to Q^i$ to a nonsingular quasi-projective algebraic variety $Q^i$ whose fibers are exactly the $G$-orbits in $Z^i$. \end{itemize} \end{lemma}

\begin{definition}\label{d.GLA}
  Let $\fv$ be a vector space with a gradation $\fv^{\bullet} = \oplus_{i = -k}^{\ell} \fv^i$.
Consider  a graded Lie algebra structure $\fg^{\bullet} = \oplus_{i=-k}^{\ell} \fg^i$  on $\fv^{\bullet}$ such that $\fg^{<1}$ is a fundamental graded Lie algebra and $\fg^{\bullet}$ is its universal prolongation. It is clear that the set of all such graded Lie algebra structures on $\fv^{\bullet}$ is a quasi-projective subvariety in $\Hom(\wedge^2 \fv, \fv)$, which we denote by $${\rm Tanaka}(\fv^{\bullet}) \ \subset \Hom(\wedge^2 \fv, \fv).$$
Let $\Aut(\fv^{\bullet}) \subset {\rm GL}(\fv)$ be  the group of  the graded vector space automorphisms of $\fv^{\bullet}$. The natural action of ${\rm GL}(\fv)$ on $\Hom(\wedge^2 \fv, \fv)$
induces a natural action of $\Aut(\fv^{\bullet})$ on ${\rm Tanaka}(\fv^{\bullet})$. By Lemma \ref{l.Rosenlicht} applied to each irreducible component of ${\rm Tanaka}(\fv^{\bullet})$, we have a decomposition into a disjoint union of $\Aut(\fv^{\bullet})$-stable nonsingular quasi-projective algebraic subvarieties equipped with submersions $$
 \begin{array}{ccccccc} {\rm Tanaka}(\fv^{\bullet}) & = & {\rm Tanaka}(\fv^{\bullet})^1 & \cup & \cdots &\cup &  {\rm Tanaka}(\fv^{\bullet})^d \\ & &  \downarrow \varsigma^1 & & & &  \downarrow \varsigma^d \\
  & & \sB^1 & & & & \sB^d \end{array} $$ for some positive integer $d$ such that for each $ 1 \leq i \leq d$, \begin{itemize} \item[(i)]  $\sB^i$ is a  nonsingular quasi-projective algebraic variety; and \item[(ii)]  the orbits of the action of $\Aut(\fv^{\bullet})$ on ${\rm Tanaka}(\fv^{\bullet})^i$ are fibers of  the submersion $ \varsigma^i:  {\rm Tanaka}(\fv^{\bullet})^i \to \sB^i$. \end{itemize}
   Note that such a decomposition is not unique: we may stratify $\sB^i$ further into a finite union of nonsingular quasi-projective subvarieties without violating the conditions (i) and (ii). Let us choose such a decomposition once and for all. Then
   the nonsingular varieties $\sB^1, \ldots, \sB^d$ depend only on the sequence of integers $(\dim \fv^{-k}, \ldots, \dim \fv^{\ell})$, not on the choice of the graded vector space $\fv^{\bullet}$.  In other words, they are determined by the isomorphism type of the graded vector space $\fv^{\bullet}$. \end{definition}

 \begin{lemma}\label{l.nB}
 In Definition \ref{d.GLA},
 each point  $a \in \sB^i, 1 \leq i \leq d,$ admits a  neighborhood $\sB \subset \sB^i$ equipped with
a  vector bundle homomorphism $$\Lambda: (\wedge^2 \fv) \times \sB \longrightarrow  \fv \times \sB$$
 satisfying the properties (i) and (ii)  in Set-up \ref{n.B} (replacing $B$ by $\sB$) such that the graded Lie algebra structure of $\fg(b)$ for each $b \in \sB$ is isomorphic to the one in $(\varsigma^{i})^{-1}(b) \subset {\rm Tanaka}(\fv^{\bullet})^i$. We can also equip $\sB$ with all the additional data (1)-(7) in Set-up \ref{n.B}. \end{lemma}

\begin{proof}
Since $\varsigma^i$ in Definition \ref{d.GLA} is a submersion, we can choose a section of $\varsigma^i$ over a neighborhood $\sB$ of $a$ to obtain the vector bundle homomorphism $\Lambda$ satisfying the properties in the first sentence of the lemma.  The data (1)- (6) in Set-up \ref{n.B} are naturally determined by $\Lambda$. Finally, by shrinking $\sB$ if necessary, we can choose the vector bundle $\sW^{n+1}$ for $ n \geq 0$ in (7) of Set-up \ref{n.B}.
\end{proof}

\begin{remark}\label{r.nB}
By using more algebraic geometry, one can choose the neighborhood $\sB$ in Lemma \ref{l.nB} as an \'etale neighborhood, namely, an unramified cover of a Zariski-open subset in $\sB^i$. This would enable us to extend the neighborhood $O_x$ in Theorem \ref{t.AP} to an \'etale neighborhood.
Since we do not have a specific application of such a strengthened version of Theorem \ref{t.AP} and  the arguments are somewhat involved,  we do not include this extension here. \end{remark}

\begin{definition}\label{d.choice}
In Lemma \ref{l.nB},  the  choice  of the data  on $\sB$ satisfying Set-up \ref{n.B} is not unique. We fix a choice of such  data  on a  neighborhood  of each point of $\sB^i, 1 \leq i \leq d,$ once and for all. We call it {\em our choice of Set-up} \ref{n.B} on $\sB^i, 1 \leq i \leq d.$
\end{definition}

\begin{definition}\label{d.regulartype}
Let $(E, F)$ be a bracket-generating pseudo-product structure  on a complex manifold $M$ of finite height $\ell \geq 0$ and let $M_* \subset M$ be the Zariski-open subset from Definition \ref{d.PP} (iii). \begin{itemize} \item[(i)]
For each $x \in M_*$,  denote by $\sV^{\bullet}_x = \oplus_{i=-k}^{\ell} \sV^i_x$ the underlying graded vector space of $\fg_x^{\bullet}$ from Definition \ref{d.PP} (ii). Let $\sV^{\bullet} = \oplus_{i=-k}^{\ell} \sV^i$ be the graded vector bundle on $M_*$ whose fiber at $x \in M_*$ is $\sV^{\bullet}_x = \oplus_{i=-k}^{\ell} \sV^i_x.$
\item[(ii)] By Definition \ref{d.GLA}, we have a natural decomposition with submersions
$$
 \begin{array}{ccccccc}  {\rm Tanaka}(\sV_x^{\bullet}) & = &  {\rm Tanaka}(\sV^{\bullet}_x)^1 & \cup & \cdots &\cup & {\rm Tanaka}(\sV^{\bullet}_x)^d \\ & &  \downarrow \beta^1_x & & & &  \downarrow \beta^d_x\\
  & & \sB^1 & & & & \sB^d \end{array} $$ for each $x \in M_*$. Putting them together over all $x \in M_*$, we obtain a natural decomposition into fiber bundles over $M_*$ equipped with submersions $$
 \begin{array}{ccccccc}  {\rm Tanaka}(\sV^{\bullet}) & = &  {\rm Tanaka}(\sV^{\bullet})^1 & \cup & \cdots &\cup &  {\rm Tanaka}(\sV^{\bullet})^d \\ & &  \downarrow \beta^1 & & & &  \downarrow \beta^d\\
  & & \sB^1 & & & & \sB^d. \end{array} $$
\item[(iii)] The family of Lie algebras $\fg_x^{\bullet}$ parametrized by $x \in M_*$ determine a holomorphic section $\lambda: M_* \to  {\rm Tanaka}(\sV^{\bullet}) $ of the surjective holomorphic map ${\rm Tanaka}(\sV^{\bullet}) \to M_*$. A point $x \in M_*$ is {\em of moduli type} $\sB^i$ if there exists a neighborhood $O\subset M_*$ of $x$ such that   $\lambda(y)  \in  {\rm Tanaka}(\sV^{\bullet}_y)^i$ for all $y \in O$. There must be an index $1 \leq i \leq d$ such that
 all general points of $ M_*$ are of moduli type $\sB^i$. In this case, we say that the pseudo-product structure $(E,F)$ is {\em of moduli type} $\sB^i.$
\item[(iv)] Assume that $(E,F)$ is of moduli type $\sB^i$. A point $x \in M_*$ of moduli type $\sB^i$ is said to be  a {\em moduli-regular} point of the pseudo-product structure, equivalently, the pseudo-product structure  is {\em moduli-regular} at $x$,  if
  $\beta^i \circ \lambda$ is a smooth morphism in a neighborhood of $x$. Denote by $M_{\flat} \subset M$ the Zariski-open subset consisting of moduli-regular points.  For $x \in M_{\flat}$,  the composition $\beta^i \circ \lambda$ defines  a submersion $\beta_x: O_x \to B_x$ from  a neighborhood $O_x \subset M_{\flat}$ of $x$ onto a locally closed submanifold $B_x \subset \sB^i$. \end{itemize}
\end{definition}

The following proposition is straightforward.

\begin{proposition}\label{p.PPtoTanaka}
In Definition \ref{d.regulartype}, assume that $(E,F)$ is of moduli type $\sB^i$. Let $x \in M_{\flat}$ be a moduli-regular point with the submersion $\beta_x: O_x \to B_x$ on a neighborhood $O_x \subset M$ of $x$.  Let $\fv^{\bullet}$ be the graded vector space underlying $\fg_x^{\bullet}$ and let $\fe \subset \fv^{-1}$ (resp. $\ff \subset \fv^{-1}$) be the subspace corresponding to $E_x$ (resp. $F_x$).
\begin{itemize}
\item[(1)] By shrinking $O_x$ if necessary, we can assume that $B := B_x$ is  a locally closed submanifold in  the neighborhood $\sB \subset \sB^i$ of $a := \beta_x(x) \in \sB^i$ in Lemma  \ref{l.nB}. Thus $B$ is equipped with all the data in Set-up \ref{n.B} inheriting those on $\sB$ from Lemma \ref{l.nB}.
\item[(2)]
    The filtration   $D_{\bullet} := D_{-k} \supset \cdots \supset D_{-1} = E+ F$
restricted to $O_x$   is a $\beta_x$-Tanaka filtration  in the sense of Definition \ref{d.betaframe}.
   \item[(3)]
Consider the fiber subbundle   $\sP \subset \I^{<0} O_x$ whose fiber $\sP_y $ at $y \in O_x$ consists of graded vector space isomorphisms $\fv^- \to {\rm symb}_y(D_{\bullet})$   that sends $\fe$ to $E_y$ and $\ff$ to $F_y$. Then $\sP \subset \I^{<0}O_x$  is a  $\beta_x$-principal subbundle  with the structure $B$-group $\mathsf{G}^0.$
 \end{itemize}  Consequently, we have a $\beta_x$-Tanaka structure $(D_{\bullet}, \sP\subset \I^{<0} O_x),$ which is canonically determined by the pseudo-product structure $(E,F)$ on $O_x$, under our choice of Set-up \ref{n.B} on $\sB^i$ in Definition \ref{d.choice}. \end{proposition}

We have the following consequence.

\begin{theorem}\label{t.PPAP}
Let $(E,F)$ be a bracket-generating pseudo-product structure on a complex manifold $M$ of finite height. Assume that it is of moduli type $\sB^i$. Then each moduli-regular point $x \in M_{\flat}$ admits a neighborhood $O_x \subset M_{\flat}$ such that we can find
\begin{itemize} \item a  complex manifold $\overline{O}_x$ with a submersion $\overline{O}_x \to O_x$; and \item an absolute parallelism on $\overline{O}_x$, \end{itemize}
which are canonically determined under our choice of Set-up \ref{n.B} on $\sB^i$ in Definition \ref{d.choice}. \end{theorem}

\begin{proof}
By Proposition \ref{p.PPtoTanaka},
each  point $x \in M_{\flat}$ admits a neighborhood $O_x$ equipped with a submersion $\beta_x: O_x \to B_x \subset \sB$ and a natural $\beta_x$-Tanaka structure.
By Corollary \ref{c.Tanaka2}, we have a canonically associated complex manifold $\overline{O_x}$ with a submersion $\overline{O_x} \to O_x$ and a natural absolute parallelism on $\overline{O_x}$. \end{proof}

Theorem \ref{t.PPAP} implies Theorem \ref{t.AP}. To see this, note that
the family of submanifolds in Set-up \ref{set.sK} gives rise to a pseudo-product structure in the following way.

\begin{definition}\label{d.sK}
Let us work under Set-up \ref{set.sK}.
  For a general point $y \in \sU$, the morphism $\mu$ sends the fiber $\rho^{-1}(\rho(y))$ isomorphically to the $m$-dimensional compact submanifold $A_z := \mu(\rho^{-1}(\rho(y)) \subset X$ corresponding to the point $z:= \rho(y)$ of the Douady space of $X$. This implies that $${\rm Ker}({\rm d}_y \mu)  \cap {\rm Ker}({\rm d}_y \rho) \ = \ 0.$$ Thus we have a nonempty Zariski-open subset  $\sU^o \subset \sU'$ such that \begin{itemize} \item[(U1)] $\sU^o$ and $\sK^o:= \rho(\sU^o) \subset \sK$ are complex manifolds; \item[(U2)] both $\rho|_{\sU^o}$ and $\mu|_{\sU^o}$ are smooth morphisms;
\item[(U3)] the vector bundles $E:= {\rm Ker}({\rm d} \mu|_{\sU^o})$ of rank $r$  and $F:= {\rm Ker}({\rm d} \rho|_{\sU^o})$ of rank $m$ define a pseudo-product structure on $\sU^o$; \item[(U4)] any point of $\sU^o$ is a regular point of the distribution $D = E + F$. \end{itemize}
    We say that $(E,F)$ on $\sU^o$ is the {\em pseudo-product structure arising from the family } $\sK$. \end{definition}

\begin{lemma}\label{l.BG}
In Definition \ref{d.sK}, the distribution $D = E+F$ is bracket-generating (in the sense of Definition \ref{d.distribution}) if and only if the family $\sK$ is bracket-generating (in the sense of Definition \ref{d.bracketgen}). \end{lemma}

\begin{proof}
Suppose that $D$ is not bracket-generating. Then a general point $y \in \sU^o$ has a neighborhood $U \subset \sU^o$ with a nonconstant submersion $\xi: U \to R$ to a positive-dimensional complex manifold $R$  such that $ {\rm d}_u \xi ( D_u ) =0 $ for all $u \in U$.  Since fibers of $\mu|_U$ and $\rho|_U$ are contained in the fibers of $\xi,$ we can choose neighborhoods $\rho(y) \in W \subset \sK$ and $\mu(y) \in O \subset X$ equipped with a submersion $\zeta: O \to R$ such that for any $z \in W$, the intersection $\mu(\rho^{-1}(z)) \cap O$ is contained in a fiber of $\zeta$. This implies that  $\sK$ is not bracket-generating.  The proof of the converse is similar.  \end{proof}

\begin{proposition}\label{p.immersive}
Under the assumptions of Theorem \ref{t.FPC} (or equivalently, Theorem \ref{t.AP}),
the  pseudo-product structure $(E,F)$ in Definition \ref{d.sK} is bracket-generating and of finite height. \end{proposition}

\begin{proof} Lemma \ref{l.BG}  implies that $(E,F)$ is bracket-generating.
 The maps $\mu^{\sharp}$ and $\rho^{\sharp}$ in Definition \ref{d.sharp} are restrictions of those in Definition \ref{d.tau}. Thus the assumption in Theorem \ref{t.AP} implies
by Corollary \ref{c.Cauchy} that $(E,F)$ is of finite height.  \end{proof}

 By Proposition \ref{p.immersive}, Theorem \ref{t.AP} follows from Theorem \ref{t.PPAP}.

\section{Convergence of formal equivalences}\label{s.converge}

As mentioned in Section \ref{s.introduction}, our strategy to prove Theorem \ref{t.FPC} is to deduce it from Theorem \ref{t.AP}. It is more natural to  explain this deduction  from  a general perspective as follows.

\begin{theorem}\label{t.converge} Let  $M$ (resp. $\w{M}$)   be a complex manifold
with a submersion $\beta: M \to B$ (resp. $\w{\beta}: \w{M} \to B$) and let $(\sD_{\bullet}, \sP \subset \I^{<0}M)$ (resp. $(\w{\sD}_{\bullet}, \w{\sP} \subset \I^{<0}\w{M})$ be a $\beta$-Tanaka structure on $M$ (resp. $\w{\beta}$-Tanaka structure on $\w{M}$. Let $\varphi: (x/M)_{\infty} \to (\w{x}/\w{M})_{\infty}$ be a formal isomorphism between formal neighborhoods of points $x \in M$ and $\w{x} \in \w{M}$ such that \begin{itemize} \item  $\w{\beta} \circ \varphi = \beta|_{(x/M)_{\infty}}$ and \item $\varphi$ sends the restriction of $(\sD_{\bullet}, \sP \subset \I^{<0}M)$ to $(x/M)_{\infty}$ to the restriction of $(\w{\sD}_{\bullet}, \w{\sP} \subset \I^{<0}\w{M})$ to $(\w{x}/\w{M})_{\infty}$. \end{itemize} Then $\varphi$ is convergent. \end{theorem}

The proof of Theorem \ref{t.converge} uses the following result, which is  a direct consequence of Theorem 7.2 in Chapter VI of \cite{KN} (see also the proof of Theorem 5.1 in \cite{Hw24}). This is

\begin{theorem}\label{t.KN}
Let $\nabla$ (resp. $ \widetilde{\nabla}$)  be a holomorphic affine connection (for example, an absolute parallelism) on a complex manifold $R$ (resp. $\w{R}$).  Let  $\xi: (y/R)_{\infty} \to (\widetilde{y}/\widetilde{R})_{\infty}$ be a formal isomorphism for some points $y \in R$ and $\widetilde{y} \in \widetilde{R}$  such that $\xi$ sends the restriction  $\nabla|_{(y/R)_{\infty}}$ to the restriction $\widetilde{\nabla}|_{(\widetilde{y}/\widetilde{R})_{\infty}}.$ Then $\xi$ is convergent. \end{theorem}

We also need the following lemma.

\begin{lemma}\label{l.Stoll}
Let $\psi: R \to M$ (resp. $\w{\psi}: \w{R} \to \w{M}$) be a submersion between complex manifolds and let $y \in R, \w{y} \in \w{R}, x \in M, \w{x} \in \w{M}$ be points satisfying $x= \psi(y)$ and $\w{x} = \w{\psi}(\w{y})$.   Let $$\xi: (y/R)_{\infty} \to (\w{y}/\w{R})_{\infty} \mbox{ and } \varphi: (x/M)_{\infty} \to (\w{x}/\w{M})_{\infty}$$ be  formal isomorphisms satisfying $\w{\psi} \circ \xi = \varphi \circ \psi$.
If $\xi$ is convergent, then $\varphi$ is convergent. \end{lemma}

\begin{proof}
A classical result of F. Hartogs (Theorem of Hartogs in \cite{St}, also cited in the proof of Lemma 3.5 in \cite{Hir}) says that  $\varphi$ converges if its restriction to any smooth curve germ $x \in C \subset M$ converges. Since $\psi$ is a submersion,  for any smooth curve germ $x \in C$,  we can find a smooth curve germ $y \in C^{\sharp} \subset R$ such that $\psi|_{C^{\sharp}}: C^{\sharp} \to C$ is a biholomorphism, admitting an inverse $(\psi |_{(y/C^{\sharp})_{\infty}})^{-1}: (x/C)_{\infty} \to (y/C^{\sharp})_{\infty}$.   Then $$\varphi|_{(x/C)_{\infty}} = \w{\psi} \circ \xi \circ (\psi |_{(y/C^{\sharp})_{\infty}})^{-1}.$$ Thus the convergence of $\xi$ implies the convergence of $\varphi$ at $x$. \end{proof}

\begin{proof}[Proof of Theorem \ref{t.converge}]
By Corollary \ref{c.Tanaka2}, we have a natural fiber bundle $\psi: R:= \bar{P}^{\ell} \to M$ (resp. $\w{\psi}: \w{R}:= \w{\bar{P}^{\ell}} \to \w{M}$) equipped with a natural absolute parallelism $\nabla$ (resp. $\w{\nabla}$). We claim that the formal equivalence $\varphi$ can be lifted to a formal equivalence $\xi: (y/R)_{\infty} \to (\widetilde{y}/\widetilde{R})_{\infty}$ for some point $y \in R$ over $x \in M$ (resp.  $\w{y} \in \w{R}$ over $\w{x} \in \w{M}$) satisfying $\w{\psi} \circ \xi = \varphi \circ \psi$ such that  $\xi$ sends $\nabla|_{(y/R)_{\infty}}$ to $\w{\nabla}|_{(\w{y}/\w{R})_{\infty}}.$ The claim combined with Theorem \ref{t.KN}
implies that $\xi$ converges. Thus $\varphi$ converges by Lemma \ref{l.Stoll}.

To check the claim, it is enough to note that $\varphi$ can be lifted successively in the inductive constructions in Theorem \ref{t.Tanaka1} to a formal equivalence
$$\varphi^{(n)}: (x^{(n)}/\bar{P}^{(n)})_{\infty} \to (\w{x}^{(n)}/\w{\bar{P}^{(n)}})_{\infty}$$
 that sends the $\beta^{(n)}$-Tanaka parallelism $I^{(n)}|_{(x^{(n)}/\bar{P}^{(n)})_{\infty}}$ to the corresponding $\widetilde{\beta}^{(n)}$-Tanaka parallelism $\widetilde{I^{(n)}}|_{(\w{x}^{(n)}/\w{\bar{P}^{(n)}})_{\infty}}$ for each $n \geq 1$, where  $x^{(n)}$ is some point over $x$ and $\w{x}^{(n)}$ is some point over $\w{x}$. In fact, such $\varphi^{(n)}$ can be lifted  to the next step $\varphi^{(n+1)}$ from the inductive construction in (B3) of Theorem \ref{t.Tanaka1}.
\end{proof}

To apply Theorem \ref{t.converge} to Theorem \ref{t.FPC}, we first consider the more general setting of pseudo-product structures as follows.

\begin{definition}\label{d.equivalence}
Let $(E, F) $ (resp. $(\widetilde{E}, \widetilde{F})$) be a pseudo-product structure on a complex manifold $M$ (resp. $\widetilde{M}$).
For $x \in M$ and $\widetilde{x} \in \widetilde{M},$
a formal isomorphism $\varphi: (x/M)_{\infty} \to (\widetilde{x}/\widetilde{M})_{\infty}$  is a {\em formal equivalence of the pseudo-product structures},  if ${\rm d} \varphi (E|_{(x/M)_{\infty}}) = \widetilde{E}|_{(\widetilde{x}/\widetilde{M})_{\infty}}$ and ${\rm d} \varphi (F|_{(x/M)_{\infty} }) = \widetilde{F}|_{(\widetilde{x}/\widetilde{M})_{\infty}}.$
\end{definition}

\begin{proposition}\label{p.constrank}
 Let $\varphi: (x/M)_{\infty} \to (\widetilde{x}/\widetilde{M})_{\infty}$ be a formal equivalence of the pseudo-product structures in Definition \ref{d.equivalence}. Suppose that  the distribution $D= E+F$ is bracket-generating of depth $k$; the height of $(E, F)$ is $\ell < + \infty$; and $x \in M_*$ in Definition \ref{d.PP} (iv).
 Then the following holds.  \begin{itemize}
 \item[(i)] The distribution $\w{D} = \widetilde{E} + \widetilde{F}$ is bracket-generating of depth  $k$; the height of $(\w{E}, \w{F})$ is $\ell$; and the point $\w{x}$ is contained in $ \widetilde{M}_{*} \subset \widetilde{M}.$
         \item[(ii)]  The prolongation of the symbol algebra of $(\w{E}, \w{F})$ at $\w{x}$ is isomorphic to that of $(E, F)$ at $x$  as graded Lie algebras. \end{itemize} \end{proposition}

To prove Proposition \ref{p.constrank}, we use the following easy lemma.

\begin{lemma}\label{l.constrank}
Let $h: \sV \to \sW$ be a vector bundle homomorphism between two vector bundles $\sV$ and $\sW$ on a complex manifold $R$. For a point $x \in R$,
whether the homomorphism $h$ has constant rank in a neighborhood of $x \in R $ is determined by the restriction of $h$ to the formal neighborhood $(x/R)_{\infty}$.
\end{lemma}

\begin{proof} Observe  that the following two statements are obviously equivalent. \begin{itemize} \item  The homomorphism $h$ has constant rank in   a neighborhood of $x.$
\item  For any $v \in {\rm Ker}(h_x)$ and any positive integer $N >0$,  there is a holomorphic section $v^N$ of $\sV$ in a neighborhood of $x$ such that $ v = v^N_x$ and the local holomorphic section $h(v^N)$ of $\sW$ vanishes to order at least $N$ at $x$.
     \end{itemize}
     The latter statement is certainly determined by the restriction of $h$ to the formal neighborhood. \end{proof}

             \begin{proof}[Proof of Proposition \ref{p.constrank}]  Each of the properties of $D= E+F$ and $\w{D} = \w{E} + \w{F}$ involved in (i)   can be formulated as the constant rank condition for a suitable vector bundle homomorphism. Thus it is unchanged under the formal equivalence by Lemma \ref{l.constrank}. This proves (i).

             The Lie algebra structures of  the prolongation of the symbol algebra at $x$ (resp. at $\w{x}$) of $(E,F)$ (resp. $(\w{E}, \w{F})$) is determined by finite jets at $x$ (resp. at $\w{x}$) of vector fields arising from local sections of the distribution $D$  (resp. $\w{D}$). Thus it is unchanged under formal equivalences, proving (ii).  \end{proof}

\begin{proposition}\label{p.moduliregular}   Let $(E,F)$ (resp. $(\w{E}, \w{F})$) be a bracket-generating pseudo-product structure of finite height (for example, a bracket-generating and Levi-nondegenerate pseudo-product structure by Theorem \ref{t.TanakaPP}) on a complex manifold $M$ (resp. $\w{M}$).  Let $\varphi: (x/M)_{\infty} \to (\widetilde{x}/\widetilde{M})_{\infty}$ be a formal equivalence of the pseudo-product structures at points $x \in M$ and $\w{x} \in \w{M}$. Assume that  $(E,F)$ is of moduli type $\sB^i$ and  moduli-regular  at $x$. \begin{itemize} \item[(i)] Then  $(\widetilde{E}, \widetilde{F})$ is of moduli type $\sB^i$ and  moduli-regular  at $\widetilde{x}$. \item[(ii)] If $\beta_x: O_x \to B_x \subset \sB^i$ (resp. $\beta_{\w{x}}: O_{\w{x}} \to B_{\w{x}} \subset \sB^i$) is the submersion  in a neighborhood $x \in O_x \subset M$ (resp. $\w{x} \in O_{\w{x}} \subset \w{M}$) given in Definition \ref{d.regulartype} (iv), then $\beta_x(x) = \beta_{\w{x}}(\w{x})$ and  the germ of $B_x \subset \sB^i$ at $\beta(x)$ coincides with the germ of $B_{\w{x}} \subset \sB^i$ at $\beta_{\w{x}}(\w{x})$. \end{itemize}  \end{proposition}

To prove Proposition \ref{p.moduliregular}, we use the following lemma. Its nature is similar to the argument in the proof of Lemma \ref{l.constrank}. It may look like an unnecessarily complicated formulation of a simple fact, but this is one way to translate the formal condition in Proposition \ref{p.moduliregular} into a holomorphic condition.

\begin{lemma}\label{l.regulartype}
Let $(E, F)$ be a bracket-generating pseudo-product structure  of finite height $\ell$ on a complex manifold $M$.  From Definition \ref{d.regulartype},  we have a section $\lambda: M_* \to {\rm Tanaka}(\sV^{\bullet})$ and a   locally closed complex submanifold $$\lambda(M_*) \ \subset \ {\rm Tanaka}(\sV^{\bullet}) \ \subset \ \Hom(\wedge^2 \sV, \sV)$$ in the vector bundle $\Hom(\wedge^2 \sV, \sV)$ over $M_*$.
Then a point $x \in M_*$ is of moduli type $\sB^i$  if and only if \begin{itemize} \item[(1)] $\lambda(x) \in  {\rm Tanaka}(\sV^{\bullet}_x)^i$; and \item[(2)] for any positive integer $N >0$, there exist  a neighborhood $x \in O^N \subset M_*$ and a holomorphic section $\lambda^N: O^N \to  {\rm Tanaka}(\sV^{\bullet})|_{O^N}$ of the surjective holomorphic map ${\rm Tanaka}(\sV^{\bullet}) \to M_*$ restricted to $O^N$,  such that  $\lambda^N(x) = \lambda(x)$ and the submanifold $\lambda^N(O^N) $ of $\Hom(\wedge^2 \sV, \sV)$  has order of contact bigger than $N$ at $\lambda(x)$ with each of the two submanifolds $\lambda(O^N)$ and ${\rm Tanaka}(\sV^{\bullet})^i$ in $\Hom(\wedge^2 \sV, \sV)$. \end{itemize}
Furthermore, a point  $x \in M_*$  of moduli type $\sL^i$  is a moduli-regular point of the pseudo-product structure if and only if the submanifold $\lambda^N(O^N) $ in (2) satisfies
\begin{equation}\label{e.ON} \dim (T_{\lambda(x)} (\lambda^N(O^N)) \cap {\rm Ker}( {\rm d}_{\lambda(x)} \beta^i)) = \dim  M - \dim (\beta^i \circ \lambda (O^N)). \end{equation}
\end{lemma}

\begin{proof}
Since ${\rm Tanaka}^i(\sV^{\bullet})$ is a locally closed submanifold in $\Hom(\wedge^2 \sV, \sV)$, it is clear that $x \in M_*$ is of moduli type $\sL^i$ if and only if (1) and (2) hold.
When $x \in M_*$ is of moduli type $\sL^i$ and (\ref{e.ON}) holds, then the morphism $\beta^i \circ \lambda$ restricted to a neighborhood of $x$ is a submersion over its image. This implies that $x$ is a moduli-regular point. The converse is straightforward. \end{proof}

\begin{proof}[Proof of Proposition \ref{p.moduliregular}]
The finite-order jets of the formal isomorphism $\varphi$ sends a holomorphic map $\lambda^N$ defined in a neighborhood of $x \in M$ in Lemma \ref{l.regulartype} to a holomorphic map $\w{\lambda}^N$ defined in a neighborhood of $\w{x} \in \w{M}_*$ with corresponding properties.  This implies (i) by Lemma \ref{l.regulartype}.

We have $\beta_x(x) = \w{\beta}_{\w{x}}(\w{x})$ from Proposition \ref{p.constrank} (ii).
By Lemma \ref{l.regulartype},  the two submanifolds $B_x$ and $B_{\w{x}}$ in $\sB^i$ have arbitrarily high order contact at the point $\beta_x(x) = \beta_{\w{x}}(\w{x})$. Thus they must have identical germs, proving (ii).
\end{proof}

\begin{theorem}\label{t.formalEF}
 Let $(E,F)$ (resp. $(\w{E}, \w{F})$) be a bracket-generating pseudo-product structure of finite height  on a complex manifold $M$ (resp. $\w{M}$).  Let $\varphi: (x/M)_{\infty} \to (\widetilde{x}/\widetilde{M})_{\infty}$ be a formal equivalence of the pseudo-product structures at points $x \in M$ and $\w{x} \in \w{M}$. Assume that the pseudo-product structure $(E,F)$ is moduli-regular at $x \in M$. Then $\varphi$ is convergent. \end{theorem}

 \begin{proof}
 By Proposition \ref{p.moduliregular},
 the pseudo-product structure $(\w{E}, \w{F})$ is moduli-regular at $\w{x}$ and we can identify $B_x$ with $B_{\w{x}}$ by choosing $O_x$ and $O_{\w{x}}$ suitably.
 Replace $M$ by $ O_x$ (resp. $\w{M}$ by $ O_{\w{x}}$) and set $\beta := \beta_x$ (resp. $\w{\beta} := \beta_{\w{x}}$.) By Proposition \ref{p.PPtoTanaka}, the pseudo-product structure $(E,F)$ on $M$ (resp. $(\w{E}, \w{F})$ on $\w{M}$) determines  a $\beta$-Tanaka structure $(D_{\bullet}, \sP \subset \I^{<0}M)$ on $M$ (resp. a $\w{\beta}$-Tanaka structure  $(\w{D}_{\bullet}, \w{\sP}\subset \I^{<0} \w{M})$ on $\w{M}$).
   Then $\varphi$ is a formal isomorphism between the $\beta$-Tanaka structure and the $\w{\beta}$-Tanaka structure. Thus it is convergent by Theorem \ref{t.converge}.  \end{proof}

To prove Theorem \ref{t.FPC}, we need the following lemma, a consequence of the maximum principle.

   \begin{lemma}\label{l.maxprinciple}
Let $\varphi: (A/X)_{\infty} \to (\w{A}/\w{X})_{\infty}$ be a formal isomorphism between two compact complex submanifolds $A \subset X$ and $\widetilde{A} \subset \widetilde{X}.$
Suppose there exists a closed analytic subset $S \subset A, S \neq A,$ such that  the formal isomorphism $\varphi$ is convergent at any point of $A \setminus S$. Then $\varphi$ is convergent at every point of $A$. \end{lemma}

\begin{proof}
Let $m$ be the dimension of $A$ and let $n$ be the dimension of $X$.
For a point $x \in S$, we can choose   a neighborhood $U$ of $ x \in X$ biholomorphic to $(U \cap A) \times \Delta^{n-m}$ with coordinates $(z^1, \ldots z^m)$ on $U \cap A$ and $(w^1, \ldots, w^{n-m})$ on the polydisc $\Delta^{n-m}$ such that under a suitable  coordinate system in a neighborhood of $\varphi(x) \in \w{X}$, the formal isomorphism $\varphi$ is represented by an $n$-tuple of  formal power series  in $n-m$ variables $w= (w^1, \ldots, w^{n-m})$ (for example, see page 107 of \cite{CG}) $$ g_{i}(z,w) = \sum_{|\nu| = 1}^{\infty} g_{i, \nu}(z) \cdot  w^{\nu} \mbox{ for } 1 \leq i \leq n, $$ where the coefficient $g_{i,\nu}(z)$ is a holomorphic function in $z=(z^1, \ldots, z^m)$ and $\nu = (\nu^1, \ldots, \nu^{n-m})$ is the multi-index of nonnegative integers with $$|\nu| := \nu^1 + \cdots + \nu^{n-m} \mbox{ and } w^{\nu} := (w^1)^{\nu^1} \cdots (w^{n-m})^{\nu^{n-m}}.$$   This formal power series converges at a point $ (z, w=0) \in A\cap U$, if and only if the formal isomorphism $\varphi$ converges at that point.

Writing $$\Delta_{r}:= \{ t \in \C \mid |t| < r\} \mbox{ and } \p \Delta_r:= \{ t \in \C \mid |t|=r\},$$  choose an arc $\gamma: \Delta_{2} \to A\cap U$ such that  $\gamma(0) =x \in S$ and $\gamma (\p \Delta_1) \cap S = \emptyset$.
Then the formal power series $g_{i}(t, w) := g_i(\gamma(t), w)$  in $w$ converges when $ t \in \p \Delta_1$. Thus there is some $\epsilon >0$ such that for each fixed $w = (w^1, \ldots, w^{n-m})$ with $ |w^1|, \ldots, |w^{n-m}| < \epsilon$,  the infinite sequence $$\{  | \sum_{|\nu|=1}^{s} g_{i, \nu}(\gamma(t))  \cdot w^{\nu}|^2 \in \R \mid s \in \N \} $$  is bounded on $\gamma(\p \Delta_1)$.   By the maximum principle, they are bounded on the whole $\Delta_1$.  It follows that the formal power series converges at $x$.
\end{proof}

We are ready to prove Theorem \ref{t.FPC}.

   \begin{proof}[Proof of Theorem \ref{t.FPC}]
   By Proposition \ref{p.immersive}, we have a pseudo-product structure $(E,F)$ on $\sU^{\circ}$ satisfying the conditions in Definition \ref{d.regulartype}.   Let $y \in \sU_{\flat}$ be a moduli-regular point for $(E,F)$. We claim that the submanifold $A:= A_z \subset X$ corresponding to  $z = \rho(y)\in \mathcal K_{\flat}:=\rho(\mathcal U_{\flat})$ satisfies the formal principle with convergence.

   Let $\w{A} \subset \w{X}$ be a compact complex submanifold with a formal isomorphism $\phi: (A/X)_{\infty} \to (\w{A}/\w{X})_{\infty}$. Let $\w{\sK} \subset {\rm Douady}(\w{X})$ be the component containing the deformations of $\w{A}$ in $\w{X}$
    with the universal family $$\w{\sK} \ \stackrel{\w{\rho}}{\longleftarrow} \ \w{\sU} \ \stackrel{\w{\mu}}{\longrightarrow} \ \w{X}.$$ Let $\w{z} \in \w{\sK}$ be the point corresponding to $\w{A}$. Write $A^{\sharp} \subset \sU$ for the fiber $\rho^{-1}(z)$ and  $\w{A}^{\sharp} \subset \w{\sU}$ for the fiber $\w{\rho}^{-1}(\w{z})$.

    By the functorial property of Douady space (see p. 509 of \cite{Hir}, also Lemma 3.5 of \cite{Hw19}), the formal isomorphism $\phi$ induces a formal isomorphism $\phi^{\sharp}: (A^{\sharp}/\sU)_{\infty} \to (\w{A}^{\sharp}/\w{\sU})_{\infty}$ such that for  a general point $y \in A^{\sharp}$ and $\w{y} = \phi(y)$, the restriction
    $$\varphi:= \phi^{\sharp}|_{(y/\sU)_{\infty}} : (y/\sU)_{\infty} \to (\w{y}/\w{\sU})_{\infty}$$ is a formal equivalence of the pseudo-product structures. Since $y$ is a moduli-regular point, Theorem \ref{t.formalEF} implies that $\varphi$ converges. Hence $\phi$ converges at $x = \mu(y)$ by Lemma \ref{l.Stoll}.
    Since $x$ is a general point of $A$, we see that $\phi$ converges at every point of $A$ by Lemma \ref{l.maxprinciple} \end{proof}

\section{Examples of the family $\sK$ in Theorem \ref{t.FPC}}\label{s.examples}

Let us examine  the conditions for the family $\sK$ in Theorems \ref{t.FPC} and \ref{t.AP} in terms of the geometry of the corresponding submanifolds in $X$.  The generic immersiveness of the map $\mu^{\sharp}$ has a clear geometric meaning: general members have distinct tangent spaces at general points of the submanifolds.
The generic immersiveness of  $\rho^{\sharp}$ has the following interpretation, which is a general version of the condition (iv) in Theorem \ref{t.Hw}.

\begin{proposition}\label{p.rhosharp}
In Definition \ref{d.tau}, the map $\rho^{\sharp}$ is generically immersive if and only if for a general member $A \subset X$ of $\sK$, a general point $x \in A$ and any point $x' \neq x$ in a neighborhood of $x$ in $A$,
$$H^0(A, N_{A/X} \otimes \mathbf{m}_x) \neq H^0(A, N_{A/X} \otimes \mathbf{m}_{x'}).$$
\end{proposition}

\begin{proof}
Let $z \in \sK$ be the point corresponding to $A \subset X.$
The basic deformation theory of compact submanifolds gives a natural identification $H^0(A, N_{A/X}) = T_z \sK$  such that the subspace $$H^0(A, N_{A/X} \otimes \mathbf{m}_x) \subset H^0(A, N_{A/X}) = T_z \sK $$ consisting of sections vanishing at $x \in A$ corresponds to ${\rm d} \rho ({\rm Ker}({\rm d}_y \mu) )\subset T_z \sK$ for the point $y \in \sU$ with $z = \rho(y)$ and $x = \mu(y)$. Thus the condition $H^0(A, N_{A/X} \otimes \mathbf{m}_x) \neq H^0(A, N_{A/X} \otimes \mathbf{m}_{x'})$ is equivalent to saying that $\rho^{\sharp}$ is injective in a neighborhood of $y$. This is equivalent to the generic immersiveness of $\rho^{\sharp}$.
\end{proof}

 For the bracket-generating condition, it is convenient to introduce the following.

   \begin{definition}\label{d.sKdistribution}
In Set-up \ref{set.sK}, assume that the fibers of $\mu$ are irreducible.
For a general point $y \in \sU$, consider a neighborhood $U_y \subset \sU$ of $y$. For $u \in U_y$ and  $x = \mu(u)$, let ${\rm Dist}^{\sK}_x \subset T_x X$ be the vector subspace spanned by the family of $m$-dimensional subspaces $\{{\rm d} \mu ({\rm Ker} ({\rm d}_u \rho)) \mid u \in U_y\}$. By the assumption that $\mu$ has irreducible fibers, this does not depend on the choice of the neighborhood $U_y$ and  determines a distribution ${\rm Dist}^{\sK} \subset T X^o$ on a nonempty open subset $X^o \subset X$, called the {\em distribution spanned by } $\sK$. \end{definition}

\begin{proposition}\label{p.sKdistribution}
The distribution ${\rm Dist}^{\sK}$ spanned by $\sK$ in Definition \ref{d.sKdistribution} is a bracket-generating distribution on $X^o$ if and only if $\sK$ is  bracket-generating. \end{proposition}

\begin{proof}
By Lemma \ref{l.BG}, we may prove that  the distribution $D=E+F$ in Definition \ref{d.sK} is  bracket-generating if and only if ${\rm Dist}^{\sK}$ is  bracket-generating.  Suppose $D$ is not bracket-generating. Then there is a nonempty Zariski-open subset $\sU^* \subset \sU^o$ with an integrable distribution  $\sL \subset T\sU^*$ such that  $E|_{\sU^*} \subset \sL$ and $F|_{\sU^*} \subset \sL$. Since general fibers of $\mu$ are irreducible, this implies that $\sL$ descends to a distribution ${\rm d} \mu (\sL) \subset T X'$ on some nonempty open subset $X' \subset X$.
Since $F|_{\sU^*} \subset \sL$, we see that ${\rm Dist}^{\sK}|_{X'} \subset {\rm d} \mu(\sL)$, proving that ${\rm Dist}^{\sK}$ is not bracket-generating. The proof of the converse is similar.  \end{proof}

The following is well-known  (Theorem 1.2 of \cite{Hw01}, Section 1 of \cite{HwM}). We give a proof for the reader's convenience.

\begin{lemma}\label{l.unbendable}
  If a general member of $\sK$ in Definition \ref{d.tau} is an unbendable rational curve, namely, a rational curve with normal bundle isomorphic to $\sO (1)^{\oplus r} \oplus \sO^{\oplus (\dim X - r -1)}$, then both $\rho^{\sharp}$ and $\mu^{\sharp}$ are generically immersive. \end{lemma}

   \begin{proof} Note that $r >0$ by our assumption in Definition \ref{d.tau}. Proposition \ref{p.rhosharp} implies that $\rho^{\sharp}$ is generically immersive.  Suppose that $\mu^{\sharp}$ is not generically immersive, then there must be a nontrivial deformation of an unbendable rational curves   $$\{ C_t \subset X \mid t \in \Delta\}$$ satisfying $T_x C_t = T_x C_0$ for all $t \in \Delta$ at some point $x \in C_0$. Then its infinitesimal deformation $\frac{{\rm d}}{{\rm d} t} C_t \in H^0(C_t, N_{C_t/X})$ must vanish at $x$ to second order. But a holomorphic section of the bundle $\sO(1)^{\oplus r} \oplus \sO^{\oplus (\dim X -r -1)}$ vanishing to second order at a point must vanish identically. A contradiction.
   \end{proof}

   A corollary of Proposition \ref{p.sKdistribution} and  Lemma \ref{l.unbendable} is the following.

 \begin{corollary}\label{c.Pic1}
    Let $X$ be a Fano manifold of Picard number 1 and let $\sK$ be as in Set-up \ref{set.sK} such that its members are unbendable rational curves with nontrivial normal bundles and the morphism $\mu$ has irreducible fibers. Then $\sK$ is bracket-generating and both $\rho^{\sharp}$ and $\mu^{\sharp}$ are generically immersive. \end{corollary}

    \begin{proof}
    It is sufficient to prove that ${\rm Dist}^{\sK}$  in Definition \ref{d.sKdistribution} is bracket-generating.  But this is proved in Proposition 6.8 of \cite{Hw12} (where it was stated for a family of minimal rational curves, but the proof works under the assumption that  general members of $\sK$ are unbendable).  \end{proof}

Corollary  \ref{c.Pic1} provides many examples (see Section 1.4 in \cite{Hw01}) of unbendable rational curves ,  in particular,  minimal rational curves, satisfying the conditions of Theorem \ref{t.FPC}.
  Let us just list  a few well-known examples.

\begin{example}\label{ex.homo}
Let $X$ be a rational homogeneous space $G/P$ of Picard number 1. Then a general line on $X$  (under the minimal projective embedding) satisfies the assumption in Corollary \ref{c.Pic1}. Moreover, if $P$ is associated to a long simple root, then all lines on $G/P$ are equivalent under the action of $G$ (e.g. see Proposition 1 of  \cite{HM02}), hence  any line on $X$ satisfies the formal principle with convergence.\end{example}

\begin{example}\label{ex.hypersurface}
Let  $X \subset \BP^{n+1}$ be a smooth hypersurface of degree less than $n  \geq 4$. Then the family of lines on $X$ satisfies the assumption on Corollary \ref{c.Pic1}. If $X$ has degree $n $, a general fiber of $\mu$ is not irreducible and Corollary \ref{c.Pic1} cannot be applied. In this case, a general line has trivial normal bundle and admits a holomorphic tubular neighborhood, namely, a neighborhood biholomorphic to $\BP^1 \times \Delta^{n-1}$. The formal principle with convergence does not hold for the same reason as in Example \ref{ex.product}.
\end{example}

When $n = 4$, Example \ref{ex.hypersurface}  (Corollary \ref{c.hypersurface}) is  Corollary 1.14 of \cite{Hw24}, which is a direct consequence of  Theorem 1.13, of \cite{Hw24}. The latter is a special case of Theorem \ref{t.FPC},  where $\sK$ is a family of rational curves of Goursat type.

Among examples with $X$ of higher Picard numbers, the simplest one is when $X$ is the blow-up of $\BP^n$ along a submanifold $S$ spanning a hyperplane $\BP^{n-1} \subset \BP^n$  and $\sK$ is the family of proper transformations of lines on $\BP^n$ intersecting the submanifold $S$ (see Lemma 5.6 in \cite{Hw21}). A more involved example is  the following, which is a direct consequence of Theorem 1.1 of \cite{BF} and our Proposition \ref{p.sKdistribution}.

\begin{example}\label{ex.symmetric}
Let $X$ be a wonderful compactification of the adjoint group of a simple Lie algebra. Then the family of minimal rational curves on $X$ satisfies the conditions in Theorem \ref{t.FPC}. \end{example}

There are also many examples of    submanifolds of dimension bigger than one satisfying the conditions in Theorem \ref{t.FPC}. Generalizing Example \ref{ex.hypersurface}, one can check (using results from \cite{DM}) that general linear subspaces of low dimension in hypersurfaces of sufficiently low degree satisfy the conditions in Theorem \ref{t.FPC}. Another class of examples is the following.

\begin{definition} \label{d.Qcycle} Let $X$ be a rational homogeneous space $G/P$  and let $Q$ be a parabolic subgroup $Q$ of $G$ such that $P \cap Q$ is a parabolic subgroup. Then the $Q$-orbit $A$ of the base point of $G/P$ is again a rational homogeneous space $Q/Q\cap P \simeq L/L \cap P$, where $L$ is the semisimple part of the reductive part of $Q$. Let $\mathcal K$ be the connected component of ${\rm Douady}(X)$ containing the point $[A]$ corresponding to $A$.  A member of $\sK$ is called a {\it $Q$-cycle} on $G/P$.
\end{definition}

For example, lines on $G/P$ in Example \ref{ex.homo}
are  one-dimensional $Q$-cycles on   $G/P$.
Using Proposition 2.3 and Proposition 2.6 of \cite{HoN}, one can check the following.
Since the proof is somewhat lengthy involving the theory of semisimple Lie algebras, we just state the results.

 \begin{example}\label{ex.Q}
 In Definition \ref{d.Qcycle}, assume that $P$ is associated to a set of long simple roots and there is no proper parabolic subgroup containing both $P$ and $Q$. Then
 the family $\sK$ of $Q$-cycles on $X$ is $G$-homogenous and  satisfies the conditions in Theorem \ref{t.FPC}. Hence any $Q$-cycle in this case satisfies the formal principle with convergence.
 \end{example}

\bigskip
Jaehyun Hong (jhhong00@ibs.re.kr)

Jun-Muk Hwang (jmhwang@ibs.re.kr)

\smallskip

Center for Complex Geometry,
Institute for Basic Science (IBS),
Daejeon 34126, Republic of Korea

\bigskip

\end{document}